\documentclass[ECP, preprint]{ejpecp} 

\usepackage{amsmath}
\usepackage{amstext}
\usepackage{amsfonts}
\usepackage{amssymb}
\usepackage{graphicx}
\usepackage{bbm}

\usepackage{tikz}
\usepackage{pgfplots}
\usepackage[T2A,T1]{fontenc}

\SHORTTITLE{Concentration of product of random vectors}

\TITLE{Concentration of measure \\ and generalized product of random vectors \\ with an application to Hanson-Wright-like inequalities.}

\AUTHORS{%
  Cosme Louart\footnote{ GIPSA-lab. \EMAIL{cosmelouart@gmail.com}}
  \and 
  Romain Couillet\footnote{LIG-lab,
GIPSA-lab. \EMAIL{romain.couillet@gipsa-lab.grenoble-inp.fr}}}


\AMSSUBJ{60-08, 60B20, 62J07} 

\SUBMITTED{September 19, 2021} 

\ARXIVID{2102.08020} 

\VOLUME{0}
\YEAR{2020}
\PAPERNUM{0}
\DOI{10.1214/YY-TN}

\ABSTRACT{Starting from concentration of measure hypotheses on $m$ random vectors $Z_1,\ldots, Z_m$, this article provides an expression for the concentration of functionals $\phi(Z_1,\ldots, Z_m)$, where the variations of $\phi$ on each variable depend on the product of the norms (or semi-norms) of the other variables (as if $\phi$ were a product). We illustrate the importance of this result through various generalisations of the Hanson-Wright concentration inequality and through a study of the random matrix $XDX^T$ and its resolvent $Q = (I_p - \frac{1}{n}XDX^T)^{-1}$, where $X$ and $D$ are random, which are of fundamental interest in statistical machine learning applications.}

\newtheorem{defpro}[theorem]{Definition/Proposition}
\def\bibsection{\section*{References}}

\DeclareMathOperator{\tr}{Tr}
\DeclareMathOperator{\diag}{Diag}

\DeclareMathOperator{\dimm}{dim}
\DeclareMathOperator{\vect}{Vect}
\def\restriction#1#2{\mathchoice
              {\setbox1\hbox{${\displaystyle #1}_{\scriptstyle #2}$}
              \restrictionaux{#1}{#2}}
              {\setbox1\hbox{${\textstyle #1}_{\scriptstyle #2}$}
              \restrictionaux{#1}{#2}}
              {\setbox1\hbox{${\scriptstyle #1}_{\scriptscriptstyle #2}$}
              \restrictionaux{#1}{#2}}
              {\setbox1\hbox{${\scriptscriptstyle #1}_{\scriptscriptstyle #2}$}
              \restrictionaux{#1}{#2}}}
\def\restrictionaux#1#2{{#1\,\smash{\vrule height .8\ht1 depth .85\dp1}}_{\,#2}} 

\begin{document}
\section*{Introduction}

Among the various assumptions one could pose on random vectors $Z_1,\ldots, Z_m$ to study the concentration of a functional $\phi(Z_1,\ldots, Z_m)$ of limited variations (on $Z_1,\ldots, Z_m$), Concentration of measure hypotheses provide flexible properties that allow one to $(i)$ characterize a wide range of settings where, in particular, the independent entries hypothesis is relaxed, and $(ii)$ to obtain rich concentration inequalities with precise convergence bounds. 
The historical result of concentration of measure theory was first obtained on the uniform distribution on the sphere by Lévy \cite{Le51} and later Milman and Gromov extended the approach to other families of distributions, in particular involving isoperimetric inequalities and the Ricci curvature in \cite{gromov1983topological}. An important part of the theory was developed by Talagrand, whose results are not discussed here, and a full description of the various results can be found in the monographs \cite{LED05,BLM13}.

To present the simplest picture possible, we admit for the moment that what we call ``concentrated vectors'' (or ``Lipschitz concentrated vectors'') are transformations $X = F(Z) \in \mathbb R^p$ of a Gaussian vector $Z \sim \mathcal N(0,I_d)$ for a given $1$-Lipschitz (for the Euclidean norm) mapping $F : \mathbb R^d \to \mathbb R^p$. This class of random vectors derives from a core result of concentration of measure theory \cite[Corollary 2.6]{LED05}, which states that for any $\lambda$-Lipschitz mapping $f: \mathbb R^d \to \mathbb R$ (where $\mathbb R^d$ and $\mathbb R$ are endowed with the Euclidean norm $\|\cdot\|$ and the absolute value $|\cdot|$, respectively),
\begin{align}\label{eq:first_concentration_inequality}
  \forall t >0 : \mathbb P \left(\left\vert f(Z) - \mathbb E[f(Z)]\right\vert \geq t\right) \leq C e^{-(t/c\lambda)^2},
\end{align}
where $C=2$ and $c = \sqrt 2$ (these constants do not depend on the dimensions $d$ !). Note that the concentration rate is proportional to the Lipschitz parameter of $f$. In particular, this implies that the standard deviation of the random variable $f(Z)$ -- called the ``$\lambda$-Lipschitz observation of $Z$'' -- does not depend on the dimension $d$ (if $\lambda$ remains constant as $d$ tends to $\infty$).
We briefly denote this property as $Z \propto C \mathcal E_2(c)$ or, if we are in the quasi-asymptotic regime where the dimension $d$ (or $p$) is large, we do not pay attention to the constants appearing in the exponential bound (as long as $C,c \leq_{d \to \infty} O(1)$, the result would not be much different) and instead write $Z \propto \mathcal E_2$. 

We can then derive a variety of concentration inequalities for any observation $g(F(Z))$ for $g: \mathbb R^p \to \mathbb R$ Lipschitz. If $F$ is, say, $\sigma$-Lipschitz, with $\sigma$ possibly depending on the dimension, we have the concentration $X = F(Z) \propto \mathcal E_2(\sigma)$. This succinct notation ($X \propto \mathcal E_2(\sigma)$ is analogous to \eqref{eq:first_concentration_inequality} with $\sigma$ replacing $c$) shows only the central quantity describing the concentration of $X$, namely the ``observable diameter of $X$'': $O(\sigma)$. In fact, the implicit concentration inequalities constrain the standard deviations of \textit{any} $\nu$-Lipschitz observations of $X$ to be of the order of $O(\nu \sigma)$.




\medskip

The goal of this article is to go beyond the Lipschitz case and express, using our shorthand notation, the concentration of products of concentrated random vectors. As an illustrative example, let $Y = \underbrace{X \circ \cdots \circ X}_{m\textmd{ times}} \in \mathbb R^p$, for a given product ``$\circ $'' satisfying:
\begin{align}\label{eq:norm_d_algebre}
  \forall x_1,\ldots, x_m \in \mathbb R^p, \quad \|x_1 \circ \cdots \circ x_m \| \leq \|x_1\|\cdots \|x_m\|.
\end{align}
(For example, $\circ $ could be the entry-wise product, and the norm would be the infinite norm).
In particular, if $\mathbb E[X] = 0$, we will see that
\begin{align}\label{eq:concentration_X_X_X_X_X}
   Y \propto \mathcal E_{2}(p^{\frac{m-1}{2}} \sigma^m) + \mathcal E_{\frac{2}{m}}(\sigma^m),
 \end{align} where $\mu$ satisfies $\mathbb E[\|X\|] \leq O(\mu)$, which in our framework means that there exist two constants $C,c>0$ (independent of $p$) such that for any $1$ Lipschitz mapping $f : \mathbb R^p \to \mathbb R$, $\forall t >0$.
\begin{align*}
   \mathbb P \left(\left\vert f(Y) - \mathbb E[f(Y)]\right\vert \geq t\right) \leq C \exp \left(- \left(\frac{t}{cp^{\frac{m-1}{2}} \sigma^m}\right)^2\right) + C \exp \left(-\left(\frac{t}{c\sigma^m}\right)^{2/m}\right).
\end{align*}
We see here that the term $\mathcal E_{\frac{2}{m}}(\sigma^m)$ in \eqref{eq:concentration_X_X_X_X_X} controls the tail of the distribution of $f(Y)$, but its first moments are controlled by the term $\mathcal E_{2}(p^{\frac{m-1}{2}} \sigma^m)$. Specifically, its standard deviation is of the order of $O(p^{\frac{m-1}{2}} \sigma^m)$, the observable diameter of $Y$. In a sense, the result is quite intuitive, looking back at the algebraic inequality \eqref{eq:norm_d_algebre}: here $\mathbb E[\|X\|]\leq O(\sqrt{p})$ and the variations of $Y$ are bounded by $\mathbb E[\|X\|]^{m-1}$ times the variation of $X$.
This simple scheme generalizes to more complex products of random vectors $X_1,\ldots, X_m$ belonging to different normed vector spaces, where $Y$ may not be a multilinear mapping of $(X_1,\ldots, X_m)$ but still satisfy an inequality similar to \eqref{eq:norm_d_algebre} (with semi-norms possibly replacing some of the norms). The complete description of these possible settings is the central result of this paper: Theorem~\ref{the:Concentration_produit_de_vecteurs_d_algebre_optimise}.
 
\medskip
Quite similar results with this multiple exponential regime can be found in \cite{latala2006estimates} in the Gaussian case and later in \cite{adamczak2015concentration}, which extends Latala's result to more general hypotheses of concentration.
In several aspects, the approach of \cite{adamczak2015concentration} may seem more "structural" --but also somehow less flexible-- than ours, in particular because the lower bound given in the case of polynomials of Gaussian variables meets (up to a constant) the upper bound given in the case of d-differentiable functionals of more general variables. However, it does not seem that their result can recover ours - even if one does not consider the specific case of nondifferentiable functionals and operations on general algebras (and not only on random variables), which we are the only ones to treat, and the hypotheses of concentration, which are quite similar. We give an expression of their result and compare it with ours after Theorem~\ref{the:Concentration_produit_de_vecteurs_d_algebre_optimise}.

\medskip

\sloppypar{As a simple but fundamental application of our main result is given by the Hanson-Wright inequality expressing the concentration of $X^TAX$, where $A\in\mathcal M_{p}$ and $X$ is either a random vector of $\mathbb R^p$ or a random matrix of $M_{p,n}$. The historical result (see \cite{BLM13} or \cite{VER17}) was given on random vectors $X=(X_1,\ldots, X_p) \in \mathbb R^p$ with independent sub-Gaussian entries, satisfying some sub-Gaussian concentration inequality, say $X_i \propto \mathcal E_2(K)$, then one gets the concentration:
\begin{align}\label{eq:hanson_wright_1_intro}
   X^TAX \propto \mathcal E_2(K^2\|A\|_F) + \mathcal E_1(K^2\|A\|)
\end{align} 
Note that the standard deviation of $X^TAX$ is of order $\|A\|_F$ (where $\|\cdot\|_F$ is the Frobenius norm). Based on concentration of measure hypotheses (allowing for dependence between entries), good concentration inequalities were already obtained in \cite{VU14} (with a term $\mathcal E_2(\sqrt{\log n}\|A\|_F)$ replacing $\mathcal E_2(\|A\|_F)$) and then improven in \cite{ADA14} to exactly the same result as \eqref{eq:hanson_wright_1_intro}. 

Although we extend this concentration result to the case of random matrices $X, Y \in \mathcal{M}_{p,n}$, which is not a big improvement, unlike \cite{VU14} and \cite{ADA14}, we do not take convex concentration hypotheses (derived from a well-known result of Talagrand), because Theorem~\ref{the:Concentration_produit_de_vecteurs_d_algebre_optimise} could not be proven in this setting\footnote{A result analogous to Theorem~\ref{the:Concentration_produit_de_vecteurs_d_algebre_optimise} can be proven in the convex concentration setting and for the entry wise product in $\mathbb R^p$ or the matrix product in $\mathcal M_{p,n}$, but this is not the purpose of this article. }.}

To illustrate our central result with more general products (when $m\geq 3$), we consider the concentration of $X^TDY$, where $D \in \mathcal M_{n}$ is a diagonal random matrix and $X,Y \in \mathcal M_{p,n}$ are two random matrices, all satisfying $X,D,Y \propto \mathcal E_2$. In a last step, in the same setting, to go beyond the multilinear case, we consider the concentration of the resolvent $Q = (I_p - \frac{1}{n}XDX^T)^{-1}$ studied in \cite{PAJ09, GUE14}, but with a diagonal matrix $D$ possibly depending on $X,Y$. This setting appears in robust regression problems \cite{ELK13, MAI19, SED21}. 
With the possibly complex dependencies between the entries of $x_i$ they allow, the concentration of measure hypotheses are very light compared to the classical Gaussian hypotheses adopted in large dimensional statistics and statistical learning \cite{HUA17, DEN20}. 
To obtain a good concentration of $Q$, one must assume that the columns $x_1,\ldots, x_n$ of $X$ are all independent and that for all $i \in [n]$ there exists a diagonal random matrix $D^{(i)} \in \mathcal M_n$, not too far from $D$ and independent with $x_i$.

\medskip

\sloppypar{The remainder of the article is organized as follows. After presenting Lipschitz concentration of measure hypotheses and basic probabilistic inferences \textbf{(I)}, we introduce the class of \emph{linearly concentrated random vectors} \textbf{(II)} and explain how their norm can be controlled in generic normed vector spaces \textbf{(III)}. We then briefly discuss the fact that the random vector $(X_1,\ldots, X_m)$ (as a whole) is not always concentrated if one only assumes that each of the $X_i$'s, $i\in[m]$, is concentrated \textbf{(IV)}. This provides us with the ingredients to establish the concentration of $\phi(X_1,\ldots, X_m)$ in Theorem~\ref{the:Concentration_produit_de_vecteurs_d_algebre_optimise}, the core result of the article, and to provide a first set of elementary consequences \textbf{(V)}. As an application of Theorem~\ref{the:Concentration_produit_de_vecteurs_d_algebre_optimise}, we next provide a generalization of the Hanson-Wright Theorem \textbf{(VI)}. Then we end the article with a study of the concentration of $XDX^T$ \textbf{(VII)} and the resolvent $Q = (I_p - \frac{1}{n}XDX^T)^{-1}$ \textbf{(VIII)}. 
}

\section{Basics and notations of the concentration of measure framework}
To discuss concentration of measure, we choose here to adopt the viewpoint of Levy families where the goal is to track the influence of the vector dimension over the concentration. Specifically, given a sequence of random vectors $(Z_p)_{p\geq \mathbb N}$ where each $Z_p$ belongs to a space of dimension $p$ (typically $\mathbb R^p$), we wish to obtain inequalities of the form:
\begin{align}\label{eq:inegalite_de_concentration}
  \forall p \in \mathbb N, \forall t>0 : \mathbb P \left(\left\vert f_p(Z_p) - a_p\right\vert\geq t\right) \leq \alpha_p(t),
\end{align}
where, for every $p \in \mathbb N$, $\alpha_p : \mathbb R^+ \rightarrow [0,1]$ is called a \textit{concentration function}: it is left-continuous, decreasing, and tends to $0$ at infinity; $f_p : \mathbb R^p \rightarrow \mathbb R$ is a $1$-Lipschitz function; and $a_p$ is either a deterministic variable (typically $\mathbb E[f_p(Z_p)]$) or a random variable (for instance $f_p(Z_p')$ with $Z_p'$ an independent copy of $Z_p$). The sequences of random vectors $(Z_p)_{p\geq 0}$ satisfying inequality \eqref{eq:inegalite_de_concentration} for all sequences of $1$-Lipschitz functions $(f_p)_{p\geq 0}$ are called \textit{Levy families} or more simply \textit{concentrated vectors} (with this denomination, we implicitly omit the dependence on $p$ and abusively call ``vectors'' the \textit{sequences} of random vectors of growing dimension). 

Concentrated vectors admitting an exponentially decreasing concentration function $\alpha_p$ are extremely flexible objects. We dedicate the next two subsections to further definitions of the fundamental notions involved under this setting. These are of central interest to the present article -- this approach is primarily inspired by the Gaussian fundamental example satisfying \eqref{eq:first_concentration_inequality}. 

Our main interest is in two classes of concentrated vectors, characterized by the regularity of the class of admissible sequences of functions $(f_p)_{p\in \mathbb N}$ satisfying \eqref{eq:inegalite_de_concentration}. When \eqref{eq:inegalite_de_concentration} holds for all the $1$-Lipschitz mappings $f_p$, $Z_p$ is said to be \textit{Lipschitz concentrated};
when true for all $1$-Lipschitz \textit{linear} mappings $f_p$, $Z_p$ is said to be \textit{linearly concentrated} (the convex concentration, not studied here, occurs when \eqref{eq:inegalite_de_concentration} is satisfied for all $1$-Lipschitz convex mappings $f_p$ \cite{VU14}). As such, the concentration of a random vector $Z_p$ is only defined through the concentration of what we refer to as its ``observations'' $f_p(Z_p)$ for all $f_p$ in a specific class of functions. 


We will work with normed (or semi-normed) vector spaces, although concentration of measure theory is classically developed in metric spaces. The presence of a norm (or a semi-norm\footnote{A semi-norm becomes a norm when it satisfies the implication $\|x\| = 0 \Rightarrow x=0$.}) on the vector space is particularly important when establishing the concentration of a product of random vectors.
\begin{defpro}\label{def:concentrated_sequence}
  Given a sequence of normed or semi-normed vector spaces $(E_p, \Vert \cdot \Vert_p)_{p\geq 0}$, a sequence of random vectors\footnote{A random vector $Z$ of $E$ is a measurable function from a probability space $(\Omega, \mathcal F, \mathbb P)$ to the normed vector space $(E, \|\cdot\|)$ (endowed with the Borel $\sigma$-algebra); one should indeed write $Z: \Omega \to E$, but we abusively simply denote $Z\in E$.} $(Z_p)_{p\geq 0} \in \prod_{p\geq 0} E_p$, a sequence of positive reals $(\sigma_p)_{p\geq 0} \in \mathbb R_+ ^{\mathbb N}$ and a parameter $q>0$, we say that $Z_p$ is Lipschitz \emph{$q$-exponentially concentrated} with \emph{observable diameter} of order $O(\sigma_p)$ iff one of the following three equivalent assertions is satisfied:\footnote{Aside from the fact that they all give interesting interpretation of the concentration of a random vector, all three characterizations can be relevant, depending on the needs:
\begin{itemize}
  \item the characterization with the independent copy is employed in Remark~\ref{rem:vecteur_conditionne} and in the proof of Theorem~\ref{the:Concentration_produit_de_vecteurs_d_algebre_optimise};
  \item the characterization with the median is employed in the proof of Lemma~\ref{lem:concentration_sous_ensemble};
  \item the characterization with the expectation, likely the most intuitive, is used to establish Proposition~\ref{pro:estimation_XDY}, Theorem~\ref{the:concentration_resolvent} and Lemma~\ref{lem:concentration_Q_m_i_x_i}.
\end{itemize}}
  \begin{itemize}
     \item $\exists C,c>0 \ | \ \forall p \in \mathbb N, \forall \ 1\text{-Lipschitz} \ f : E_p \rightarrow \mathbb{R}, \forall t>0:$ 
     \begin{center}
      $\mathbb P \left(\left\vert f(Z_p) - f(Z_p')\right\vert\geq t\right) \leq C e^{-(t/c\sigma_p)^q}$,
     \end{center}
     \item $\exists C,c>0 \ | \ \forall p \in \mathbb N, \forall \  1\text{-Lipschitz} \ f : E_p \rightarrow \mathbb{R}, \forall t>0:$ 
     \begin{center}
      $\mathbb P \left(\left\vert f(Z_p) - m_{f}\right\vert\geq t\right) \leq C e^{-(t/c\sigma_p)^q}$,
     \end{center}
     \item $\exists C,c>0 \ | \ \forall p \in \mathbb N, \forall \ 1\text{-Lipschitz} \ f : E_p \rightarrow \mathbb{R}, \forall t>0:$ 
     \begin{center}
      $\mathbb P \left(\left\vert f(Z_p) - \mathbb E[f(Z_p)]\right\vert\geq t\right) \leq C e^{-(t/c\sigma_p)^q}$, 
     \end{center}
   \end{itemize} 
   where $Z_p'$ is an independent copy of $Z_p$ and $m_{f}$ is a median\footnote{$\mathbb P \left(f(Z_p) \geq m_{f}\right)\geq \frac12$ and $\mathbb P \left(f(Z_p) \leq m_{f}\right) \geq \frac{1}{2}$.} of $f(Z_p)$; the mappings $f$ are $1$-Lipschitz for the norm (or semi-norm) $\|\cdot\|_p$. We denote in this case $Z_p \propto \mathcal E_{q}(\sigma_p)$ (or more simply $Z \propto \mathcal E_{q}(\sigma)$).  
\end{defpro}
The equivalence between the three definitions is proven in \cite{LED05} (full details are given in \cite[Propositions~1.2, 1.18 Corollary~1.24]{louart2018concentration}), also the existence of the expectation in the third point is guaranteed if one of the two first item is valid thanks to Fubini Theorem that transforms concentration inequalities into bounds on the moments (\cite[Proposition 1.7]{LED05}).

\begin{remark}[Quasi-asymptotic regime]\label{rem:definiiton_of_the_order}
  Most of our results will be expressed under the quasi-asymptotic regime where $p$ is large. Sometimes, it will be natural to index the sequences of random vectors with two (or more) indices (e.g., the numbers of rows and columns for random matrices): in these cases, the quasi-asymptotic regime is not well defined since the different indices could have different convergence speed. This issue is overcome with the extensive use of the notation $O(\sigma_t)$, where $t \in \Theta$ designates the (possibly multivariate) index. Given two sequences $(a_t)_{t\in \Theta},(b_t)_{t\in \Theta} \in \mathbb R_+^\Theta$, we will denote $a_t \leq O(b_t)$ if there  exists a constant $C>0$ such that $\forall t\in\Theta$, $a_t \leq C b_t$ and $a_t \geq O(b_t)$ if $\forall t \in \Theta$, $a_t \geq C b_t$. A ``constant'' $K>0$ is a quantity that does not depend on our asymptotic variables, it satisfies therefore $K \leq O(1)$ and $K \geq O(1)$. For a concentrated random vector $Z_t \propto \mathcal E_q(\sigma_t)$, any sequence $(\nu_t)_{t\in\Theta} \in \mathbb R_+^\Theta$ such that $\sigma_t \leq O(\nu_t)$ is also an observable diameter of $Z_t$. When $\sigma_t \leq O(1)$, we simply write $Z_t \propto \mathcal E_q$.
\end{remark}
\begin{remark}[Metric versus normed spaces]\label{rem:definiiton_concentration_avec_metrique}
  It is more natural, as done in \cite{LED05}, to introduce the notion of concentration in metric spaces, as one only needs to resort to Lipschitz mappings which merely require a metric structure on $E$. 
  However, to exploit Theorem~\ref{the:Concentration_produit_de_vecteurs_d_algebre_optimise}, we will need to control the amplitude of concentrated vectors which is easily conducted when the metric is a norm, under linear concentration assumptions. 

\end{remark}
When a concentrated vector $Z_p\propto \mathcal E_q(\sigma_p)$ takes values only on some subset $A_p \equiv Z_p(\Omega)\subset E_p$ (where $\Omega$ is the universe), it might be useful to be able to establish the concentration of observations $f_p(Z_p)$ where $f_p$ is only $1$-Lipschitz on $A_p$ (and possibly non Lipschitz on $E_p \setminus A_p$). This would be an immediate consequence of Definition~\ref{def:concentrated_sequence} if one were able to extend $\restriction{f_p}A_p$ into a mapping $\tilde f_p$ Lipschitz on the whole vector space $E_p$; but this is rarely possible. Yet, the observation $f_p(Z_p)$ does concentrate under the hypotheses of Definition~\ref{def:concentrated_sequence}.
\begin{lemma}[Concentration of locally Lipschitz observations]\label{lem:concentration_sous_ensemble}
  Given a sequence of random vectors $Z_p: \Omega \to E_p$, satisfying $Z_p\propto \mathcal E_q(\sigma_p)$, for any sequence of mappings $f_p: E_p \to \mathbb R$, which are $1$-Lipschitz on $Z_p(\Omega)$, we have the concentration $f_p(Z_p) \propto \mathcal E_q(\sigma_p)$.
\end{lemma}
\begin{proof}
  considering a sequence of median $m_{f_p}$ of $f_p(Z_p)$ and the (sequence of) sets $S_p = \{f_p \leq m_{f_p}\}\subset E_p$, if we note for any $z \in E_p$ and $U \subset E_p$, $U\neq \emptyset$, $d(z, U) = \inf\{\|z-y\|, y \in U\}$, then we have for any $z\in A$ and $t >0$:
  \begin{align*}
    f_p(z) \geq m_{f_p} + t&
    &\Longrightarrow&
    &d(z,S_p) \geq t\\
    f_p(z) \leq m_{f_p} - t&
    &\Longrightarrow&
    &d(z,S_p^c) \geq t,
  \end{align*}
  since $f_p$ is $1$-Lipschitz on $A$. Therefore, since $z \mapsto d(z,S_p)$ and $z \mapsto d(z,S_p^c)$ are both $1$-Lipschitz on $E$ and both admit $0$ as a median ($\mathbb P(d(Z_p, S_p) \geq 0) = 1 \geq \frac{1}{2}$ and $\mathbb P(d(Z_p, S_p) \leq 0) \geq \mathbb P(f_p(Z_p) \leq m_{f_p})\geq \frac{1}{2}$),
  \begin{align*}
    \mathbb P \left( \left\vert f_p(Z_p) - m_{f_p}\right\vert \geq t\right) 
    &\leq \mathbb P \left(d(Z_p, S_p) \geq t\right)+\mathbb P \left(d(Z_p, S_p) \geq t\right)
    \leq 2C e^{- (t/c\sigma_p)}.
  \end{align*}
\end{proof}

One could argue that, instead of Definition~\ref{def:concentrated_sequence}, we could have posed hypotheses on the concentration of $Z_p$ on $Z_p(\Omega)$ only; however, we considered the present definition of concentration already quite complex as it stands. This locality aspect must be kept in mind: it will be exploited to obtain the concentration of products of random vectors.

Lemma~\ref{lem:concentration_sous_ensemble}  is particularly interesting when working with conditioned variables.\footnote{Letting $X: \Omega \to E$ be a random vector and $\mathcal A \subset \Omega$ be a measurable subset of the universe $\Omega$, $\mathcal A\in \mathcal F$, when $\mathbb P(\mathcal A)>0$, the random vector $X \ | \mathcal A$ designates the random vector $X$ conditioned with $\mathcal A$ defined as the measurable mapping $(\mathcal A, \mathcal F_A, \mathbb P/ \mathbb P(\mathcal A)) \to (E,\|\cdot\|)$ satisfying: $\forall \omega \in \mathcal A$, $(X \ | \ \mathcal A)(\omega) = X(\omega)$. When there is no ambiguity, we will allow ourselves to designate abusively with the same notation ``$\mathcal A$'' the actual $\mathcal A \subset \Omega$ and the subset $X(\mathcal A) \subset E$.}

\begin{remark}[Concentration of conditioned vectors]\label{rem:vecteur_conditionne}
  Given a (sequence of) random vectors $Z \propto \mathcal E_q(\sigma)$ and a (sequence of) events $\mathcal A $ such that $\mathbb P(\mathcal A) \geq O(1)$, it is straightforward to show that $(Z \ | \ \mathcal A) \propto  \mathcal E_q(\sigma)$, since there exist two constants $C,c>0$ such that for any $p \in \mathbb N$ and any $1$-Lipschitz mapping $f : E_p \to \mathbb R$:
  \begin{align*}
    \forall t>0: \mathbb P \left(\left\vert f(Z_p) - f(Z_p')\right\vert \geq t \ | \ \mathcal A\right) \leq \frac{1}{\mathbb P(A)} \mathbb P \left(\left\vert f(Z_p) - f(Z_p')\right\vert \geq t \right) \leq C e^{-c(t/\sigma)^q}. 
  \end{align*}
  This being said, Lemma~\ref{lem:concentration_sous_ensemble} allows us to obtain the same concentration inequality for any mapping $f: E_p \to \mathbb R$ $1$-Lipschitz on $Z_p(\mathcal A)$ (that will be abusively denoted $\mathcal A$ later on).
\end{remark}
A simple but fundamental consequence of Definition~\ref{def:concentrated_sequence} is that, as announced in the introduction, any Lipschitz transformation of a concentrated vector is also a concentrated vector. The Lipschitz coefficient of the transformation controls the concentration.
\begin{proposition}[Stability through Lipschitz mappings]\label{pro:stabilite_lipschitz_mappings}
  In the setting of Definition~\ref{def:concentrated_sequence}, given a sequence $(\lambda_p)_{p\geq 0} \in \mathbb R_+^{\mathbb N}$, a supplementary sequence of normed vector spaces $(E_p', \Vert \cdot \Vert_p')_{p\geq 0}$ and a sequence of $\lambda_p$-Lipschitz transformations $F_p : (E_p, \Vert \cdot \Vert_p) \rightarrow (E_p', \Vert \cdot \Vert_p')$, we have
  \begin{align*}
    Z_p \propto \mathcal E_q(\sigma_p)&
    &\Longrightarrow&
    &F_p(Z_p) \propto \mathcal E_q(\lambda_p\sigma_p).
  \end{align*}
\end{proposition}

There exists a range of elemental concentrated random vectors, which may be found for instance in the monograph \cite{LED05}. We recall below some of the major examples. In the following theorems, we only consider sequences of random vectors of the normed vector spaces $(\mathbb R^p, \left\Vert \cdot\right\Vert)$. For readability of the results, we will omit the index $p$.
\begin{theorem}[Fundamental examples of concentrated vectors]\label{the:concentration_vecteur_gaaussien}
\begin{sloppypar} The following sequences of random vectors are concentrated and satisfy $ Z \propto \mathcal E_2$: 
\begin{itemize}
  \item $Z$ is uniformly distributed on the sphere $\sqrt p \mathbb S^{p-1}$.
  \item $Z \sim \mathcal N(0,I_p)$ has independent standard Gaussian entries.
  \item $Z$ is uniformly distributed on the ball $\sqrt p\mathcal B = \{x \in \mathbb R^p, \Vert x \Vert \leq \sqrt p\}$.
  \item $Z$ is uniformly distributed on $[0,\sqrt p]^p$.
  \item $Z$ has the density $d\mathbb P_Z(z) = e^{-U(z)}d\lambda_p(z)$ where $U : \mathbb R^p \rightarrow \mathbb R$ is a positive functional with Hessian bounded from below by, say, $cI_p$ with $c\geq O(1)$ and $d\lambda_p$ is the Lebesgue measure on $\mathbb R^p$.
\end{itemize}
\end{sloppypar}
Some fundamental results also give concentrations $Z \propto \mathcal E_1$ (when $Z \in \mathbb R^p$ has independent entries with density density $\frac 1 2 e^{-\vert \cdot \vert}d\lambda_1$, \cite{TAL95}) or $Z \propto \mathcal E_q\left(p^{-\frac{1}{q}}\right)$ (when $Z\in \mathbb R^p$ is uniformly distributed on the unit ball of the norm $ \Vert \cdot\Vert_q$, \cite{LED05}).
 .
\end{theorem}

A very explicit characterization of \textit{exponential concentration} is given by a bound on the different centered moments that lets appear as expected a dependence on the observable diameter.
\begin{proposition}[Characterization with the centered moments]\label{pro:characterization_moments}
  \cite[Proposition 1.10]{LED05} A random vector $Z \in E$ is $q$-exponentially concentrated with an observable diameter of order $\sigma$ (i.e., $Z \propto \mathcal E_q(\sigma)$) if and only if there exist two constants $C,c >0$ such that for all $p \in \mathbb N$, any (sequence of) $1$-Lipschitz functions $f : E_p \to \mathbb R$:
  \begin{align}\label{eq:characterization_concentration_avec_moments}
    \forall r >0 : \mathbb{E}\left[\left\vert f(Z_p) - f(Z'_p)\right\vert^r\right] \leq C \left(\frac{r}{q}\right)^{\frac{r}{q}}(c\sigma_p)^r,
   \end{align} 
   where $Z'_p$ is an independent copy of $Z_p$.
   Inequality~\eqref{eq:characterization_concentration_avec_moments} also holds if we replace $f(Z'_p)$ with $\mathbb E[f(Z_p)]$ (of course the constants $C$ and $c$ might be slightly different).
\end{proposition}

\medskip

The Lipschitz-concentrated vectors described in Definition~\ref{def:concentrated_sequence} belong to the larger class of linearly concentrated random vectors that only requires the linear observations to concentrate. This ``linear concentration'' presents less stability properties than those described by Proposition~\ref{pro:stabilite_lipschitz_mappings} but is still a relevant notion because:
\begin{enumerate}
    \item although it must be clear that a concentrated vector $Z$ is generally far from its expectation (for instance Gaussian vectors lie on an ellipse), it can still be useful to have some control on $\Vert Z -\mathbb E[Z] \Vert$ to express the concentration of product of vectors; linear concentration is a sufficient assumption for this control,
    \item there are some examples (Proposition~\ref{pro:concentration_lineaire_YAX} and~\ref{pro:concentration_lineaire_XDY}) where we can only derive linear concentration inequalities from a Lipschitz concentration hypothesis. In that case, we say that the Lipschitz concentration ``degenerates'' into linear concentration that appears as a ``residual'' concentration property. 
\end{enumerate}

\section{Linear concentration and control on high order statistics}\label{sse:linear_concentration_norm}



\begin{definition}[Linearly concentrated vectors]\label{def:linear_concentration}
  Given a sequence of normed vector spaces $(E_p, \Vert \cdot \Vert_p)_{p\geq 0}$, a sequence of random vectors $(Z_p)_{p\geq 0} \in \prod_{p\geq 0} E_p$, a sequence of deterministic vectors $(\tilde Z_p)_{p\geq 0} \in \prod_{p\geq 0} E_p$, a sequence of positive reals $(\sigma_p)_{p\geq 0} \in \mathbb R_+ ^{\mathbb N}$ and a parameter $q>0$, $Z_p$ is said to be \emph{$q$-exponentially linearly concentrated} around the \emph{deterministic equivalent} 
  $\tilde Z_p$ with an \emph{observable diameter} of order $O(\sigma_p)$ iff there exist two constants $c,C >0$ such that $\forall p \in \mathbb N$ and for any unit-normed linear form $f \in E_p'$ ($\forall p \in \mathbb N$, $\forall x \in E$: $|f(x)| \leq \|x\|$):
  \begin{align*}
    \forall t>0: \ 
      \mathbb P \left(\left\vert f(Z_p ) - f(\tilde Z_p)\right\vert\geq t\right) \leq C e^{(t/c\sigma_p)^q}.
  \end{align*}
  When the property holds, we write $Z \in \tilde Z \pm \mathcal E_q(\sigma)$. If it is unnecessary to mention the deterministic equivalent, we will simply write $Z \in \mathcal E_q(\sigma)$.
  If we just need to control its amplitude, we can write $Z \in O(\theta)\pm \mathcal E_q(\sigma)$ when $\|\tilde Z_p \| \leq O(\theta_p)$.
\end{definition}
When $q=2$, we retrieve the well known class of sub-Gaussian random vectors. We need this definition with generic $q$ to prove Proposition~\ref{pro:concentration_lineaire_YAX} which involves a weaker than $\mathcal E_2$ tail decay.

Of course linear concentration is stable through affine transformations. 
\begin{proposition}[Stability through affine mappings]\label{pro:stabilite_concentration_lineaire_affine}
  \sloppypar{Given two (sequences of) normed vector spaces $(E,\|\cdot\|_E)$ and $(F,\|\cdot\|_F)$, a (sequence of) random vectors $Z \in E$, a (sequence of) deterministic vectors $\tilde Z \in E$ and a (sequence of) affine mappings $\phi : E \to F$ such that $\forall x \in E: \|\phi(x) - \phi(0)\|_F\leq \lambda \|x \|_E$:
  \begin{align*}
    Z \in \tilde Z \pm \mathcal E_q(\sigma)&
    &\Longrightarrow&
    &\phi(Z) \in \phi(\tilde Z) \pm \mathcal E_q(\lambda\sigma).
  \end{align*}}
\end{proposition}
When the expectation can be defined, there exists an implication link between Lipschitz concentration (Definitions~\ref{def:concentrated_sequence}) and linear concentration (Definition~ \ref{def:linear_concentration}).
\begin{lemma}\label{lem:conc_lip_imply_conc_linear}
  Given a normed space $(E,\|\cdot \|)$ and a random vector $Z \in E$ admitting an expectation, we have the implication:
  \begin{align*}
    Z \propto \mathcal E_q(\sigma)&
    &\Longrightarrow&
    &Z \in \mathbb E[Z] \pm \mathcal E_q(\sigma).
  \end{align*}
\end{lemma}

This implication becomes an equivalence in law dimensional spaces (i.e. when the sequence index ``$p$'' is not linked to the dimension of the vector spaces $E_p$); then the distinction between linear concentration and Lipschitz concentration is not relevant anymore.

The next lemma is a formal expression of the assessment that ``any deterministic vector located at a distance smaller than the observable diameter to a deterministic equivalent is also a deterministic equivalent'', we omit the proof that straightforward.
\begin{lemma}\label{lem:diametre_observable_pivot}
  Given a random vector $Z \in E$, a deterministic vector $\tilde Z \in E$ such that $Z \in \tilde Z \pm \mathcal E_q(\sigma)$, we have the equivalence:
  \begin{align*}
    Z \in \tilde Z' \pm \mathcal E_q(\sigma)&
    &\Longleftrightarrow&
    & \left\Vert \tilde Z - \tilde Z'\right\Vert \leq O(\sigma)
  \end{align*}
\end{lemma}

\begin{definition}[Centered moments of random vectors]\label{def:moment_vecteur_aleatoire}
  Given a random vector $X \in \mathbb R^p$ and an integer $r \in \mathbb N$, we call the ``$r^{\textit{th}}$ centered moment of $X$'' the symmetric $r$-linear form
$C_r^X : (\mathbb R^p)^r \to \mathbb R$ defined for any $u_1,\ldots,u_r \in \mathbb R^p$ by
\begin{align*}
  C_r^X(u_1,\ldots,u_p) = \mathbb E \left[\prod_{i=1}^p\left(u_i^TX - \mathbb E[u_i^TX]\right)\right].
\end{align*}
When $r=2$, the centered moment is the covariance matrix.
\end{definition}
We define the operator norm of an $r$-linear form $S$ of $\mathbb R^p$ as
\begin{align*}
  \|S\| \equiv \sup_{\|u_1\| ,\ldots,\|u_r\|\leq 1}S(u_1,\ldots,u_p).
\end{align*}
When $S$ is symmetric, we employ the simpler formula $\|S\| = \sup_{\|u\|\leq 1}S(u,\ldots,u)$.
We then have the following characterization, similar to Proposition~\ref{pro:characterization_moments} (refer to \cite[Proposition~1.21, Lemma~1.21]{louart2018concentration} for the technical arguments required to go from a bound on $r \in\mathbb N$ to a bound on $r >0$).
\begin{proposition}[Moment characterization of linear concentration]\label{pro:carcaterisation_vecteur_linearirement_concentre_avec_moments}
  Given $q>0$, a sequence of random vectors $X_p \in \mathbb R^p$, and a sequence of positive numbers $\sigma_p >0$, we have the following equivalence:
  \begin{align*}
    X \in \mathcal E_q(\sigma)&
    &\Longleftrightarrow&
    &\exists C,c>0, \forall p \in \mathbb N, \forall r \geq q : \|C^{X_p}_r\| \leq C \left(\frac{r}{q}\right)^{\frac{r}{q}}(c\sigma_p)^r
  \end{align*}
\end{proposition}
In particular, if we note $C = \mathbb E[XX^T] - \mathbb E[X]\mathbb E[X]^T$, the covariance of $X\in \mathcal E_q(\sigma)$, we see that $\|C\| \leq O(\sigma^2)$, if in addition $X \in O(\sigma) \pm \mathcal E_q(\sigma)$ (which means that $\|\mathbb E[X]\| \leq O(\sigma)$), then $\|\mathbb E[XX^T] \| \leq O(\sigma^2)$

\medskip

With these results at hand, we are in position to explain how a control on the norm can be deduced from a linear concentration hypothesis.

\section{Control of the norm of linearly concentrated random vectors}\label{sec:concentration_norm}
Given a random vector $Z \in (E,\|\cdot \|)$, if $Z \in \tilde Z \pm \mathcal E_q(\sigma)$, the control of $\|Z - \tilde Z \|$ can be done easily when the norm $\left\Vert  \cdot \right\Vert$ can be defined as the supremum on a set of linear forms; for instance when $(E,\|\cdot \|) = (\mathbb{R}^p, \left\Vert \cdot \right\Vert_{\infty})$: $\left\Vert x \right\Vert_{\infty} =\sup_{ 1 \leq i \leq p} e_i^T x$ (where $(e_1, \ldots,e_p)$ is the canonical basis of $\mathbb R^p$). We can then bound:
\begin{align*}
    \mathbb{P}\left(\Vert Z - \tilde Z \Vert_{\infty} \geq t\right) 
    &=\mathbb{P}\left(\sup_{1\leq i \leq p} e_i^T(Z - \tilde Z) \geq t\right) \\
    &\leq \min \left(1, p \sup_{1\leq i \leq p}\mathbb{P}\left( e_i^T(Z - \tilde Z) \geq t\right)\right) \\
    &\leq \min \left(1,p C e^{-c(t/\sigma)^q}  \right) \ \ \leq \ \max(C,e) \exp \left(-\frac{ct^q}{2\sigma^q \log(p)}\right),
\end{align*}
for some constants $c, C> 0$ ($C\leq O(1)$, $c\geq O(1)$).

To manage the infinity norm, the supremum is taken on a finite set $\{e_1, \ldots e_p\}$.
Problems arise when this supremum must be taken on an infinite set. For instance, for the Euclidean norm, the supremum is taken over the whole unit ball $\mathcal B_{\mathbb R^p} \equiv \{u \in \mathbb R^p, \Vert u \Vert \leq 1\}$ since for any $x \in \mathbb R^p$, $\left\Vert x \right\Vert = \sup\{ u^T x, \Vert u \Vert \leq 1\}$.
This loss of cardinality control can be overcome if one introduces so-called $\varepsilon$-nets to discretize the ball with a net $\{u_i\}_{i \in I}$ (with $I$ finite -- $|I |<\infty$) in order to simultaneously :
\begin{enumerate}
  \item approach sufficiently the norm to ensure $$\mathbb{P}\left(\Vert Z - \tilde Z \Vert_{\infty} \geq t\right)\approx\mathbb{P}\left(\sup_{i \in I} u_i^T(Z - \tilde Z) \geq t\right),$$
  \item control the cardinality $|I|$ for the inequality $$\mathbb{P}\left(\sup_{i \in I} u_i^T(Z - \tilde Z) \geq t\right)\leq |I |\mathbb{P}\left( u_i^T(Z - \tilde Z) \geq t\right)$$ not to be too loose.
\end{enumerate}
One can then show that there exist two constants $C,c>0$ such that:
\begin{align}\label{eq:concentration_norme_euclidienne}
  \mathbb P(\Vert Z - \tilde Z \Vert\geq t) \leq \ \max(C,e) \exp \left(-\frac{ct^q}{p\sigma^q}\right).
\end{align} 
The approach with $\varepsilon$-nets in $(\mathbb R^p,\|\cdot\|)$ can be generalized to any normed vector space $(E,\|\cdot \|)$ when the norm can be written as a supremum through an identity of the kind :
\begin{align}\label{eq:norm_egal_supremum}
  \forall x \in E : \|x\| = \sup_{\genfrac{}{}{0pt}{2}{f \in H}{\|f\| \leq 1}} f(x), \quad
  \text{with}\  H\subset E' \ \text{and} \ \dim(\vect(H)) < \infty,
 \end{align} for a given $H\subset E'$ (for $E'$, the dual space of $H$) and with $\vect(H)$ the subspace of $E'$ generated by $H$. Such a $H\subset E'$ exists in particular when $(E,\|\cdot\|)$ is a reflexive\footnote{Introducing the mapping $J : E \to E'{}'$ (where $E'{}'$ is the bidual of $E$) satisfying $\forall x \in E$ and $\phi\in E'$: $J(x)(\phi) = \phi(x)$, the normed vector space $E$ is said to be ``reflexive'' if $J$ is onto.} space \cite{JAM57}.
 
When $(E,\|\cdot\|)$ is of infinite dimension, it is possible to establish \eqref{eq:norm_egal_supremum} for some $H \subset E$ when $E$ is reflexive thanks to a result from \cite{JAM57}, or for some choice of semi-norms $\|\cdot\|$. 
Without going into details, we introduce the notion of \textit{norm degree} which will help us adapt the concentration rate $p$ appearing in the exponential term of concentration inequality \eqref{eq:concentration_norme_euclidienne} (concerning $(\mathbb R^p, \| \cdot\|)$) to other normed vector spaces.
\begin{definition}[Norm degree]\label{def:norm_degree}
  Given a normed (or semi-normed) vector space $(E, \Vert \cdot \Vert)$, and a subset $H\subset E'$, the degree $\eta_H$ of $H$ is defined as~:
  \begin{itemize}
      \item $\eta_H \equiv \log(| H|)$ if $H$ is finite,
      \item $\eta_H \equiv \dimm(\vect H)$ if $H$ is infinite.
  \end{itemize}
  If there exists a subset $H \subset E'$ such that \eqref{eq:norm_egal_supremum} is satisfied, we denote $\eta(E, \Vert \cdot \Vert)$, or more simply $\eta_{\Vert \cdot \Vert}$, the degree of $\Vert \cdot \Vert$, defined as~:
  \begin{align*}
      \eta_{\Vert \cdot \Vert}=\eta(E, \Vert \cdot \Vert)\equiv\inf \left\{\eta_H, H\subset E' \ | \ \forall x \in E, \Vert x \Vert = \sup_{f \in H}f(x)\right\}.
  \end{align*}
\end{definition}


\begin{example}\label{exe:norm_degree}
    We can give some examples of norm degrees~:
    \begin{itemize}
        \item $\eta \left( \mathbb R^p, \Vert \cdot \Vert_\infty \right) = \log(p)$ ($H = \{x \mapsto e_i^Tx, 1\leq i\leq p\}$),
        \item $\eta \left( \mathbb R^p, \Vert \cdot \Vert \right) = p$ ($H = \{x \mapsto u^Tx, u \in \mathcal B_{\mathbb R^p}\}$),
        \item $\eta \left( \mathcal M_{p,n}, \Vert \cdot \Vert \right) = n+p$ ($H = \{M \mapsto u^TMv, (u,v)\in \mathcal B_{\mathbb R^p} \times \mathcal B_{\mathbb R^n}\}$),
        \item $\eta \left( \mathcal M_{p,n}, \Vert \cdot \Vert_F \right) = np$ ($H = \{M \mapsto \tr(AM), A \in \mathcal M_{n,p}, \|A\|_F\leq 1\}$),
        \item $\eta \left( \mathcal M_{p,n}, \Vert \cdot \Vert_* \right) = np$ ($H = \{M \mapsto \tr(AM), A \in \mathcal M_{n,p}, \|A\|\leq 1\}$).\footnote{$\|\cdot\|_*$ is the nuclear norm defined for any $M \in \mathcal M_{p,n}$ by $\|M\|_* = \tr(\sqrt{MM^T})$; it is the dual norm of $\| \cdot \|$, which means that for any $A,B \in \mathcal M_{p,n}$, $\tr(AB^T) \leq \|A\| \|B\|_*$. One must be careful that Proposition~\ref{pro:tao_conc_exp} is rarely useful to bound the nuclear norm as explained in footnote~\ref{foo:attention_control_norm_ev_euclidien}.}
    \end{itemize}
    Just to give some justification, if $E=\mathbb R^p$ or $E = \mathcal M_{p,n}$, the dual space $E'$ can be identified with $E$ through the representation with the scalar product. Given a subset $H' \subset E$ such that:
    \begin{align*}
        \forall x \in \mathbb R^p, \|x \|_\infty = \sup_{u\in H'} u^Tx,
    \end{align*}
    we can set that all $u \in H'$ satisfy $\|u\|_1 = \sum_{i=1}^p |u_i| \leq 1$ because if we note $u' = (\text{sign}(u_i))_{i\in [p]}$, we can bound $\|u\|_1 = u^Tu' \leq \sup_{v\in H'} v^Tu' \leq \|u'\|_\infty \leq 1$.
    Then, noting $H = \{e_1,\ldots, e_p\}$, we know that $H \subset H'$, otherwise, if, say $e_i \notin H'$, then one could bound $\|e_i \|_\infty = \sup_{u\in H'} u^Te_i <1$ (because if $\|u\|_1 \leq 1$ and $u \neq e_i$, then $u_i <1$).
    Therefore $H \subset H'$ and it consequently reaches the minimum of $\eta_{H'}$. The value of the other norm indexes is justified with the same arguments.
\end{example}
Depending on the ambient vector space, one can employ one of these examples along with the following proposition borrowed from \cite[Proposition~2.9,2.11, Corollary~2.13]{louart2018concentration} to establish the concentration of the norm of a random vector.
\begin{proposition}\label{pro:tao_conc_exp}
Given a reflexive vector space $(E, \Vert \cdot \Vert)$ and a concentrated vector $Z \in E$ satisfying $Z \in \tilde Z \pm \mathcal E_q(\sigma)$:
  \begin{align*}
       \Vert Z - \tilde Z \Vert \propto \mathcal E_{q}\left(\eta_{\Vert \cdot \Vert}^{1/q}\sigma\right)&
       &\text{and}&
       &\mathbb E \left[\Vert Z - \tilde Z \Vert\right] \leq O \left(\eta_{\Vert \cdot \Vert}^{1/q}\sigma\right).
  \end{align*}
  
\end{proposition}
\begin{remark}\label{rem:concentration_norme_hypo_lipschitz}
  In Proposition~\ref{pro:tao_conc_exp}, if $Z \propto \mathcal E_q(\sigma)$, the norm satisfies the same concentration as it is a Lipschitz observation, and one gets\footnote{The notation $Z \in O(\theta) \pm \mathcal E_q(\sigma)$ was presented in Definition~\ref{def:linear_concentration} for linearly concentrated vectors, it can be extended to concentrated random variables.}: 
  $$\Vert Z - \tilde Z \Vert \in O \left(\eta_{\Vert \cdot \Vert}^{1/q}\sigma\right) \pm \mathcal E_{q}\left(\sigma\right).$$
\end{remark}

\begin{example}\label{exe:borne_esp_norm_vecteur_lin_conc}
Given two random vectors $Z \in \mathbb{R}^p$ and $X \in \mathcal M_{p,n}$: 
\begin{itemize}
  \item if $Z\propto \mathcal E_2$ in $(\mathbb{R}^p,\left\Vert \cdot\right\Vert)$, then $\mathbb{E}\left\Vert Z\right\Vert\leq\Vert \mathbb E[Z] \Vert + O(\sqrt{p})$,
  \item if $X \propto \mathcal E_2$ in $(\mathcal M_{p,n}, \left\Vert \, \cdot \,\right\Vert)$, then 
  $\mathbb{E}\left\Vert X\right\Vert \leq \Vert \mathbb E[X]\Vert + O(\sqrt{p+n}),$
  \item if $X \propto \mathcal E_2$ in $(\mathcal M_{p,n}, \left\Vert \, \cdot \,\right\Vert_F)$, then
  $\mathbb{E}\left\Vert X\right\Vert \leq \Vert \mathbb E[X]\Vert_F + O(\sqrt{pn})$.
  \item
   if $X \propto \mathcal E_2(\sqrt{\min(p,n)})$ in $(\mathcal M_{p,n}, \left\Vert \, \cdot \,\right\Vert_*)$, then
  $\mathbb{E}\left\Vert X\right\Vert_* \leq \Vert \mathbb E[X]\Vert_* + O(\sqrt{pn}\sqrt{\min(p,n)})$.\footnote{One must be careful here that Theorem~\ref{the:concentration_vecteur_gaaussien} just provides concentration in the Euclidean spaces $(\mathbb R^p, \|\cdot \|)$ or $(\mathcal M_{p,n}, \|\cdot \|_F)$ from which one can deduce concentration in $(\mathbb R^p, \|\cdot \|_\infty)$ or $(\mathcal M_{p,n}, \|\cdot \|)$ since for all $x \in \mathbb R^p$, $\|x \|_\infty\leq \| x \|$ and for all $M \in \mathcal M_{p,n}$, $\|M\|\leq \|M\|_F$. However one cannot obtain a better bound than $\|M\|_*\leq \sqrt{\min(n,p)}\|M\|_F$: this for instance implies that a random matrix $X = (x_1,\ldots, x_n)$ with $x_1,\ldots, x_n$ i.i.d.\@ satisfying $\forall i\in[n]$, $x_i \sim \mathcal N(0, I_p)$ follows the concentration $X \propto \mathcal E_2(\sqrt{\min(p,n)})$ in $(\mathcal M_{p,n}, \|\cdot \|_*)$.\label{foo:attention_control_norm_ev_euclidien}}
\end{itemize}
\end{example}
Let us consider the semi norm $\|\cdot \|_d$ that will be useful later and that satisfies:
\begin{definition}\label{def:diagonal_norm_d}
  Given $M \in \mathcal M_{n}$, we define the diagonal norm of $M$ as:
  \begin{align*}
    \|M\|_d = \left(\sum_{i=1}^n M_{i,i}^2\right)^{\frac{1}{2}} = \sup_{\genfrac{}{}{0pt}{2}{D \in \mathcal D_n}{\|D\|_F \leq 1} }\tr(DM).
  \end{align*}  
  where $D_{p,n}$ is the set of diagonal matrices of $\mathcal M_{p,n}$ defined as:
  \begin{align*}
    D\in D_{p,n} \Leftrightarrow (i\neq j \implies D_{i,j} = 0).
  \end{align*}
  For simplicity, we note the (non zero) diagonal terms of any $D \in \mathcal D_{n,p}$: $D_1,\ldots, D_{\min(n,p)}$.
\end{definition}
\begin{example}\label{exe:semi_norme_diag}
  We see directly that $\eta_{(\mathcal M_{n}, \|\cdot\|_d)} = \# \mathcal D_n =  n$ and therefore for a given $X \in \mathcal M_{n}$ such that $X \propto \mathcal E_2$, we can bound $\mathbb E\|X\|_d \leq \|\mathbb E [X]\|_d + O(\sqrt n)$.
\end{example}


 Proposition~\ref{pro:tao_conc_exp} is not always the optimal way to bound norms. 
For instance, given a vector $Z\in \mathbb R^p$ and a deterministic matrix $A \in \mathcal M_{p}$, if $Z \propto \mathcal E_q$, one is tempted to bound naively thanks to Proposition~\ref{pro:tao_conc_exp}:
\begin{itemize}
  \item if $\|\mathbb E[Z]\| \leq O( p^{1/q})$, $\mathbb E[\|AZ\|]\leq \|A\|\mathbb E[\|Z\|] \leq O(\|A\|p^{\frac{1}{q}})$; 
  \item if $\|\mathbb E[Z]\| \leq O( 1)$, decomposing $A=P^T\Lambda Q$, where $P,Q \in \mathcal O_p$, $\Lambda = \diag(\lambda)$, $\lambda = (\lambda_1,\ldots, \lambda_p) \in \mathbb R^p$ and setting $\check Z = (\check Z_1,\ldots, \check Z_p) \equiv QZ$:
  \begin{align*}
     \mathbb E[\|AZ\|] = \mathbb E[\|\Lambda QZ \|] = \mathbb E \left[\sqrt{\sum_{i=1}^p \lambda_i^2 \check Z_i^2}\right] \leq \|\lambda\| \mathbb E \left[\|\check Z \|_\infty\right] \leq \|A\|_F O \left((\log p)^{\frac{1}{q}}\right).
   \end{align*} 
   Note indeed that $\check Z \propto \mathcal E_2$ and therefore $\mathbb E[\|\check Z\|_\infty] \leq \|\mathbb E[\check Z]\|_\infty + O \left((\log p)^{\frac{1}{q}}\right)\leq \|\mathbb E[Z]\| + O \left((\log p)^{\frac{1}{q}}\right)$.
\end{itemize}
However, here, Proposition~\ref{pro:tao_conc_exp} is suboptimal: one can reach a better bound thanks to the following lemma that was taken from the proof of \cite[Theorem 2.5]{ADA14}. We give a result for random vectors and random matrices, they are actually equivalent.
\begin{lemma}\label{lem:borne_Ax}
  Given a random vector  in $(\mathbb R^p, \|\cdot \|)$ such that $\|\mathbb E[ZZ^T]\| \leq O(1)$ and a deterministic matrix $A \in \mathcal M_{p}$:
  \begin{align*}
    \mathbb E[\|AZ\|]\leq O \left(\|A\|_F\right).
  \end{align*}
  and given a random matrix $X \in \mathcal E_2$ in $(\mathcal M_{p,n}, \| \cdot \|_F)$ such that $\|\mathbb E[X]\|_F \leq O(1)$ and a supplementary deterministic matrix $B \in \mathcal M_{n}$:
  \begin{align*}
    \mathbb E[\|AXB\|_F]\leq O \left(\|A\|_F\|B\|_F\right).
  \end{align*}
\end{lemma}
Thanks to \ref{pro:carcaterisation_vecteur_linearirement_concentre_avec_moments}, we know that the condition $\|\mathbb E[ZZ^T]\|\leq O(1)$ can be obtained if one assumes that $Z \in \mathcal E_q$ and $\|\mathbb E[Z]\| \leq O(1)$ thanks to the inequality $\mathbb E[ZZ^T] = C^Z_2 + \mathbb E[Z] \mathbb E[Z]^T$, where $C_Z$ is the covariance of $Z$. The same way, 
\begin{proof}
One can bound with Jensen's inequality:
\begin{align*}
  \mathbb E[\|AZ\|] \leq \sqrt{\mathbb E[Z^TA^TAZ]} = \sqrt{\mathbb E[\tr(\Sigma A^TA)]} \leq \sqrt{\left\Vert \Sigma\right\Vert} \|A\|_F \leq O(\|A\|_F).
\end{align*}
The second result is basically the same. If we introduce $\check X \in \mathbb R^{pn}$ satisfying $\check X_{i(j-1)+j} = X_{i,j}$, we know that $\tilde X \propto \mathcal E_2$ like $X$ (since $\|\tilde X\| = \|X\|_F$) and thanks to the previous result we can bound:
\begin{align*}
  \mathbb E[\|AXB\|_F]=\mathbb E[\|A\otimes B \tilde X\|]\leq O \left(\|A \otimes B\|_F\right) = O \left(\|A\|_F\|B\|_F\right).
\end{align*}
\end{proof}

\medskip

Returning to Lipschitz concentration, in order to control the concentration of the sum $X+Y$ or the product $XY$ of two random vectors $X$ and $Y$, a first step is to express the concentration of the concatenation $(X,Y)$. This last result is easily obtained for the class of linearly concentrated random vectors but a tight concentration of the product with good observable diameter is in general not accessible. 
In the class of Lipschitz concentrated vectors, the concentration of $(X,Y)$ is far more involved, and assumptions of independence here play a central role (unlike for linear concentration). 

\medskip

The sum being a $2$-Lipschitz operation (for the norm $\|\cdot \|_{\ell^\infty}$), the concentration of $X+Y$ is easily handled with Proposition~\ref{pro:stabilite_lipschitz_mappings} and directly follows from the concentration of $(X,Y)$. For products of vectors, more work is required.

\section{Multi regime concentration of expression with concentrated variations}
Given four parameters $\sigma_1,\sigma2,q_1,q_2>0$, such that $\sigma_1\leq \sigma_2$ and $q_1\leq q_2$, we have the equivalence:
\begin{align}\label{eq:definition_t_sigma_q}
  e^{-c(t/\sigma_1)^{q_1}}\geq e^{-c(t/\sigma_2)^{q_2}}&
  &\Longleftrightarrow&
  & t\leq \left( \frac{\sigma_1^{q_1}}{\sigma_2^{q_2}} \right)^{\frac{1}{q_1-q_2}},
\end{align}
one then introduce naturally the threshold $t_{\sigma_1, q_1}^{\sigma_2, q_2} \equiv \left( \frac{\sigma_1^{q_1}}{\sigma_2^{q_2}} \right)^{\frac{1}{q_1-q_2}}$ that will allow us to distinguish the different regimes of exponential concentration and see is some of them can be removed when they are irrelevant for all $t>0$.
 \begin{definition}\label{def:canonical_multiregime parameters}
   Given $m \in \mathbb N_*$, a set of $m$ couples of parameters $(q_1,\sigma_1), \ldots, (q_m,\sigma_m) \in (\mathbb R^2)^m$ is called a family of multi-regime parameters if:
   \begin{itemize}
     \item $q_1>\cdots >q_m>0$
     \item $\sigma_1>\cdots >\sigma_m>0$.
   \end{itemize}
This family is said to be thrifty if $m\leq 2$ or $\forall 1 \leq k<l\leq m$:
\begin{align*}
  t_{\sigma_k,q_k}^{\sigma_{k+1},q_{k+1}} \leq t_{\sigma_k,q_k}^{\sigma_{l},q_{l}} \leq t_{\sigma_{l-1},q_{l-1}}^{\sigma_{l},q_{l}},&
  &\text{where we noted:}&
    &t_{\sigma_k,q_k}^{\sigma_{l},q_{l}} \equiv \left( \frac{\sigma_{k}^{q_k}}{\sigma_{l}^{q_l}} \right)^{\frac{1}{q_k-q_l}}.
  \end{align*}
   It is called a free family of multi-regime parameters if in addition for all $1\leq k<l\leq m-1$, $k< l$:
   \begin{align*}
  t_{\sigma_k,q_k}^{\sigma_{k+1},q_{k+1}} < t_{\sigma_l,q_l}^{\sigma_{l+1},q_{l+1}}.
  \end{align*}
 \end{definition}
This definition is justified by the following proposition that explains how one can remove some regimes of concentration when they are always under other regimes. One first need a preliminary lemma that we provide without proof since it is trivial.

\begin{lemma}\label{lem:time_limit_relation}
  Given three couples $(\sigma_{1},q_{1}),(\sigma_{2},q_{2}),(\sigma_{3},q_{3}) \in \mathbb R_+^2$, such that $\sigma_1\leq \sigma_2\leq \sigma_3$ and $q_1\leq q_2\leq q_3$, if we note for any $l,k \in [3]$, $l\neq k$, $t_{k, l}\equiv t_{\sigma_k,q_k}^{\sigma_{l},q_{l}} = (\sigma_k^{q_k}/\sigma_l^{q_l})^{1/q_k - q_l}$, then, we have the relation:
  \begin{align*}
    t_{1,2}^{q_1 - q_2}t_{2,3}^{q_2 - q_3} = t_{1,3}^{q_1 - q_3}.
  \end{align*}
\end{lemma}
Now, we can set the following proposition:

\begin{proposition}\label{pro:regime_cleaning}
  In the setting and with the notations of Lemma~\ref{lem:time_limit_relation} one has the equivalence:
  \begin{align*}
     t_{1,2}\geq t_{1,3}&
     &\Longleftrightarrow&
     &t_{1,3} \geq t_{2,3}
   \end{align*}
  and if one assumes the two assertions of this equivalence, one can bound:
  \begin{align*}
    \forall t>0 \quad: e^{-(t/\sigma_2)^{q_2}} \leq \sup(e^{-(t/\sigma_1)^{q_1}}, e^{-(t/\sigma_{3})^{q_3}}).
  \end{align*}
  In other words, $\mathcal E_{q_2}(\sigma_2) \leq \mathcal E_{q_1}(\sigma_1) + \mathcal E_{q_3}(\sigma_3)$.
\end{proposition}
\begin{proof}
  The equivalence is just a consequence of Lemma~\ref{lem:time_limit_relation}. If we now assume that the two assertion of the equivalence are true, one can bound when $t \leq t_{1,3}$ thanks to \eqref{eq:definition_t_sigma_q} and the fact that $t_{1,3}\leq t_{1,2}$:
  \begin{align*}
    e^{-c(t/\sigma_2)^{q_2}} \leq e^{-c(t/\sigma_1)^{q_1}}.
  \end{align*}
  The same way, when $t \geq t_{1,3}\geq t_{2,3}$:
  \begin{align*}
    e^{-c(t/\sigma_2)^{q_2}} \leq e^{-c(t/\sigma_3)^{q_3}},
  \end{align*}
  we thus exactly showed that $\mathcal E_{q_2}(\sigma_2) \leq \mathcal E_{q_1}(\sigma_1) + \mathcal E_{q_3}(\sigma_3)$.
\end{proof}

An important local characterization of thrifty multi-regime families is given in next proposition.
\begin{proposition}\label{pro:identity_multiregime_parameters}
  Given $m\geq 3$, a family of multi regime parameters $(q_1,\sigma_1), \ldots, (q_m,\sigma_m) \in (\mathbb R^2)^m$ is thrifty if and only if one of the following property is satisfied:
  \begin{itemize}
    \item for any tuple of integers $(i,j,k,l) \in [m]^4$ such that $i \leq j \leq k \leq l$, we have the inequality:
  \begin{align}\label{eq:relation_sensible_multiregime_1}
    \left( \frac{\sigma_i}{\sigma_j} \right)^{\frac{q_iq_j}{q_{i} - q_{j}}}
    \leq \left( \frac{\sigma_k}{\sigma_l} \right)^{\frac{q_kq_l}{q_{k} - q_{l}}}
  \end{align}
    \item for any $i \in [m-2]$:
    \begin{align}\label{pro:charct_sensible_3}
    \left( \frac{\sigma_i}{\sigma_{i+1}} \right)^{\frac{q_iq_{i+1}}{q_i - q_{i+1}}} \leq \left( \frac{\sigma_{i+1}}{\sigma_{i+2}} \right)^{\frac{q_{i+1}q_{i+2}}{q_{i+1} - q_{i+2}}},
  \end{align}
  \end{itemize}
\end{proposition}

\begin{proof}
  Let us start with the following equivalence:
  \begin{align}\label{eq:expession_flat_sensible}
  t_{\sigma_i,q_i}^{\sigma_{j},q_{j}} \leq t_{\sigma_j,q_j}^{\sigma_{k},q_{k}}
  &\Longleftrightarrow
  \left( \frac{\sigma_{i}^{q_i}}{\sigma_{j}^{q_{j}}} \right)^{q_j-q_{k}} \leq \left( \frac{\sigma_{j}^{q_j}}{\sigma_{k}^{q_{k}}} \right)^{q_i-q_{j}}\nonumber\\
  &\Longleftrightarrow
  \sigma_i^{\frac{1}{q_{k}} - \frac{1}{q_{j}}} \sigma_{k}^{\frac{1}{q_{j}} - \frac{1}{q_{i}}} \leq \sigma_j^{\frac{1}{q_{k}} - \frac{1}{q_{i}}}\\
  &\Longleftrightarrow \left( \frac{\sigma_i}{\sigma_j} \right)^{\frac{1}{q_{k}} - \frac{1}{q_{j}}}  \leq \left( \frac{\sigma_j}{\sigma_k} \right)^{\frac{1}{q_{j}} - \frac{1}{q_{i}}}
  &\Longleftrightarrow
  \left( \frac{\sigma_i}{\sigma_j} \right)^{\frac{q_iq_j}{q_{i} - q_{j}}}  \leq \left( \frac{\sigma_j}{\sigma_k} \right)^{\frac{q_jq_k}{q_{j} - q_{k}}}&\nonumber
  \end{align}
  (the last equivalence is obtained putting the inequality to the power $\frac{q_jq_k}{q_j-q_k} \frac{q_iq_j}{q_i-q_j}\geq 0$, we have indeed $\frac{\sigma_i}{\sigma_j} \geq 1$).
  Now, given a supplementary integer $l\geq k$, assuming $t_{\sigma_i,q_i}^{\sigma_{j},q_{j}} \leq t_{\sigma_j,q_j}^{\sigma_{k},q_{k}} \leq t_{\sigma_k,q_k}^{\sigma_{l},q_{l}}$, we can conclude thanks to the previous equivalence:
  \begin{align*}
      \left( \frac{\sigma_i}{\sigma_j} \right)^{\frac{q_iq_j}{q_{i} - q_{j}}}  \leq \left( \frac{\sigma_j}{\sigma_k} \right)^{\frac{q_jq_k}{q_{j} - q_{k}}} \leq \left( \frac{\sigma_k}{\sigma_l} \right)^{\frac{q_kq_l}{q_{k} - q_{l}}}.
  \end{align*}
  The reverse implication is simply obtained taking $k=j$.

  Since relation~\eqref{pro:charct_sensible_3} is a particular case of relation~\eqref{eq:relation_sensible_multiregime_1}, we just show the reverse implication. The proof is done iteratively on $k-i$ with the characterisation given in \eqref{eq:expession_flat_sensible}. When $k-i=2$, we are in the case of relation~\eqref{pro:charct_sensible_3}, given an integer $n \in \{2,\ldots, m-1\}$, we then assume that relation~\eqref{eq:expession_flat_sensible} is true when $k-i\leq n$, let us then consider a case where $k-i = n+1$ and, say, $k-j\geq 2$, employing $\sigma_{j}$ and $\sigma_{j+1}$ as a pivot ($j+1-i, k-j \leq n$) we bound:
  \begin{align*}
    \sigma_i^{\frac{1}{q_{k}} - \frac{1}{q_{j}}} \sigma_{k}^{\frac{1}{q_{j}} - \frac{1}{q_{i}}}
    &\leq \frac{\sigma_i^{\frac{1}{q_{k}} - \frac{1}{q_{j}}} \left(\sigma_j^{\frac{1}{q_{k}} - \frac{1}{q_{j+1}}} \sigma_{k}^{\frac{1}{q_{j+1}} - \frac{1}{q_{j}}} \right)^{\left( \frac{1}{q_{j}} - \frac{1}{q_{i}} \right) / \left( \frac{1}{q_{j+1}} - \frac{1}{q_{j}} \right)  } }{\sigma_j^{\left( \frac{1}{q_{k}} - \frac{1}{q_{j+1}} \right)\left( \frac{1}{q_{j}} - \frac{1}{q_{i}} \right) / \left( \frac{1}{q_{j+1}} - \frac{1}{q_{j}} \right)  } }\\
    &\leq \frac{\sigma_i^{\frac{1}{q_{k}} - \frac{1}{q_{j}}} \left(\sigma_{j+1}^{\frac{1}{q_{k}} - \frac{1}{q_{j}}} \right)^{\left( \frac{1}{q_{j}} - \frac{1}{q_{i}} \right) / \left( \frac{1}{q_{j+1}} - \frac{1}{q_{j}} \right)  } }{\sigma_j^{\left( \frac{1}{q_{k}} - \frac{1}{q_{j+1}} \right)\left( \frac{1}{q_{j}} - \frac{1}{q_{i}} \right) / \left( \frac{1}{q_{j+1}} - \frac{1}{q_{j}} \right)  } }\\
    &\leq \frac{\left(\sigma_i^{\frac{1}{q_{j+1}} - \frac{1}{q_{j}}} \sigma_{j+1}^{\frac{1}{q_{j}} - \frac{1}{q_{i}}} \right)^{\left( \frac{1}{q_{k}} - \frac{1}{q_{j}} \right) / \left( \frac{1}{q_{j+1}} - \frac{1}{q_{j}} \right)  } }{\sigma_j^{\left( \frac{1}{q_{k}} - \frac{1}{q_{j+1}} \right)\left( \frac{1}{q_{j}} - \frac{1}{q_{i}} \right) / \left( \frac{1}{q_{j+1}} - \frac{1}{q_{j}} \right)  } }
    \ \ \leq \frac{\sigma_j^{\left(\frac{1}{q_{j+1}} - \frac{1}{q_{i}}  \right)\left( \frac{1}{q_{k}} - \frac{1}{q_{j}} \right) / \left( \frac{1}{q_{j+1}} - \frac{1}{q_{j}} \right)  } }{\sigma_j^{\left( \frac{1}{q_{k}} - \frac{1}{q_{j+1}} \right)\left( \frac{1}{q_{j}} - \frac{1}{q_{i}} \right) / \left( \frac{1}{q_{j+1}} - \frac{1}{q_{j}} \right)  } }.
  \end{align*}
  One can then conclude thanks to the identity:
  \begin{align*}
    \left(\frac{1}{q_{j+1}} - \frac{1}{q_{i}}  \right)\left( \frac{1}{q_{k}} - \frac{1}{q_{j}} \right) - \left( \frac{1}{q_{k}} - \frac{1}{q_{j+1}} \right)\left( \frac{1}{q_{j}} - \frac{1}{q_{i}} \right) = \left( \frac{1}{q_{j+1}} - \frac{1}{q_{j}} \right)\left( \frac{1}{q_{k}} - \frac{1}{q_{i}} \right)
  \end{align*}

\end{proof}
This last proposition allows to extract thrifty multi-regime families of multi-regime parameters from non thrifty ones.
\begin{proposition}\label{pro:unique_free_family}
  Given $m\geq 3$ and a family of multi regime parameters $(q_1,\sigma_1), \ldots, (q_m,\sigma_m) \in (\mathbb R_+^2)^m$, there exists a unique free thrifty extraction of this family $(q_{i_1},\sigma_{i_1}), \ldots, (q_{i_k},\sigma_{i_k}) \in (\mathbb R_+^2)^k$ for $k\leq m$ such that $1 \leq i_1 <\cdots <i_k = m$ and for any $t >0$:
  \begin{align}\label{eq:regime_inequality_extraction}
    \sup_{i\in [m]} \mathcal E_{q_i}(\sigma_i) = \sup_{j\in [k]} \mathcal E_{q_{i_j}}(\sigma_{i_j}).
  \end{align}
\end{proposition}
\begin{proof}
  For all $k,l\in[m]$, we employ again the short notation $t_{k, l}\equiv t_{\sigma_k,q_k}^{\sigma_{l},q_{l}}$. The extraction can be done iteratively. The first component of our iterative extraction is $(\sigma_{i_1}, q_{i_1}) = (\sigma_{1}, q_{1})$. Now, given $2\leq l\leq m-1$ and assuming that $(q_{i_1},\sigma_{i_1}), \ldots, (q_{i_l},\sigma_{i_l})$ is thrifty, we introduce the set:
  \begin{align*}
    A_{i_l} \equiv \left\{ i \in \{l+1,\ldots, m\} \ s.t. \ t_{i_{l-1}, i_{l}} < t_{i_{l}, i} \ \right\},
  \end{align*}
  and if $A_{i_l}$ is not empty, we denote $i_{l+1} = \inf A_{i_l}$. We know from Proposition~\ref{pro:identity_multiregime_parameters} that $(q_{i_ {1} },\sigma_{i_ {1} }), \ldots, (q_{i_ {l+1} },\sigma_{i_ {l+1} })$ is thrifty, we can thus repeat the procedure until we found $i_{l_0} \in [m]$ such that such that $ A_{i_{l_0}}$ is empty (we know that $A_{i_2} \supset\cdots \supset A_{i_l} \supset A_{1} \supset \cdots A_{m} = \emptyset$), then, noting $k = l_0 + 1$, and $i_k = m$, by construction, $(q_{i_ {1} },\sigma_{i_ {1} }) \cdots (q_{i_ {k} },\sigma_{i_ {k} })$ is a thrifty family.


  To prove \eqref{eq:regime_inequality_extraction}, we consider $l\in [k-1]$, and $i\in[m]$ such that $i_l\leq i \leq i_{l+1}$ we know from the definition of $A_{i_1} = \emptyset$ that $t_{i_l,i} \geq t_{i_{l}, i_{l+1}}$, therefore we know from Proposition~\ref{pro:regime_cleaning} that we also have the inequality $t_{i,i_{l+1}} \leq t_{i_{l}, i_{l+1}}$ and one can bound:
  \begin{align*}
    \mathcal E_{q_i}(\sigma_i) \leq \mathcal E_{q_{i_l}}(\sigma_{i_l}) + \mathcal E_{q_{i_{l+1}}}(\sigma_{i_{l+1}}),
  \end{align*}
  which ends our proof.
\end{proof}

Let us now provide 
We can then adapt Proposition~\ref{pro:identity_multiregime_parameters} to the case where $\forall i \in [m]$, $q_i = \frac{q}{i}$, for some $q>0$, to obtain the following characterization.
\begin{theorem}\label{the:concentration_concentrated_variations}
  Given two (sequence of) normed or (semi normed) vector spaces $(\mathbb E, \|\cdot\|)$ and $(F, \|\cdot \|)$, a (sequence of) random vectors $Z \in (\mathbb E, \|\cdot\|)$ such that $Z \propto \mathcal E_{q}(\sigma)$ and a mapping $\Phi:E \to F$ satisfying on the set of drawings of $Z$:
  \begin{align*}
   \left\Vert \phi(Z) -  \phi(Z')\right\Vert
  \leq V\left\Vert Z - Z'\right\Vert,
  \end{align*}
  for $Z'\in E$, an independent copy of $Z$ and $V$ a random variable satisfying:
  \begin{align*}
    V \in O(\sigma_0) \pm \sum_{l=1}^m
       \mathcal E_{q_l}\left(\sigma_{l}\right),
  \end{align*}
  for some multi-regime family $(\sigma_1, q_1),\ldots, (\sigma_{m}, q_{m})$ and $\sigma_0\in (0, \sigma_1]$.
  Then one can deduce the concentration:
   \begin{align*}
     \phi \left(Z\right) \propto \mathcal E_{q}\left(\sigma_{0}\sigma\right) + \sum_{i = 1}^m \mathcal E_{\frac{q_lq}{q_l + q}}\left(\sigma_{l}\sigma\right)
   \end{align*}
\end{theorem}
\begin{proof}
  Let us prove the theorem iteratively on the number of regimes describing the concentration of $V$. In a first time,we consider that the multi-regime family $(\sigma_1,q_1),\ldots, (\sigma_m,q_m)$ is thrifty.
  To prove the initializing point we assume that $V \in O(\sigma_0) \pm \mathcal E_{q_1}(\sigma_1)$. Then introducing:
  \begin{align*}
    K_t = \max \left( \sigma_0, \sigma_1\left( \frac{ t}{\sigma\sigma_1} \right)^{\frac{q}{q_1+q}} \right), 
  \end{align*}
  we know that $K_t \geq \sigma_0$ and we can first bound:
  \begin{align*}
    \mathbb P \left( V \geq 2K_t \right) 
    \leq \mathbb P \left( V \geq 2K_t \right) 
    \leq \mathbb P \left(  \left\vert V - \sigma_0 \right\vert \geq K_t \right) 
    \leq  C e^{- c\left( t/ \sigma_1K_t \right)^{q_1}}
    \leq  C e^{- c\left( t/ \sigma_1\sigma \right)^{\frac{qq_1}{q_1+q}}}.
  \end{align*}
  Besides, under $\{V \leq K_t\}$, $\phi$ is $K_t$-Lipschitz, and one can bound:
  \begin{align*}
    \mathbb P \left( \left\vert \Phi(Z) - \Phi(Z') \right\vert\geq t, V\leq 2K_t \right)
    &\leq C e^{- c\left( t/\sigma K_t \right)^q}
    \leq \max \left( C e^{- c\left( t/\sigma \sigma_0 \right)^q}, C e^{- c\left( t/\sigma \sigma_1 \right)^{\frac{qq_1}{q_1+q}}} \right).
  \end{align*}
  Finally we are able to bound:
  \begin{align*}
    \mathbb P \left( \left\vert \Phi(Z) - \Phi(Z') \right\vert\geq t\right) \leq \max \left( C e^{- c\left( t/\sigma \sigma_0 \right)^q}, C e^{- c\left( t/\sigma \sigma_1 \right)^{\frac{qq_1}{q_1+q}}} \right).
  \end{align*}

  Let us now assume that our theorem is true for any $m\leq M-1$ and that we have the concentration:
  \begin{align*}
    \quad V \in O(\sigma_0) \pm \sum_{l\in [M]} \mathcal E_{q_l}(\sigma_l).
  \end{align*}
  With the notation $t_{i-1}\equiv t_{\sigma_{i-1},q_{i-1}}^{\sigma_{i},q_{i}} = (\frac{\sigma_{i-1}^{q_{i-1}}}{\sigma_i^{q_i}})^{\frac{1}{q_{i-1}-q_i}}$, we can bound (since the family is thrifty):
  \begin{align*}
    \forall t >0: \quad \mathbb P \left( V \geq t \ | \ V \leq t_{M-1} \right) 
    \leq \sup_{l \in [M-1]} C e^{-c (t/\sigma_{l})^{q_l}},
  \end{align*}
  and our iteration hypothesis allows us to write:
  \begin{align*}
    \forall t >0: \quad
    \mathbb P \left( \left\vert \Phi(Z) - \Phi(Z') \right\vert\geq t \ | \ V \leq t_{M-1} \right) 
    \leq \sup_{l \in [M-1]} C e^{-c (t/\sigma \sigma_l)^{q_l'}}.
  \end{align*}
  When $V \geq t_{M-1}$, we introduce the parameter
  \begin{align*}
    K_t = \max \left( \sigma_0, \sigma_{M}\left( \frac{ t}{\sigma\sigma_{M}} \right)^{\frac{q}{q_M+q}} \right),
  \end{align*}
  it satisfies:
  \begin{align*}
    \mathbb P \left( V \geq K_t \ | \ V \geq t_{M-1} \right) 
    \leq  C e^{-c (t/\sigma\sigma_{M})^{q_{M}'}}
  \end{align*}
  and:
  \begin{align*}
    \mathbb P \left( \left\vert \Phi(Z) - \Phi(Z') \right\vert\geq t, t_{M-1} \leq V \leq K_t \right)
    \leq \max \left( C e^{- c\left( t/m\sigma \sigma_0 \right)^q}, C e^{- c\left( t/m^{q'_m/q}\sigma \sigma_M \right)^{q_M'}} \right).
  \end{align*}

  When the family $(\sigma_1,q_1),\ldots, (\sigma_m, q_m)$ is not thrifty, one can still consider its free thrifty extraction $(\sigma_{i_1},q_{i_1}),\ldots, (\sigma_{i_k}, q_{i_k})$ thanks to Proposition~\ref{pro:unique_free_family} and then show:
  \begin{align*}
    \Phi(Z)\propto \mathcal E_q(\sigma\sigma_0) +  \sum_{l=1}^k \mathcal E_{\frac{qq_{i_l}}{q + q_{i_l}}} \left( \sigma \sigma_{i_l} \right)\propto \mathcal E_q(\sigma\sigma_0) + \sum_{i=1}^m \mathcal E_{\frac{qq_{i}}{q + q_{i}}} \left( \sigma \sigma_{i} \right)
  \end{align*}
\end{proof}
It is possible to optimize the expression of the concentration in Theorem~\ref{the:concentration_concentrated_variations} if one uses some stability properties provided by the following proposition.
some stability properties of the class of thrifty multi-regime families.

\begin{proposition}\label{pro:stability_thrifty}
  Given a positive parameter $\alpha>0$, a sequence of decreasing positive parameters $q_1>\cdots>q_m>0$, given $\sigma_1>\cdots>\sigma_m>0$ and $\sigma'_1>\cdots>\sigma'_m>0$ such that $(\sigma_1, q_1), \ldots, (\sigma_m, q_m)$ and $(\sigma'_1, q_1), \ldots, (\sigma'_m, q_m)$ are two thrifty multi-regime family, the family $((\sigma_1\sigma'_1)^\alpha, q_1), \ldots, ((\sigma_m\sigma'_m)^\alpha, q_m)$ is also thrifty.

  The same way, given a sequence of decreasing positive parameters $\sigma_1>\cdots>\sigma_m>0$, given $q_1>\cdots>q_m>0$ and $q'_1>\cdots>q'_m>0$ such that $(\sigma_1, q_1'), \ldots, (\sigma_m, q_m')$ and $(\sigma_1, q'_1), \ldots, (\sigma_m, q'_m)$ are two thrifty multi-regime family, the family $(\sigma_1, \alpha\frac{q_1q_1'}{q_1 + q_1'}), \ldots, (\sigma_m, \alpha\frac{q_mq_m'}{q_m + q_m'})$ is also thrifty.
\end{proposition}
\begin{remark}\label{rem:conic_interprestation_of_multi_regime}
  To provide a geometric interpretation of the last proposition, note that the set of families of positive parameters $\mu_1>\cdots>\mu_m$ such that $(e^{\mu_1}, q_1), \ldots, (e^{\mu_m}, q_m)$ is a thrifty multi-regime family is a cone of $\mathbb R^m$. The same way, the set of families of positive parameters $s_1>\cdots>s_m>0$ such that $(\sigma_1, \frac{1}{s_1}), \ldots, (\sigma_m, \frac{1}{s_m})$ is a thrifty multi-regime family is a cone of $\mathbb R^m_+$.
\end{remark}
\begin{proof}
  For all $i\in [m]$, we note:
  \begin{align*}
    \mu_i^{(}{}'{}^{)} = \log(\sigma_i^{(}{}'{}^{)})&
    &\text{and}&
    &s_i^{(}{}'{}^{)} = \frac{1}{q_i{}^{(}{}'{}^{)}}.
  \end{align*}
  It is then easy to see that for any $1<i<m$:
  \begin{align*}
    \left( \frac{\sigma_{i-1} }{\sigma_{i} } \right)^{\frac{q_{i-1}q_{i}}{q_{i-1} - q_{i}}}\leq \left( \frac{\sigma_i }{\sigma_{i+1} } \right)^{\frac{q_iq_{i+1}}{q_i - q_{i+1}}}&
    &\Longleftrightarrow&
    & \frac{\mu_{i-1} - \mu_{i}}{s_{i-1} - s_{i}}\leq \frac{\mu_{i} - \mu_{i+1}}{s_{i} - s_{i+1}},
   \end{align*}
   and the same identity hold for $\mu_{i-1}', \mu_{i}', \mu_{i+1}', s_{i-1}, s_{i}, s_{i+1}$ and $\mu_{i-1}, \mu_{i}, \mu_{i+1}, s'_{i-1}, s'_{i}, s'_{i+1}$.
   One can then conclude with the implications:
   \begin{align*}
     \left\{\begin{aligned}
       &\frac{\mu_{i-1} - \mu_{i}}{s_{i-1} - s_{i}}\leq \frac{\mu_{i} - \mu_{i+1}}{s_{i} - s_{i+1}}\\
       &\frac{\mu'_{i-1} - \mu'_{i}}{s_{i-1} - s_{i}}\leq \frac{\mu'_{i} - \mu'_{i+1}}{s_{i} - s_{i+1}}
     \end{aligned}\right.
     &\Longrightarrow
      \frac{\mu_{i-1} + \mu'_{i-1} - (\mu_{i} + \mu'_{i})}{s_{i-1} - s_{i}}\leq \frac{\mu_{i} + \mu_{i} - (\mu_{i+1} + \mu'_{i+1})}{s_{i} - s_{i+1}}\\
     &\Longrightarrow
     \left( \frac{\sigma_{i-1}\sigma'_{i-1} }{\sigma_{i}\sigma'_{i} } \right)^{\frac{1}{\frac{1}{q_{i-1}} - \frac{1}{q_{i}}}}\leq \left( \frac{(\sigma_i\sigma'_i) }{\sigma_{i+1}\sigma'_{i+1}} \right)^{\frac{1}{\frac{1}{q_i} - \frac{1}{q_{i+1}}}}\\
     \left\{\begin{aligned}
       &\frac{\mu_{i-1} - \mu_{i}}{s_{i-1} - s_{i}}\leq \frac{\mu_{i} - \mu_{i+1}}{s_{i} - s_{i+1}}\\
       &\frac{\mu_{i-1} - \mu_{i}}{s'_{i-1} - s'_{i}}\leq \frac{\mu_{i} - \mu_{i+1}}{s'_{i} - s'_{i+1}}
     \end{aligned}\right.
     &\Longrightarrow
      \frac{\mu_{i-1} - \mu_{i}}{s_{i-1} + s'_{i-1} - (s_{i} + s'_{i})}\leq \frac{\mu_{i} - \mu_{i+1}}{s_{i} + s'_{i} - (s_{i+1} + s'_{i+1})}\\
     &\Longrightarrow
      \left( \frac{\sigma_{i-1} }{\sigma_{i} } \right)^{\frac{1}{q_{i-1}} + \frac{1}{q'_{i-1}} - \left( \frac{1}{q_{i}} + \frac{1}{q'_{i}} \right)}\leq \left( \frac{\sigma_i }{\sigma_{i+1} } \right)^{\frac{1}{q_{i}} + \frac{1}{q'_{i}} - \left( \frac{1}{q_{i+1}}+ \frac{1}{q'_{i+1}} \right)}
   \end{align*}
\end{proof}
Taking into account those stability properties of the class of thrifty multi-regime families, one can obtain the following offshoot of Theorem~\ref{the:concentration_concentrated_variations}.
\begin{corollary}\label{cor:concentration_produit}
  In the setting of Theorem~\ref{the:concentration_concentrated_variations}, if we further assume that $(q_1,\sigma_1), \ldots, (q_m,\sigma_m)$ is thrifty, and we note $l_0 = \inf\{l \in [m] \ | \  \sigma_0 \leq t_{\sigma_l,q_l}^{\sigma_{l+1},q_{l+1}}\}$, then $(q,\sigma\sigma_0), (\frac{qq_{l_0}}{q + q_{l_0}},\sigma\sigma_{l_0}),\ldots, (\frac{qq_{m}}{q + q_{m}},\sigma\sigma_{m})$ is also thrifty and one has the concentration:
  \begin{align*}
    \Phi(Z) \propto \mathcal E_q(\sigma \sigma_0) + \sum_{l=l_0}^m \mathcal E_{\frac{q_lq}{q_l + q}}\left(\sigma_{l-1}\sigma\right)
  \end{align*}
\end{corollary}
\begin{proof}
  For simplicity, we introduce for all $i\in[m]$ the notation:
  \begin{align*}
    \sigma_i' = \sigma \sigma_i,&
    &\text{and}&
    &q_i' = \frac{qq_i}{q+q_i}.&
  \end{align*}
  We already know that:
  \begin{align*}
    \Phi(Z) \propto \mathcal E_q(\sigma \sigma_0) + \sum_{l=1}^m \mathcal E_{q_l'}\left(\sigma_{l}'\right),
  \end{align*}
  Considering $l\in \{1,\ldots, l_0-1\}$ the characterization given by Proposition~\ref{pro:identity_multiregime_parameters} allows us to set the equivalence:
  \begin{align}\label{eq:equivalence_sigma_0_t_0}
  t_{\sigma\sigma_0,q}^{\sigma'_{l},q_l'} \leq t_{\sigma'_{l},q'_{l}}^{\sigma'_{l_0},q'_{l_0}}\quad
    &\Longleftrightarrow
    \quad\left( \frac{\sigma_0}{\sigma_l} \right)^{\frac{q q_l'}{q  - q_l'}} \leq \left( \frac{\sigma_l}{\sigma_{l_0}} \right)^{\frac{ q_l'q_{l_0}'}{q_{l_0}'  - q_l'}}\\
    &\Longleftrightarrow\quad
    \left( \frac{\sigma_0}{\sigma_l} \right)^{q_l} \leq \left( \frac{\sigma_l}{\sigma_{l_0}} \right)^{\frac{q_lq_{l_0}}{q_l-q_{l_0}}}
    \quad\Longleftrightarrow\quad
    \sigma_0 \leq \left( \frac{\sigma_l^{q_l}}{\sigma_{l_0}^{q_{l_0}}} \right)^{\frac{1}{q_l-q_{l_0}}} \equiv t_{\sigma_l,q_l}^{\sigma_{l_0},q_{l_0}}
  \end{align}
  (no primes in this last inequality).
  Therefore, since $\sigma_0>t_{\sigma_{l_0-1},q_{l_0-1}}^{\sigma_{l_0},q_{l_0}} \geq t_{\sigma_l,q_l}^{\sigma_{l_0},q_{l_0}}$, the upper equivalence allows us to set:
  \begin{align*}
    t_{\sigma\sigma_0,q}^{\sigma'_{l},q_l'} \geq t_{\sigma'_{l},q'_{l}}^{\sigma'_{l_0},q'_{l_0}},
  \end{align*}
  which implies, thanks to Proposition~\ref{pro:regime_cleaning} $\mathcal E_{q_l'}(\sigma_l') \leq \mathcal E_{q}(\sigma \sigma_0) + \mathcal E_{q_{l_0}'}(\sigma_{l_0}')$ and therefore:
  \begin{align*}
    \Phi(Z) \propto \mathcal E_q(\sigma \sigma_0) + \sum_{l=l_0}^m \mathcal E_{q_l'}\left(\sigma_{l}'\right).
  \end{align*}

  To show that $(q,\sigma\sigma_0), (\frac{qq_{l_0}}{q + q_{l_0}},\sigma\sigma_{l_0}),\ldots, (\frac{qq_{m}}{q + q_{m}},\sigma\sigma_{m})$ is thrifty, let us first recall from Proposition~\ref{pro:stability_thrifty} that $(\sigma_1,q'_1), \ldots, (\sigma_m,q'_m)$ is thrifty since $(\sigma_1,q'_1), \ldots, (\sigma_m,q'_m)$ and $(\sigma_1,q), \ldots, (\sigma_m,q)$ are both thrifty and $(\sigma'_1,q'_1), \ldots, (\sigma'_m,q'_m)$ is thrifty since $(\sigma_1,q'_1), \ldots, (\sigma_m,q'_m)$ and $(\sigma,q'_1), \ldots, (\sigma,q'_m)$ are both thrifty. We then know from Proposition~\ref{pro:identity_multiregime_parameters} that we are just left to show that $t_{\sigma\sigma_0,q}^{\sigma'_{l_0},q_{l_0}'} \leq t_{\sigma'_{l_0},q'_{l_0}}^{\sigma'_{l_0+1},q'_{l_0+1}}$, but this is immediate thanks to equivalence~\eqref{eq:equivalence_sigma_0_t_0} and the hypothesis $\sigma_0 \leq t_{\sigma_{l_0},q_{l_0}}^{\sigma_{l_0+1},q_{l_0+1}}$
  \end{proof}
We end with the characterization of thrifty family in the case where $\forall i, q_i = \frac{q}{i}$, for some $q>0$. This s $\sigma_0>t_{\sigma_{l_0-1},q_{l_0-1}}^{\sigma_{l_0},q_{l_0}}$etting comes naturally when one uses several times Theorem~\ref{the:concentration_concentrated_variations}. Indeed, when $Z \propto \mathcal E_2(\sigma)$ and $V \in O(\sigma_0) \pm \mathcal E_2(\sigma_1)$, $\Phi(Z) \propto \mathcal E_2(\sigma\sigma_0) + \mathcal E_1(\sigma\sigma_1)$ and if $\forall i \in [p]$, $V \in O(\sigma_0) \pm \mathcal E_2(\sigma_1) + \mathcal E_1(\sigma_2)$, $\Phi(Z) \propto \mathcal E_2(\sigma\sigma_0) + \mathcal E_1(\sigma\sigma_1) + \mathcal E_{\frac{1}{2}}(\sigma\sigma_2)$. We will see it appear in some practical examples.
\begin{proposition}\label{pro:identity_multiregime_q_s_i}
  Given $q>0$ and $m$ parameters $\sigma_1,\ldots, \sigma_k>0$, the family of multi regime parameters $(q,\sigma_1), \ldots, (\frac{q}{m},\sigma_m)$ is thrifty if and only if one of the following properties is satisfied: 
  \begin{itemize}
    \item $\forall i\in [m-2]: \quad\sigma_{i} \sigma_{i+2}
        \leq \sigma_{i+1}^2$
    \item $\forall k,l\in [m-1], k\leq l: \quad\sigma_{k} \sigma_{l+  1}
        \leq \sigma_{k+1}\sigma_{l}$.  
  \end{itemize}
\end{proposition}

\begin{proof}
  Thanks to Proposition~\ref{pro:identity_multiregime_parameters}, the only difficulty is to show that the first point implies the second point. It is obtained straightforwardly by multiplying the two inequalities:
  \begin{align*}
    \sigma_k^{l - k} \sigma_{l+1} \leq \sigma_{k+1}^{l+1-k}&
    &\text{and}&
    &\sigma_k \sigma_{l+1}^{l - k} \leq \sigma_{l}^{l+1-k}.
  \end{align*}
\end{proof}


Proposition~\ref{pro:characterization_moments}, which provides a control of the centered moments of a concentrated vector, cannot be directly applied when the concentration follows differing exponential regimes as in Theorem~\ref{the:Concentration_produit_de_vecteurs_d_algebre_optimise}. We give here a generalization of this result.
\begin{proposition}[Moment characterization of multi-regime concentration]\label{pro:characterization_exponential_concentration_multiple_regime}
  Given an integer $m\in \mathbb N$, and a sensible family of $2m$ multi-regime parameters $\sigma_1, \ldots, \sigma_m>0$, $q_1, \ldots, q_m>0$, a random variable $Z\in \mathbb R$ satisfies the concentration:
  \begin{align*}
     Z \propto \sum_{l=1}^m \mathcal E_{q_l}(\sigma_l )
   \end{align*} 
   for some constants $C,c>0$ if and only if there exist two constants $C',c'>0$ depending only on $C,c$ such that for all $r > 0$, we have the bound:
  \begin{align}\label{eq:borne_moment}
      \mathbb E \left[\left\vert Z - \mathbb E \left[Z\right]\right\vert^r\right]    \leq C'\max_{l\in [m]}   \left(\frac{r}{q_lc'}\right)^{\frac{r}{q_l}} \sigma_l^r.
  \end{align}
\end{proposition}

\begin{proof}
This proof is mainly a rewriting of \cite[Proposition 1.10]{LED05} with a fine study of the different concentration regimes.
We start with the direct implication which is easier to prove. Assume that there exists two constants $C,c>0$ such that:
\begin{align*}
  \forall t >0, \ \ \mathbb P \left(\left\vert Z - \mathbb E[Z]\right\vert\geq t\right) \leq C\max_{l \in [m]}  e^{-c(t/\sigma_l)^{q_l}}.
\end{align*}
Given $r>0$:
  \begin{align*}
    \mathbb{E}\left[ \vert Z - \mathbb E[Z] \vert^{r}\right]
    &=\int_{0}^{\infty} \mathbb{P}\left( \vert Z - \mathbb E[Z] \vert^{r} \geq t\right) dt\\
    &=\int_{0}^{\infty} rt^{r-1}\mathbb{P}\left( \vert Z - \mathbb E[Z] \vert \geq t\right) dt \\
    &\leq \max_{l\in [m]} C r \int_{0}^{\infty} t^{r-1}  e^{-(t/\sigma_lc^{-1/q_l})^{q_l}}dt \\
    &=  \max_{l\in [m]} C \left( \frac{\sigma_l}{c^{1/q_l}} \right)^{r}  r\int_{0}^{\infty} t^{r-1}  e^{- c t^{q_l}}dt,
  \end{align*}
  and, if we assume that $r\geq q_1$  ($\geq q_l$ for all $l\in [m]$):
  \begin{align*}
    r\int_{0}^{\infty} t^{r-1}  e^{- t^{q_l}}dt
    &= \frac{lr}{q}\int_{0}^{\infty} t^{\frac{r}{q_l}-1} e^{- t}dt
    = \frac{r}{q_l}\Gamma\left(\frac{r}{q_l}\right)
     \leq \left(\frac{r}{q_l}\right)^{\frac{r}{q_l}},
  \end{align*} 
  when $r<q_1$, one can still bound with Jensen's inequality (since $\frac{r}{q_1}\leq 1$):
  \begin{align*}
    \mathbb{E}\left[  \vert Z - \mathbb E[Z] \vert^{r}\right]
    &\leq \mathbb{E}\left[  \vert Z - \mathbb E[Z] \vert^{q_1}\right]^{\frac{r}{q_1}} 
    \leq C^{\frac{r}{q_1}} \max_{l\in [m]} l^{\frac{r}{q_l}} (c\sigma_l)^r \leq C^{\frac{r}{q_1}} m^{m} (c\sigma_1)^r.
  \end{align*}
  Since $m^m,\max(C,1) \leq O(1)$, we can choose cleverly our constants to set the first implication of the proposition.

  Let us now assume \eqref{eq:borne_moment}. 
  We deduce from Markov inequality and basic integration calculus that $\forall r>0$:
  \begin{align}\label{eq:borne_P_markov}
    \mathbb P \left(\left\vert Z - \mathbb E[Z]\right\vert\geq t\right) \leq \frac{\mathbb E \left[\left\vert Z - \mathbb E \left[Z\right]\right\vert^r\right]}{t^r} 
    \leq  C \max_{l\in[m]}   \left(\frac{r (\sigma_l/t)^{q_l}}{q_lc}\right)^{\frac{r}{q_l}} .
  \end{align}
  Given any $k,l \in [m]$, noting:
  \begin{align*}
     t_{k,l} \equiv \left( \frac{\sigma_{k}^{q_k}}{\sigma_{l}^{q_l}} \right)^{\frac{1}{q_k-q_l}}
   \end{align*} 
   and we note $t_{l} \equiv t_{l-1,l}$.
  We know in particular that $0 = t_0 \leq t_1 \leq \cdots \leq t_{m+1} = \infty $. Given $l\in [m]$ and $t \in [t_{l}, t_{l+1}]$ if we chose $r = \frac{cq_m}{e} (\frac{t}{\sigma_l})^{q_l }$, then, for all $k\in[m]$, we want to bound with a $\mathcal E_l$ decay the quantity:
  \begin{align*}
    c_k(t) 
    \equiv C \left(\frac{r }{cq_k}(\sigma_k/t)^{q_k}\right)^{\frac{r}{q_k}} 
    = C \left(\frac{q_m}{e q_k} \left(\frac{\sigma_k^{q_k}\sigma_l^{-q_l}}{t^{q_k-q_l}}\right)\right)^{\frac{q_m}{eq_k}(\frac{t}{\sigma_l})^{q_l }}
  \end{align*}
  to be able to bound the concentration inequality \eqref{eq:borne_P_markov}.

  If $k=l$, we have directly:
  \begin{align*}
     c_k(t) = c_l(t) 
     = C \left( \frac{q_m}{eq_l} \right)^{\frac{cq_m}{eq_l}(\frac{t}{\sigma_l})^{q_l }} 
     \leq C e^{-\frac{cq_m}{eq_1}(\frac{t}{\sigma_l})^{q_l }}.
   \end{align*} If $k \leq l-1$, $q_k \geq q_l$ then $t_{k,l} \leq t_l \leq t$ which implies $1/t^{q_k-q_l} \leq 1/t_{k,l}^{q_k-q_l}$ and:
  \begin{align*}
    c_k(t) 
    \leq  \left(\frac{q_m}{e q_k} \left(\frac{\sigma_k^{q_k}\sigma_l^{-q_l}}{t_{k,l}^{q_k-q_l}}\right)\right)^{\frac{cq_m}{eq_k}(\frac{t}{\sigma_l})^{q_l }}
     \leq C e^{-\frac{cq_m}{eq_1}(\frac{t}{\sigma_l})^{q_l }} .
  \end{align*}
  And the same way, when $k \geq l+1$, $q_k \leq q_l$, then $t \leq t_{l+1} \leq t_{l,k}$, and we can bound $1/t^{q_k-q_l} = t^{q_l-q_k} \leq t_{l,k}^{q_l-q_k}$ which allows us to conclude again that $c_k(t) \leq C e^{\frac{cq_m}{eq_1}(\frac{t}{\sigma_l})^{q_l }} $. 
  When $t \in  (0,t_1]$, choosing $r = \frac{q_m}{e} (\frac{t}{c\sigma_1})^{q_1}$, we show the same way that $\forall k\in[m]$, $c_k(t) \leq C e^{-\frac{k}{me}(\frac{t}{c\sigma_1})^{q_1}} $.
   We eventually obtain for all $t \in \cup_{0 \leq l \leq m}(t_{l},t_{l+1}] \supset \mathbb R_*^+$:
  \begin{align*}
    \mathbb P \left(\left\vert Z - \mathbb E[Z]\right\vert\geq t\right) \leq \max_{l\in[m]}C e^{-c'(\frac{t}{\sigma_l})^{q_l }},
  \end{align*}
  with $c'=\frac{cq_m}{eq_1}$. This is the looked for concentration.
  \end{proof}

\begin{remark}\label{rem:controle_des_moments}
  Proposition~\ref{pro:characterization_exponential_concentration_multiple_regime} is generally employed to bound the first centered moments of an observation. In this case, $\left(\frac{rl}{q}\right)^{\frac{rl}{q}}\leq O(1)$, and when $\sigma \leq O(\mu_{(1)})$ (which is generally the case), there exists a constant $C>0$ such that we can bound
  for any constant $r>0$ ($r\leq O(1)$):
  \begin{align*}
    \mathbb E \left[\left\vert f(Z) - \mathbb E \left[f(Z)\right]\right\vert^r\right]    \leq C (\sigma_1)^r,
  \end{align*}
  since $\sigma_1\geq \sigma_l$, $\forall l \in[m]$. 
  We then see that the first exponential regime $\mathcal E_{q_1}(\sigma_1)$ controls the first statistics of the observations and we then say that the observable diameter of $Z$ is of order $O(\sigma_1)$. 
\end{remark}

\section{Concentration of generalized products of random vectors}

To treat the product of vectors, we provide a general result of concentration of what could be called ``\textit{multilinearly $m$-Lipschitz mappings}'' on normed vector spaces. Instead of properly defining this class of mappings we present it directly in the hypotheses of the theorem. Briefly, these mappings are multivariate functions which are Lipschitz on each variable, with a Lipschitz parameter depending on the product of the norms (or semi-norms) of the other variables and/or constants. To express the observable diameter of such an observation, one needs a supplementary notation. 

Given a vector of parameters $(\nu_l)_{l\in[m]} \in \mathbb R_+^m$, we denote for any $k\in[m]$:
\begin{align*}
  \nu^{(k)} \equiv \max_{1\leq l_1<\cdots<l_k\leq m} \nu_{l_{1}}\cdots \nu_{l_{k}} = \nu_{(m-k+1)} \cdots \nu_{(m)},
\end{align*}
where $\{\nu_{(l)}\}_{l\in [m]} = \{\nu_l\}_{l\in[m]}$ and $\nu_{(1)} \leq \cdots \leq \nu_{(m)}$.
This Theorem is the iterative consequence of Theorem~\ref{the:concentration_concentrated_variations} provided at the end of the section and that has the particularity to control the variations of $\phi$ with concentrated variables instead of norms (or semi-norms) $(\|\cdot \|_j')_{j\in [m]\setminus \{i\}}$.  

\begin{theorem}[Concentration of generalized product]
\label{the:Concentration_produit_de_vecteurs_d_algebre_optimise}
  
Given a constant $m$ ($m \leq O(1)$), let us consider:
\begin{itemize}
  \item $m$ (sequences of) normed vector spaces $(E_1, \| \cdot \|_1), \ldots,(E_m, \| \cdot \|_m)$.
  \item $m$ (sequences of) norms (or semi-norms) $\| \cdot \|_1', \ldots, \| \cdot \|_m'$, respectively defined on $E_1,\ldots, E_m$.
  \item $m$ (sequences of) random vectors $Z_1 \in E_1,\ldots, Z_m \in E_m$ satisfying $$Z\equiv(Z_1,\ldots,Z_m) \propto  \mathcal E_q(\sigma)$$
  for some (sequence of) positive numbers $\sigma\in \mathbb R_+$, and for both norms\footnote{One just needs to assume the concentration $\| Z_i \|'_ i \in \mu_i \pm \mathcal E_q(\sigma)$; the global concentration of $Z_i$ for the norm (or seminorm) $\|\cdot\|_i'$ is not required.} $\Vert z_1,\ldots,z_m \Vert_{\ell^\infty} = \sup_{i=1}^m \Vert z_i \Vert_i$ and $\Vert (z_1,\ldots,z_m) \Vert'_{\ell^\infty} = \sup_{i=1}^m \Vert z_i \Vert'_i$ defined on $E = E_1\times \cdots \times E_m$.
  \item a (sequence of) normed vector spaces $(F,\|\cdot \|)$, a (sequence of) mappings $\phi : E_1,\ldots, E_m \rightarrow F$, such that $\forall (z_1,\ldots,z_m) \in E_1\times \cdots\times  E_m$ and $z_i' \in E_i$:
  \begin{align*}
  \left\Vert \phi(z_1,\ldots,z_m) -  \phi(z_1,\ldots,z_{i-1},z_i', \ldots,z_m)\right\Vert  
  \leq \frac{\prod_{j=1}^m  \max(\left\Vert z_j\right\Vert'_j, \mu_j) }{ \max(\left\Vert z_i\right\Vert'_i, \mu_i )}\left\Vert z_i - z_i'\right\Vert_i.
   \end{align*}
   where $\mu_i >0$ is a (sequence of) positive reals such that $\mu_i \geq \mathbb E[ \| Z_i \|'_ i ]$. 
   We further assume $\mu_i \geq O(\sigma)$.\footnote{This is a very light assumption: it is hard to find any practical example where $\mu_i \ll \sigma$.}
\end{itemize}
  Then we have the concentration :\footnote{Which means that there exist two constants $C,c>0$ such that for all indexes and for all $1$-Lipschitz mapping $f : F \to \mathbb R$, and $\forall t>0$, $\eqref{eq:regime_complet}$ is satisfied.
  Here since $m\leq O(1)$, taking the maximum over $l\in[m]$ is equivalent to taking the sum, up to a small change of the constants; we will thus indifferently write $\phi \left(Z\right) \propto \max_{l \in [m]} \mathcal E_{lq/m}\left(\left(\sigma\mu^{(l-1)}\right)^{\frac{m}{l}}\right) $ or $\phi \left(Z\right) \propto \sum_{l=1}^{m} \mathcal E_{lq/m}\left(\left(\sigma\mu^{(l-1)}\right)^{\frac{m}{l}}\right)$.
  }
  \begin{align}\label{eq:concentration_forme_m_lin}
      \phi \left(Z\right) \propto \max_{l \in [m]} \mathcal E_{q/{l}}\left(\sigma^{l}\mu^{(m-l)}\right) 
    \end{align}

\end{theorem}
It is explained in Remark~\ref{rem:controle_des_moments} that, in the setting of Theorem~\ref{the:Concentration_produit_de_vecteurs_d_algebre_optimise}, the standard deviation (resp. the $r^{\text{th}} $ centered moment with $r \leq O(1)$) of any $1$-Lipschitz observation of $\phi(Z)$ is of order $O\left(\sigma\mu^{(m-1)}\right)$ (resp. $O\left((\sigma\mu^{(m-1)})^r\right)$) thus the observable diameter,  is given by the first exponential decay, $\mathcal E_{q}\left(\sigma\mu^{(m-1)}\right)$, which represent the guiding term of \eqref{eq:concentration_forme_m_lin}. 

Note that the multi-regime family $(\sigma^l \mu^{(m-l)}, \frac{q}{l})_{l\in [m]}$ is clearly thrifty thanks to Proposition~\ref{pro:identity_multiregime_q_s_i}.

The proof is simply a consequence of Theorem~\ref{the:concentration_concentrated_variations}.
  \begin{proof}
  We only prove the result for $\sigma = 1$ since it is easy to get back to this setting from a general cases (replacing $Z_i$ by $Z_i/\sigma$ and $\mu_i$ by $\mu_i\sigma$).
  Let us assume Theorem~\ref{the:Concentration_produit_de_vecteurs_d_algebre_optimise} up to $m=m_0-1$ and let us try to show its validity for $m=m_0$ thanks to Theorem~\ref{the:concentration_concentrated_variations}. We can show with the iteration hypothesis for $m=m_0-1$ that for all $i\in[m]$:
    \begin{align*}
      \|Z_1\|'_1\cdots \|Z_{i-1}\|'_{i-1} \|Z_{i+1}\|'_{i+1} \cdots \|Z_{m}\|'_{m} 
      &\in O(\mu_{-i}^{(m-1)}) + \sup_{l\in[m-1]}\mathcal E_l \left(\mu_{-i}^{(m-l-1)}\right),
    \end{align*}
    where $\mu_{-i} \equiv \mu_1\cdots \mu_{i-1} \mu_{i-1} \cdots \mu_{m}$. Now, since for all $k \in [m-1]$, $\mu^{(l)}_{-i} \leq O(\mu^{(l)})$, we retrieve the hypotheses of Theorem~\ref{the:concentration_concentrated_variations} with $\forall i \in[m]: \sigma_i \equiv \mu_{-i}^{(m-i)}$ and we can prove Theorem~\ref{the:Concentration_produit_de_vecteurs_d_algebre_optimise} for $m = m_0$.
  \end{proof}


Let us now provide the result of Adamczac and Wolff in \cite{adamczak2015concentration} to compare it with Theorem\ref{the:Concentration_produit_de_vecteurs_d_algebre_optimise}. It relies on some notations originally introduced by Latala in \cite{latala2006estimates}. Let us denote $P_m$, the sets of partition of $[m]$ into nonempty, pairwise disjoint sets. Given a tensor $A = (a_\alpha)_{\alpha \in [n]^m}$ and a partition under ordered sets\footnote{This ordering is not specified in \cite{adamczak2015concentration}, but otherwise, we do not see how to define properly the notation $\alpha_{J_m}$ for $\alpha \in [n]^m$ and $k \in [l]$.} $\mathcal J  = \{J_1,\ldots, J_l\}$, define:
\begin{align*}
  \|A\|_{\mathcal J} \equiv \sup \left\{ \sum_{\alpha \in [n]^m} a_\alpha \prod_{k=1}^l x_{\alpha_{J_k}}^{(k)} \ | \ \forall k \in[l] : x^{(k)} \in \mathbb R^{J_k}, \|x^{(k)}\|^2 \leq 1  \right\},
\end{align*}
where for any multiindex $\alpha \in [n]^m$, and any $k \in [l]$, we noted $\alpha_{J_k} = (\alpha_i)_{i \in J_k}$ and for any $x \in \mathbb R^{J_k}$, $\|x\| \equiv \sqrt{\sum_{\alpha \in J_k} x_\alpha^2}$.

\begin{theorem}[\cite{adamczak2015concentration}, Theorem 1.2]\label{the:Adamczac_wolff}
  Given a (sequence of) random vector $X \in \mathbb R^n$, satisfying for any mapping $h : \mathbb R^n \to \mathbb R$ and any $r\geq 2$:
  \begin{align}\label{eq:poincare_p_moment}
     \mathbb E \left[ \left\vert h(X) - \mathbb E[h(X)] \right\vert^r \right] \leq \sigma \sqrt r \mathbb E \left[ \left\vert \nabla h(X) \right\vert^r \right]
   \end{align} and a $\mathcal C^m$ mapping $f: \mathbb R^n \to \mathbb R$ such that $x \mapsto \restriction{d^mf}{x}$ is uniformly bounded on $\mathbb R^n$, we have the concentration:
  \begin{align*}
    f(X) \propto \max_{\mathcal J \in P_m} \left( \mathcal E_{\frac{2}{\#\mathcal J}} \left( \sigma^m \sup_{x\in \mathbb R^n} \|\restriction{d^mf}{x}\|_{\mathcal J} \right) \right) + \max_{\genfrac{}{}{0pt}{2}{l \in [m-1]}{\mathcal J \in P_l}} \left( \mathcal E_{\frac{2}{\#\mathcal J}} \left( \sigma^l \|\mathbb E[\restriction{d^lf}{X}]\|_{\mathcal J} \right) \right)
  \end{align*}
\end{theorem}

\begin{remark}[Comparison between Theorem~\ref{the:Concentration_produit_de_vecteurs_d_algebre_optimise} and Theorem~\ref{the:Adamczac_wolff}]\label{rem:comparaison_with_wolff_adamzcac_result}
The hypotheses of Theorem~\ref{the:Adamczac_wolff} implies that for any $1$-Lipschitz mapping $f: \mathbb R^p \to \mathbb R$:
\begin{align*}
   \forall r >0 : \mathbb{E}\left[\left\vert f(Z_p) - f(Z'_p)\right\vert^r\right] \leq \left(\frac{r}{2}\right)^{\frac{r}{2}}(\sqrt 2\sigma)^r,
 \end{align*} 
 which then allows us, thanks to Proposition~\ref{pro:characterization_moments}, to set $X \propto \mathcal E_2$. Therefore, the hypotheses look slightly weaker than those of Theorem~\ref{the:Concentration_produit_de_vecteurs_d_algebre_optimise}, but not so much because there is a way to connect their hypotheses to log-Sobolev inequalities as seen in \cite{adamczak2017moment, aida1994moment} which can be connected to our hypotheses (see \cite[Theorem 5.3.]{LED05}).

 To simplify the picture, one could observe that our result concerns only the first-order variations of the functionals on each variable, and manage the variability of the coefficient of variation (which in practice is some product of semi-norms of random vectors) with some truncation methods. One of the main limitations is then that one has to treat each random term appearing in the variation coefficient independently, thus producing an observable diameter depending on the mean of the norm of the variation term (and not the norm of the mean of the variations). The strength of \cite{adamczak2015concentration} is to postpone the computation of the norm to the higher order derivative of the functionals thanks to the iterative invocation of its hypothesis of a generalized Poincaré inequality (you can bound the differences of $f$ to its expectation with the expectation of the norm of the derivative, which you can then bound with the norm of its expectation plus the expectation of the norm of the second derivative thanks to the triangular inequality and the generalized Poincaré inequality). This approach becomes very powerful in the case of polynomial functionals, since a certain order of the derivative will vanish, the bound then being composed only of norms of expectations, allowing some cancellation of terms (our result only gives the same bound for the concentration of the monomial).
However, outside the special case of polynomials, their theorem requires the computation of many complex norms. Moreover, if no large order derivative cancels, they will still keep a term like $\sup_x||\restriction{d^m f}{x}||_{\mathcal J}$, without any hint of resolution if it is not bounded (imagine for instance a functional $f(X_1...X_n)$ with the derivatives of f not bounded). In these rather common cases, our approach seems to be the only relevant one.
  
\end{remark}

\begin{remark}[Regime decomposition]\label{rem:regime_concentration_multiple}
  Let us rewrite the concentration inequality \eqref{eq:concentration_forme_m_lin} to let appear the implicit parameter $t$. There exist two constants $C,c >0$ (in particular, $C,c\leq O(1)$) such that for any $1$-Lipschitz mapping $f : F \to \mathbb R$, for any $t>0$:
  \begin{align}\label{eq:regime_complet}
    \mathbb P \left(\left\vert f(\Phi(Z)) - \mathbb E[f(\Phi(Z))]\right\vert\geq t\right)
    \leq C\max_{l\in[m]}\exp \left(-\left(\frac{t/(c\sigma)^l}{\mu^{(m-l)}}\right)^{\frac{q}{l}}\right) ,
  \end{align}
  This expression displays $m$ regimes of concentration, depending on $l \in [m]$: the first one, $C e^{-c(t/ \sigma \mu^{(m-1)})^q}$ ($l= 1$), controls the probability for the small values of $t$, and the last one, $C e^{-ct^{q/m}/\sigma}$ ($l=m$), controls the tail. 
  Let us define:
  \begin{align*}
    t_1=0;&
     &\forall i\in[m]\setminus \{1\}: t_{i} \equiv \mu^{(m-i)} \mu_{(i)}^i 
     = \frac{(\mu^{(m-i+1)})^{i}}{(\mu^{(m-i)})^{i-1}};&
     &t_{m+1} = \infty.
   \end{align*} 
   Recalling that $\mu_{(1)} \leq \cdots \leq \mu_{(m)}$, we see that $t_{1} \leq  \cdots \leq t_{m}$. 
   One can then show that for any $i\in[m]$, we have the equivalence:
  \begin{align*}
     t \in [t_{i}, t_{i+1}]&
     &\Longleftrightarrow& 
    &\forall j \in[m]\setminus\{i\}: \exp \left(- \left(\frac{t/(c\sigma)^j}{\mu^{(m-j)}}\right) ^{\frac{q}{j}}\right) \leq \exp \left(- \left(\frac{t/ (c\sigma)^i}{\mu^{(m-i)}}\right) ^{\frac{q}{i}}\right).
  \end{align*}
  Now, if, for a given $i\in [m]$, $i\geq 2$, $\mu_{(i)} = \mu_{(i+1)}$, then $t_{i} = t_{i+1}$, therefore the term $C\exp \left(-(\frac{(t/(c\sigma)^{i}}{\mu^{(m-i)}})^{\frac{q}{i}} \right)$ can be removed from the expression of the concentration inequality since it never reaches the maximum.

  In particular, when $\mu_{(1)} = \cdots = \mu_{(m)} \equiv \mu_0$,\footnote{To be precise, it is sufficient to assume $\mu_{(2)} = \cdots = \mu_{(m)}$, since $\mu_{(1)}$ never appears in the definition of the $t_{i}$ for $i \in [m]$.} $\forall i\in [m]$, $t_{i} = \mu_0^m$. In this case, there are only two regimes and we can more simply write :
  \begin{align*}
    \Phi(Z) \propto \mathcal E_q(\sigma \mu_0^{m-1}) + \mathcal E_{\frac{q}{m}}(\sigma^m).
  \end{align*}
\end{remark}

\begin{corollary}\label{cor:expression_avec_indice_de_norme}
  In the setting of Theorem~\ref{the:Concentration_produit_de_vecteurs_d_algebre_optimise}, when $\forall i \in[m]$, $\mathbb \|E[Z_i]\|_i \leq O(\sigma \eta_{\|\cdot\|'}^{1/q})$, we have the simpler concentration:
  \begin{align*}
    \phi \left(Z\right) \propto \max_{l \in [m]} \mathcal E_{q/l}\left(\sigma^m \eta^{(m-l)}\right),
  \end{align*}
  where we denoted $\eta = \left(\eta_{\|\cdot\|_1'}^{1/q},\ldots,\eta_{\|\cdot\|_m'}^{1/q}\right)$.
\end{corollary}
\begin{proof}
  Proposition~\ref{pro:tao_conc_exp} allows us to choose $\mu_i = C'\sigma \eta_{\|\cdot\|_i'}^{1/q}$ for some constant $C'>0$; we thus retrieve the result thanks to Theorem~\ref{the:Concentration_produit_de_vecteurs_d_algebre_optimise}.
\end{proof}

Let us give examples of ``multilineary Lipschitz mappings'' that would satisfy the hypotheses of Theorem~\ref{the:Concentration_produit_de_vecteurs_d_algebre_optimise}.
\begin{example}[Entry-wise product]\label{ex:concentration_odot}

    Letting $\odot$ be the entry-wise product in $\mathbb{R}^p$ defined as $[x \odot y ]_i = x_iy_i$ (it is the Hadamard product for matrices), $\phi : (\mathbb R^p)^m \ni (x_1,\ldots,x_m) \mapsto x_1\odot \cdots \odot x_m \in \mathbb R^p$ is multilinearly Lipschitz since we have for all $i\in[m]$:
    \begin{align*}
       \left\Vert x_1\odot\cdots \odot x_{i-1} \odot (x_i - x_i') \odot x_{i+1} \odot \cdots \odot x_m)\right\Vert \leq  \left(\prod_{\genfrac{}{}{0pt}{2}{j=1}{j\neq i}}^m \|x_j\|_\infty\right) \left\Vert x_i - x_i'\right\Vert,
     \end{align*} 
    for all vectors $x_1,\ldots, x_m,x_1',\ldots,x_m' \in \mathbb R^p$.
    As a practical case, if $(Z_1,\ldots, Z_m) \propto \mathcal E_2$ and $\forall i \in[m]$, $\|\mathbb E[Z_i]\|_\infty \leq O(1) $, then Corollary~\ref{cor:expression_avec_indice_de_norme} and Remark~\ref{rem:regime_concentration_multiple} imply:
    \begin{align*}
      Z_1 \odot \cdots \odot Z_m \propto \mathcal E_2 \left(\log(p)^{\frac{m-1}{2}}\right) + \mathcal E_{\frac{2}{m}}.
    \end{align*}
    It is explained in Remark~\ref{rem:controle_des_moments} that in this case, the observable diameter of $Z_1\odot \cdots \odot Z_m$ is provided by $\mathcal E_2 \left(\log(p)^{\frac{m-1}{2}}\right)$: as such, under this very common setting, the entry-wise product has almost no impact on the rate of concentration.
\end{example}
\begin{example}[Matrix product]\label{exe:concentration_matrix_product}
     The mapping $\phi :  (\mathcal M_{p})^q \ni (M_1,\ldots,M_q) \mapsto M_1 \cdots M_m \in \mathcal M_{p}$ is multilinearly Lipschitz since for all $i\in[m]$ :\footnote{One could have equivalently considered, for even $m\in \mathbb N$, the mapping $\phi : \mathcal M_{p,n}^m \to \mathcal M_{p,n}$ satisfying $\forall M_1,\ldots, M_m \in \mathcal M_{p,n}$, $\phi(M_1,\ldots, M_m) = M_1M_2^T \cdots M_{m-1}M_m^T$.} 
    \begin{align*}
       \left\Vert M_1\cdots  M_{i-1}  (M_i - M_i')  M_{i+1}  \cdots  M_m)\right\Vert_F \leq  \left(\prod_{\genfrac{}{}{0pt}{2}{j=1}{j\neq i}}^m \|M_j\|\right) \left\Vert M_i - M_i'\right\Vert_F,
     \end{align*} 
    for all matrices $M_1,\ldots, M_m,M_1',\ldots,M_m' \in \mathcal M_{p}$.
     Given $m$ random matrices $X_1,\ldots, X_m \in \mathcal M_{p}$ such that $(X_1,\ldots, X_m) \propto \mathcal E_2$ and $\forall i \in[m]$, $\|\mathbb E[X_i]\| \leq O(\sqrt p) $, Corollary~\ref{cor:expression_avec_indice_de_norme} implies: 
     \begin{align*}
      X_1\cdots X_m \propto \mathcal E_2 \left(p^{\frac{m-1}{2}} \right) + \mathcal E_{\frac{2}{m}}.
     \end{align*}
     In particular, for a ``data'' matrix\footnote{That is, a matrix whose columns contain vectors of ``data'', as per data science terminology.} $X=(x_1,\ldots, x_n) \in \mathcal M_{p,n}$ satisfying $X \propto \mathcal E_2$ and $\mathbb E[\|X\|] \leq O(\sqrt{p+n})$, the sample covariance matrix satisfies the concentration:
     \begin{align*}
       \frac{1}{n} XX^T \propto \mathcal E_2 \left(\frac{\sqrt{p+n}}{n} \right) + \mathcal E_1 \left(\frac{1}{n}\right),
     \end{align*}
     which provides an observable diameter of order $O(1/\sqrt n)$ when $p \leq O(n)$.
     
\end{example}

\section{Generalized Hanson-Wright theorems}
To give some more elaborate consequences of Theorems~\ref{the:concentration_concentrated_variations} and~\ref{the:Concentration_produit_de_vecteurs_d_algebre_optimise}, let us first provide a matricial version of the popular Hanson-Wright concentration inequality, \cite{HW71}.
\begin{proposition}[Hanson-Wright]\label{pro:concentration_lineaire_YAX}
 Given two random matrices $X,Y\in \mathcal M_{p,n}$, assume that $(X,Y) \propto \mathcal E_2$ and $\|\mathbb E[X]\|_F, \|\mathbb E[Y]\|_F \leq O(1)$ (as $n,p \to \infty$). Then, for any deterministic matrix $A \in \mathcal M_{p}$, we have the linear concentration (in $(\mathcal M_{p,n}, \| \cdot \|_F)$):
  \begin{align*}
     Y^TAX \in  \mathcal E_2\left(  \|A\|_F\right) + \mathcal E_1(\|A\|).
  \end{align*}
\end{proposition}

This proposition, which provides a result in terms of linear concentration, points out an instability of the class of Lipschitz concentrated vectors which (here through products) degenerates into a mere linear concentration. This phenomenon fully justifies the introduction of the notion of linear concentration: it will occur again in Proposition~\ref{pro:concentration_lineaire_XDY} and Lemma~\ref{lem:concentration_Q_m_i_x_i}.
We present the proof directly here as it is a short and convincing application of Theorem~\ref{the:Concentration_produit_de_vecteurs_d_algebre_optimise}.
\begin{proof}
Considering a deterministic matrix $B \in \mathcal M_{n}$ such that $\|B\|_F \leq 1$ we introduce the semi-norm $\|\cdot \|_{A,B}$ defined on $\mathcal M_{p,n}$  and satisfying for all $M \in \mathcal M_{p,n}$, $ \|M\|_{A,B} \equiv \|AMB\|_F$. 
Note that for any $M, P \in \mathcal M_{p,n}$:
\begin{align*}
\tr(BMAP^T)
   \leq \left\{\begin{aligned}
    &\|M\|_{A,B}\|P\|_F\\
    &\|M\|_F\|P\|_{A,B}.
  \end{aligned} \right.
\end{align*}
Thanks to Lemma~\ref{lem:borne_Ax}, $\mathbb E[\|X\|_{A,B}], \mathbb E[\|Y\|_{A,B}] \leq O(\|A\|_F)$, therefore:
\begin{align*}
  \sup(\|X\|_{A,B}, \|Y\|_{A,B}) \in O(\|A\|_F) \pm \mathcal E_2(\|A\|).
\end{align*}
 Therefore, the hypotheses of Theorem~\ref{the:concentration_concentrated_variations} are satisfied and we can deduce the result.

\end{proof}
\begin{remark}\label{rem:convex_concentration_hypothesis}
  For the reader information, we mention that in \cite{ADA14}, the concentration is even expressed on the random variable $\sup_{A \in \mathcal A} X^TAX$ where $\mathcal A$ is a bounded set of matrices (and $X \in \mathbb R^p$).

  In \cite{VU14} and \cite{ADA14}, the result is even obtained assuming convex concentration for $X=Y \in \mathbb R^p$, i.e., the inequalities of Definition~\ref{def:concentrated_sequence} are satisfied for all $1$-Lipschitz and convex functionals.
  This definition is less constrained, thus the class of convexly concentrated random vector is larger\footnote{It is even strictly larger as it was shown in \cite{TAL88} that the uniform distribution on $\{0, 1\}^p$ is convexly concentrated but not Lipschitz concentrated (with interesting concentration speed).} than the class of Lipschitz concentrated random vectors. A well-known theorem of \cite{TAL95} provides the concentration of the Lipschitz \textit{and convex} observations of any random vector $X$ built as an affine transformation of a random vector with bounded (with respect to $p$) and independent entries.

  These looser hypotheses are not very hard to handle in this particular case of quadratic functionals since these observations exhibit convex properties. The main issue is to find a result analogous to Lemma~\ref{lem:concentration_sous_ensemble} to show the convex concentration of $X^TAX$ on events $\{\|X\|\leq K\}$ for $K >0$ (note that these events are associated to convex subsets of $\mathbb R^p$). These details go beyond the scope of the article: we have shown in the ongoing work \cite{louart2022sharp} that Theorem~\ref{the:Concentration_produit_de_vecteurs_d_algebre_optimise} can extend to entry-wise products of $m$ convexly random vectors and to matrix products of $m$ convexly concentrated random matrices (for the latter operations, the concentration is not as good as in the Lipschitz case).
\end{remark}
Let us end this section with a useful consequence of Proposition~\ref{pro:concentration_lineaire_YAX}.
\begin{corollary}\label{cor:borne_norme_d_XAY}
  Given a deterministic matrix $A\in \mathcal M_{p}$ satisfying $\|A\|_F\leq 1$ and two random matrices $X=(x_1,\ldots,x_n),Y=(y_1,\ldots,y_n) \in \mathcal M_{p,n}$ satisfying $(X,Y) \propto \mathcal E_2$ and $\sup_{i\in[n]} \|\mathbb E[x_i]\|, \|\mathbb E[y_i]\| \leq O(1)$ such that  we have the concentration:
  \begin{align*}
    \left\Vert Y^TAX\right\Vert_d \in \mathcal E_2 \left(\sqrt{\log(np)}\right) + \mathcal E_1.
  \end{align*} 

  If\footnote{Actually, to bound the expectation we just need the concentration of each of the couples $(x_i,y_i)$ but not of the matrix couple $(X,Y)$. Note that the concentration of each the $x_1,\ldots, x_n$ does not imply the concentration of the whole matrix $X$, even if the columns are independent. to trackle this issue, some authors \cite{PAJ09} require a logconcave distribution for all the columns because the product of logconcave distribution is also logconcave. However our assumptions are more general because they allow to take for $x_i$ any $O(1)$-Lipschitz transformation of a Gaussian vector which represents a far larger class of random vectors.}, in addition, $\|A\|_* \leq O(1)$ or $\sup_{i\in [n]} \|\mathbb E[y_ix_i^T]\|_ F\leq O(1)$\footnote{To be precise, one just needs $\sup_{i\in [n]} |\mathbb E[x_i^TAy_i]|\leq O(1)$.},
  then $\mathbb E \left[\left\Vert X^TAY\right\Vert_d\right] \leq  O(\sqrt n)$. 
\end{corollary}
\begin{remark}\label{rem:E_x_i_y_i_F_borne}
  Recall from Proposition~\ref{pro:carcaterisation_vecteur_linearirement_concentre_avec_moments} that if $(x_i,y_i) \in \mathcal E_2$ and $\|\mathbb E[(x_i,y_i)]\|_F \leq O(1)$ (as under the hypotheses of Corollary~\ref{cor:borne_norme_d_XAY}), $\|\mathbb E[y_ix_i^T]\|\leq O(1)$ and therefore, $\|\mathbb E[y_ix_i^T]\|_F \leq \sqrt p \|\mathbb E[y_ix_i^T]\|\leq O(\sqrt p)$. In particular, when $x_i$ is independent with $y_i$ (and $\|\mathbb E[x_i]\|, \|\mathbb E[y_i]\| \leq O(1)$), we can bound $\|\mathbb E[y_ix_i^T]\|_F \leq \|\mathbb E[y_i]\|\|\mathbb E[x_i]\|\leq O(1)$.
\end{remark}
\begin{proof}[Proof of Corollary~\ref{cor:borne_norme_d_XAY}]
Decomposing $A = U^T \Lambda V$ with $\Lambda= \diag(\lambda) \in \mathcal D_n$ and $U,V \in \mathcal O_{p}$, noting $\check X \equiv VX$ and $\check Y \equiv UY$ we have the identity:
  \begin{align*}
    \|Y^TAX\|_d \leq \sup_{\genfrac{}{}{0pt}{2}{D \in \mathcal D_n}{\|D\|_F\leq 1}} \tr(D\check Y^T \Lambda \check X) \leq \sup_{\|d\|\leq 1} d^T (\check X \odot \check Y) \lambda \leq \|\check X \odot \check Y\|\leq   \|\check X\|_F \|\check Y \|_\infty,
  \end{align*}
  and the same way, $\|Y^TAX\|_d \leq \|\check X \|_\infty \|\check Y\|_F$. Now we can bound thanks to Proposition~\ref{pro:tao_conc_exp}:
  \begin{align*}
     \mathbb E[\|\check X\|_\infty] \leq \left\Vert \mathbb E[\check X]\right\Vert_\infty +  O (\sqrt{\log(pn)}) \leq  \left\Vert \mathbb E[ X]\right\Vert +  O (\sqrt{\log(pn)}) \leq  O (\sqrt{\log(pn)}),
   \end{align*} 
   and the same holds for $\mathbb E[\|\check Y\|_\infty]$.
   Therefore, applying Theorem~\ref{the:concentration_concentrated_variations} with the concentrations $\check X, \check Y \propto \mathcal E_2$ and the variations concentrations $\|\check X\|_\infty, \|\check Y\|_\infty \in O(\sqrt{\log(pn)}) \pm \mathcal E_2$, we obtain the looked for concentration.

  To bound the expectation, we start with the identity $\left\Vert X^TAY\right\Vert_d = \sqrt{\sum_{i=1}^n (x_i^T Ay_i)^2}$, and we note that 
   the hypotheses of Proposition~\ref{pro:concentration_lineaire_YAX} are satisfied and therefore $x_i^TAy_i \in \mathcal E_1$. Now, if $\|A\|_* \leq 1$, we can bound:
  \begin{align*}
    \left\vert \mathbb E[x_i^TAy_i]\right\vert = \left\Vert \mathbb E[y_ix_i^T]\right\Vert \|A\|_* \leq O(1),
  \end{align*}
  thanks to Proposition~\ref{pro:carcaterisation_vecteur_linearirement_concentre_avec_moments} ($(x_i,y_i) \propto \mathcal E_2$ and $\|\mathbb E[x_i]\mathbb E[y_i]^T\|\leq O(1)$). The same bound is true when $\left\Vert \mathbb E[y_ix_i^T]\right\Vert_F \leq O(1)$ because $\left\vert \mathbb E[x_i^TAy_i]\right\vert \leq \left\Vert \mathbb E[y_ix_i^T]\right\Vert_F \|A\|_F$. As a consequence, $x_i^TAy_i \in O(1) \pm \mathcal E_1$ (with the same concentration constants for all $i\in[n]$), and we can bound:
  \begin{align*}
    \mathbb E \left[\left\Vert X^TAY\right\Vert_d \right] \leq \sqrt{ \sum_{i=1}^n \mathbb  E \left[ (x_i^T Ay_i)^2\right]} \leq O(\sqrt n).
  \end{align*}

\end{proof}

Let us now give an example of application of Theorem~\ref{the:Concentration_produit_de_vecteurs_d_algebre_optimise} when $m \geq 3$.

\section{Concentration of $XDX^T$}\label{sec:XDY}
Considering three random matrices $X,Y \in \mathcal M_{p,n}$ and $D \in \mathcal D_n$ such that $(X,Y,D) \propto \mathcal E_2$ and $\|\mathbb E[D]\|,\|\mathbb E[X]\|_F, \|\mathbb E[Y]\|_F \leq O(1)$ we wish to study the concentration of $XDX^T$. Theorem~\ref{the:Concentration_produit_de_vecteurs_d_algebre_optimise} just allows us to obtain the concentration $XDY^T \propto \mathcal E_2(n) + \mathcal E_1(\sqrt n) + \mathcal E_{2/3}$ since we cannot get a better bound than $\|XDY^T \|_F \leq \|X\|\|D\|_F\|Y\|$. 
However, considering some particular observations on $XDY^T$, it appears that the observable diameter can be smaller than $n$. 
Next Propositions reveal indeed that for any deterministic $u\in \mathbb R^p$ and $A \in \mathcal M_p$:
\begin{enumerate}
  \item $XDY^T u $ is Lipschitz concentrated with an observable diameter of order $O(\|u\|\sqrt{(n+p)\log(n)})$
  \item $\tr(A XDY^T)$ is concentrated with a standard deviation of order $O(\|A\|_F\sqrt{(n+p)\log(np)})$ if $\sup_{i\in [n]} \|\mathbb E[x_iy_i^T] \|_F \leq O(1)$ and $O(\|A\|_*\sqrt{(n+p)\log(np)})$ otherwise.
\end{enumerate}

\begin{proposition}\label{pro:Concentration_XDYu}
  Given three random matrices $X,Y= (y_1,\ldots, y_n) \in \mathcal M_{p,n}$ and $D \in \mathcal D_n$ diagonal such that $(X,Y, D)\propto \mathcal E_2 $, $\|\mathbb E[X]\| \leq O(\sqrt {p+n})$ and $\|\mathbb E[D]\|,\sup_{i\in[n]}\|\mathbb E[y_i]\| \leq O(1)$, for any deterministic vector $u \in \mathbb R^p$ such that $\|u \|\leq 1$: 
    $$X D Y^Tu \propto \mathcal E_2 \left(\sqrt{(p+n)\log (np)}  \right) + \mathcal E_{1} \left(\sqrt{p+n}  \right) +\mathcal E_{2/3}\ \ \text{in} \ \ (\mathbb R^{p},\|\cdot\|).$$ 
\end{proposition}
\begin{proof}

  The Lipschitz concentration of $XDY^Tu$ is obtained thanks to the inequalities:
  \begin{align*}
    \left\Vert XDY^Tv\right\Vert \leq \left\{
    \begin{aligned}
      &\|X\| \| D\|\|Y^Tu\|\\
      &\|X\| \| D\|_F\|Y^Tu\|_{\infty}.\\
    \end{aligned}\right.
  \end{align*}
  Thanks to the bounds already presented in Example~\ref{exe:borne_esp_norm_vecteur_lin_conc} (the spectral norm $\|\cdot \|$ on $\mathcal D_n$ is like the infinity norm $\|\cdot \|_\infty $ on $\mathbb R^n$), we know that:
  \begin{itemize}
     \item $\mathbb E[\|D\|] \leq O(\sqrt{\log n}) + O(\|\mathbb E[D]\|)\leq O(\sqrt{\log n})$,
     \item $\mathbb E[\|Y^Tu\|_\infty] \leq O(\sqrt{\log p}) + \left\Vert \mathbb E[Y^Tu] \right\Vert_\infty\leq O(\sqrt{\log p}) + \sup_{i\in[n]}\left\Vert \mathbb E[y_i] \right\Vert \leq O(\sqrt{\log p})$,
     \item $\mathbb E[\|X\|] \leq O(\sqrt{ p+n}) $
     \item $\mathbb E[\|Y^Tu\|] \leq O(\sqrt{ n}) + \left\Vert \mathbb E[Y^Tu] \right\Vert\leq O(\sqrt{ n}) + \sqrt n \sup_{i\in[n]}\left\Vert \mathbb E[y_i] \right\Vert \leq O(\sqrt n)$.
   \end{itemize}
   One then obtains the concentrations: 
  \begin{itemize}
    \item $\|D\|\|Y^Tu\| \in O(\sqrt{\log(n)n}) \pm \mathcal E_2(\sqrt{n}) +\mathcal E_1$
    \item $\|X\|\|Y^Tu\|_{\infty} \in O(\sqrt{\log(p)(n+p)}) \pm \mathcal E_2(\sqrt{n+p}) +\mathcal E_1$,
  \end{itemize}
  from which Theorem~\ref{the:concentration_concentrated_variations} allows us to conclude.

  \end{proof}

The concentration of $\tr(A XDY^T)$ signifies a linear concentration of $XDY^T$, demonstrating as in Proposition~\ref{pro:concentration_lineaire_YAX} the relevance of the notation of linear concentration. 
Note besides that this result can be seen as a weak offshot of Hanson-Wright concentration inequality if one takes $D = \sqrt{n} E_{1,1}$, where $[E_{1,1}]_{i,j} = 0$ for all $(i,j)\neq (1,1)$ and $[E_{1,1}]_{1,1} = 1$. 

\begin{proposition}\label{pro:concentration_lineaire_XDY}
  Given three random matrices $X=(x_1,\ldots, x_n),Y=(y_1,\ldots, y_n) \in \mathcal M_{p,n}$ and $D \in \mathcal D_n$ such that $(X,Y,D)\propto \mathcal E_2 $, $\|\mathbb E[D]\|_F\leq O(\sqrt n)$, $\|\mathbb E[X]\|_F, \|\mathbb E[Y]\|_F\leq O(1)$, 
  we have the linear concentration\footnote{The estimation of $\mathbb E[XD Y^T]$ is done in Proposition~\ref{pro:estimation_XDY}}:
  \begin{align*}
    XD Y^T \in   \mathcal E_{2}\left(\sqrt{n} \right) + \mathcal E_{1}\left(\sqrt{\log np} \right) + \mathcal E_{2/3}& 
    &\text{in} \ (\mathcal M_{p}, \|\cdot\|).
  \end{align*}
  If, in addition, $\sup_{i \in[n]}\|\mathbb E[y_ix_i^T]\|_F \leq O(1)$:
  \begin{align*}
    XD Y^T \in   \mathcal E_{2}\left(\sqrt{n} \right) + \mathcal E_{1}\left(\sqrt{\log np} \right) + \mathcal E_{2/3}& 
    &\text{in} \ (\mathcal M_{p}, \|\cdot\|_F).
  \end{align*}
\end{proposition}
Before proving this corollary let us give a preliminary lemma of independent interest.
  
\begin{lemma}\label{lem:concentration_AXD}
  Given two random matrices $X \in \mathcal M_{p,n}$ and $D\in \mathcal D_n$ such that $(X,D) \propto \mathcal E_2$, $\|\mathbb E[X]\| \leq O(1)$ and $\|\mathbb E[D] \| \leq O(\sqrt n)$ and a deterministic matrix $A\in \mathcal M_{p}$, such that $\|A\|_F\leq 1$, we have the concentration:
  \begin{align*}
     AXD \propto \mathcal E_2(\sqrt{\log(np)}) + \mathcal E_1&
     &\text{in} \ \ (\mathcal{M}_{p,n}, \|\cdot\|_F).
  \end{align*}
  in addition, one can bound: $\mathbb E[\|AXD\|_F] \leq O(\sqrt n)$.
\end{lemma}
\begin{proof}
  With the same decomposition $A = U^T \Lambda V$ and notation $\check X \equiv VX$ as in the proof of Corollary~\ref{cor:borne_norme_d_XAY}, we have the identity:
  \begin{align*}
    \| AXD \|_F = \| \Lambda \check XD \|_F = \sqrt{\sum_{i=1}^n \sum_{j=1}^p \lambda_j^2 \check X_{i,j}^2 D_{i}^2} \leq \|\check X\|_\infty \|D\|_F,
  \end{align*}
  and besides, $\| AXD \|_F \leq \|\check X\|_F \|D \|$, we can thus employ Theorem~\ref{the:concentration_concentrated_variations} with the concentrations $\hat X, D \propto E_2$ and:
  \begin{align*}
    \|\hat X\|_{\infty} \in O(\sqrt{\log np}) \pm \mathcal E_2&
    &\text{and}&
    &\|D\| \in O(\sqrt{\log n}) \pm \mathcal E_2,
  \end{align*}
  to obtain $AXD \propto \mathcal E_2(\sqrt{\log np}) + \mathcal E_1$. 
  To bound the expectation, let us note that for all $i\in[n], j\in [p]$, $\check X_{i,j} \in O(1) \pm \mathcal E_2$ and $D_{i} \in O(1) \pm \mathcal E_2$, therefore, $(\check X_{i,j}D_i)^2 \in O(1) \pm \sum_{i=1}^4\mathcal E_{\frac{2}{i}} $, with concentration constants independent of $i,j$. Finally, we can bound:
  %
  \begin{align*}
    \mathbb E[\| AXD \|_F] \leq \sqrt{\sum_{i=1}^n \sum_{j=1}^p \lambda_j^2\mathbb E[ \check X_{i,j}^2 D_{i}^2]} \leq \sqrt{ \left(\sum_{i=1}^n O(1)\right) \left(\sum_{j=1}^p \lambda_j^2\right)} \leq O(\sqrt n) .
  \end{align*}  
\end{proof}
\begin{proof}[Proof of Proposition~\ref{pro:concentration_lineaire_XDY}]
  Considering a deterministic matrix $A \in \mathcal M_{p,n}$, we will assume that $\|A\|_F \leq 1$ if $\sup_{i \in[n]}\|\mathbb E[y_ix_i^T]\|_F \leq O(1)$ (to show a concentration in $(\mathcal M_{p,n}, \|\cdot\|_F)$) and that $\|A\|_* \leq 1$ otherwise (to show a concentration in $(\mathcal M_{p,n}, \|\cdot\|)$). In both cases, Corollary~\ref{cor:borne_norme_d_XAY} and Lemma~\ref{lem:concentration_AXD} allows us to set:
  \begin{align*}
    \|Y^TAX\|_d, \ \|AXD\|_F, \ \|DY^TA\|_F \in O(\sqrt n) \pm \mathcal E_2(\sqrt{\log np}) + \mathcal E_1.
  \end{align*}  
  Besides, we can bound:
  \begin{align*}
     \tr(AXDY^T) \leq \left\{\begin{aligned}
       &\|AXD\|_F \|Y\|_F\\
       &\|DY^TA\|_F \|X\|_F\\
       &\|Y^TAX\|_d \|D\|_d,
     \end{aligned}\right.
  \end{align*}
  which allows us to conclude thanks to Theorem~\ref{the:concentration_concentrated_variations}.
\end{proof}

In the setting of Proposition~\ref{pro:Concentration_XDYu}, once one knows that $XDY$ is concentrated it is natural to look for a simple deterministic equivalent. The next proposition help us for such a design.
Note that the hypotheses are far lighter, in particular, we just need the linear concentration of $D$.

\begin{proposition}\label{pro:estimation_XDY}
  Given three random matrices $D \in \mathcal D_n$, $X=(x_1,\ldots, x_n),Y=(y_1,\ldots, y_n) \in \mathcal M_{p,n}$ and a deterministic matrix $\tilde D \in \mathcal D_n$, such that  $D \in \tilde D \pm\mathcal E_2$ in $(\mathcal D_n, \|\cdot\|)$ 
  and for all\footnote{If we adopt the stronger assumptions $(X,Y) \propto \mathcal E_2$ in $(\mathcal M_{p,n}, \|\cdot\|)$ and $\|\mathbb E[X]\|_F, \|\mathbb E[X]\|_F \leq O(1)$, we can show more directly thanks to Propositions~\ref{pro:concentration_lineaire_YAX} and~\ref{pro:tao_conc_exp}:
  \begin{align*}
    &\left\vert \mathbb E[\tr(AXDY^T)] -\mathbb E[\tr(AX\mathbb E[D] Y^T)]\right\vert\\
    &\hspace{1cm}=  \left\vert \mathbb E \left[\tr \left(\left(Y^TAX - \mathbb E[Y^TAX]\right) \left(D - \mathbb E[D]\right)\right) \right] \right\vert\\
    &\hspace{1cm}\leq \sqrt{\mathbb E \left[\left\Vert Y^TAX - \mathbb E[Y^TAX]\right\Vert_d^2\right] \mathbb E \left[\left\Vert D - \mathbb E[D]\right\Vert_d^2 \right]}
    \ \leq O(n)
  \end{align*}} $i \in [n]$, $(x_i,y_i)\propto \mathcal E_2$ and $\sup_{i\in[n]}\|\mathbb E[x_i]\|, \|\mathbb E[y_i]\| \leq O(1)$, 
  we have the estimate:
  \begin{align*}
    \left\Vert \mathbb E[XDY^T] -\mathbb E[X \mathbb E[D]Y^T]\right\Vert_F \leq O \left(n\right).
  \end{align*}
  We can precise the estimation with supplementary assumptions:
  \begin{itemize}
    \item if $\|\tilde D - \mathbb E[D]\|_F \leq O(1)$ then $\left\Vert \mathbb E[X(D-\tilde D)Y^T]\right\Vert_F \leq O \left(\sqrt{n \max(p,n)}\right)$
    \item if $\sup_{i\in [n]}\|\mathbb E[x_iy_i^T] \|_F\leq O(1)$ then $\left\Vert \mathbb E[XDY^T] -\mathbb E[X \tilde DY^T]\right\Vert_F \leq O \left(n\right)$.
  \end{itemize}  
  
\end{proposition}

\begin{proof}
  Considering a deterministic matrix $A\in \mathcal M_{p}$, such that $\|A\|_F \leq 1$:
  \begin{align*}
    &\left\vert \mathbb E[\tr(AXDY^T)] -\mathbb E[\tr(AX\mathbb E[D] Y^T)]\right\vert\\
    &\hspace{1cm}\leq \sum_{i=1}^n \left\vert \mathbb E \left[D_i x_i^T A y_i -\mathbb E[D_i] x_i^T A y_i\right]\right\vert \\
    &\hspace{1cm}= \sum_{i=1}^n \left\vert \mathbb E \left[ \left(x_i^T A y_i - \mathbb E[x_i^T A y_i]\right) \left(D_i-\mathbb E[D_i]\right)\right]\right\vert\\
    &\hspace{1cm}\leq \sum_{i=1}^n  \sqrt{\mathbb E \left[\left\vert x_i^T A y_i - \mathbb E [x_i^T A y_i] \right\vert^2 \right] \mathbb E \left[\left\vert D_i -\mathbb E[D_i] \right\vert^2\right]} \ \ 
    \leq  \ O \left(n \right)
  \end{align*}
  thanks to H\"older's inequality applied to the bounds given by Proposition~\ref{pro:characterization_exponential_concentration_multiple_regime} (we know that $D_i \in \mathbb E[D_i] \pm \mathcal E_2$ and from Proposition~\ref{pro:concentration_lineaire_YAX} that $x_i^T A y_i \in \mathbb E[x_i^T A y_i] \pm \mathcal E_2 + \mathcal E_1$; note that the concentration constants are the same for all $i\in [n]$). 

  Now, $|\mathbb E[x_i^T A y_i]| \leq \|A\|_F \|\mathbb E[y_ix_i^T]\|_F \leq O(\sqrt{p})$ thanks to Remark~\ref{rem:E_x_i_y_i_F_borne} and if $\|\mathbb E[D] - \tilde D\|_F\leq O(1)$, we can bound:
  \begin{align*}
    \left\vert \mathbb E \left[\tr\left(AX \! \left(\mathbb E[D] -\tilde D\right)Y^T \!\right)\right]\right\vert\! \!
    &\leq  \sum_{i=1}^n \left\vert \mathbb E \left[ x_i^T A y_i \right] \left(\mathbb E[D_i] - \tilde D_i\right)\right\vert\\
    &\leq \sup_{i\in [n]} \left\vert \mathbb E \left[ x_i^T A y_i \right] \right\vert \sqrt{n}\left\Vert \mathbb E[D] -\tilde D\right\Vert_F \! \! \leq \!  O \left(\sqrt{np} \right).
  \end{align*}
  If, $\|\mathbb E[D] - \tilde D\|_F$ is possibly of order far bigger than $O(1)$, but $\sup_{i\in [n]} \|\mathbb E[y_ix_i^T]\|_F \leq O(1)$, then $\sup_{i\in [n]}|\mathbb E[x_i^T A y_i]| \leq O(1)$, and we can still bound:
  \begin{align*}
     \left\Vert \mathbb E \left[\tr\left(AX \left(\mathbb E[D] -\tilde D\right)Y^T\right)\right]\right\Vert_F \leq n\sup_{i\in [n]} \left\vert \mathbb E \left[ x_i^T A y_i \right] \right\vert \left\Vert \mathbb E[D] -\tilde D\right\Vert \leq O(n).
  \end{align*}
\end{proof}

Let us end this article with a non multi-linear application of Theorem~\ref{the:Concentration_produit_de_vecteurs_d_algebre_optimise}.

\section{Concentration of the resolvent $(I_p - \frac1nXDY^T)^{-1}$}
We study here the concentration of a resolvent $Q =(I_p - \frac{1}{n}XDY^T)^{-1}$ with the assumption of Proposition~\ref{pro:concentration_lineaire_XDY} for $X, Y$ and $D$ (in particular $D$ is random). Among other use, this object appears when studying robust regression \cite{ELK13, MAI19}. 
In several settings, robust regression can be expressed by the following fixed point equation:
\begin{align}\label{eq:point_fixe_robust}
   \beta = \frac{1}{n} \sum_{i=1}^n f(x_i^T\beta)x_i, \ \ \beta \in \mathbb R^p,
 \end{align} 
where $\beta$ is the weight vector performing the regression (to classify data, for instance).
It was then shown in \cite{SED21} that the estimation of the expectation and covariance of $\beta$ (and therefore, of the performances of the algorithm) rely on an estimation of $Q$, with $D = \diag(f'(x_i^T\beta))$.
To obtain a sharp concentration on $Q$ (as it is done in Theorem~\ref{the:concentration_resolvent} below), one has to understand the dependence between $Q$ and $x_i$, for all $i\in[n]$. This is performed with the notation, given for any $M = (m_1,\ldots, m_n)\in \mathcal M_{p,n}$ or any $\Delta = \diag_{i\in[n]}(\Delta_i) \in \mathcal D_n$:
\begin{itemize}
  \item $M_{-i} = (m_1,\ldots, m_{i-1}, 0, m_{i+1},\ldots, m_n) \in \mathcal M_{p,n}$,
  \item $\Delta_{-i} = \diag  (\Delta_1,\ldots, \Delta_{i-1}, 0, \Delta_{i+1},\ldots, \Delta_n) \in \mathcal D_n$.
\end{itemize}

The structure of the study of the resolvent is very similar to the one conducted in Section~\ref{sec:XDY} and we will try to draw the maximum of analogy between the two sections. The first theorem should for instance be compared to Proposition~\ref{pro:estimation_XDY} and Proposition~\ref{pro:concentration_lineaire_XDY}.

\begin{theorem}\label{the:concentration_resolvent}
   Given a random diagonal matrix $D \in \mathcal D_n$ and a random matrices $X = (x_1,\ldots, x_n)$, in the regime\footnote{It is not necessary to assume that $p\leq O(n)$ but it simplifies the concentration result (if $p \gg n$, the concentration is not as good, but it can still be expressed).} $p\leq O(n)$ and under the assumptions:
   \begin{itemize}
     \item $(X,D) \propto \mathcal E_2$,
     \item all the couples $(x_i,y_i)$ are independent,
     \item $O(1) \leq \sup_{i\in[n]}\|\mathbb E[x_i] \|, \sup_{i\in[n]}\|\mathbb E[y_i] \|\leq O(1)$,
     \item for all $i \in[n]$, there exists a random diagonal matrix $D^{(i)}$, independent of $x_i$, such that $ \sup_{i\in[n]}\|D_{-i} - D_{-i}^{(i)}\|_F \leq O(1)$,
     \item there exist\footnote{The assumptions $\|X\|/\sqrt n$ and $\|Y\|/\sqrt n$ bounded and $\kappa^2 \kappa_D\leq 1-\varepsilon$ might look a bit strong (since it is not true for matrices with i.i.d. Gaussian entries) and it is indeed enough to assume that $\mathbb E[\|X\|] \leq O(\sqrt n)$ introduce a parameter $z>0$ and study the behavior of $(zI_p - \frac{1}{n}XDX^T)^{-1}$ when $z$ is far from the spectrum of $\frac{1}{n}XDX^T$ -- as it is done in \cite{louart2021spectral}.
      We however preferred here to make a relatively strong hypothesis not to have supplementary notations and proof precautions, that might have blurred the message.} three constants $\kappa_X,\kappa_D, \varepsilon>0$ ($\varepsilon \geq O(1)$ and $\kappa_X,\kappa_D\leq O(1)$), such that $\|X\| \leq \sqrt n\kappa_X$, $\|D\| \leq \kappa_D$ and $\kappa_X^2 \kappa_D\leq 1-\varepsilon$,
   \end{itemize}
   the resolvent $Q \equiv (I_p - \frac{1}{n}XDX^T)^{-1}$ follows the linear concentration
   \begin{align*}
    Q \in\mathcal E_2 \left( \frac{\log^{\frac{3}{2}}n}{\sqrt n}\right) \pm \mathcal E_{\frac{1}{2}} \left(\frac{\log^{\frac{3}{2}} n}{n}\right)&
    &\text{in} \ \ \left( \mathcal{M}_{p}, \|\cdot\| \right).
  \end{align*}

 \end{theorem} 

 Inspiring from the identities provided in \cite{SIL95, PAJ09,louart2021spectral}, one can further estimate this random matrix thanks to the following notation, given $\delta, D\in \mathcal M_n $:
 \begin{align*}
   \tilde Q^\delta(D) \equiv \left( I_p - \frac{1}{n} \sum_{i=1}^n \mathbb E \left[ \frac{D_i}{1 + \delta D_i} \right] \Sigma_i \right)^{-1},
 \end{align*}
 where we noted for all $i \in [n]$, $\Sigma_i \equiv \mathbb E[x_iy_i^T]$.
 \begin{theorem}\label{the:Estimation_resolvante}
  Given a random diagonal matrix $D \in \mathcal{D}_{n}$, the fixed point equation:
  \begin{align*}
    \delta = \diag_{i\in [n]} \left( \Sigma_i \tilde Q^\delta(D) \right)
  \end{align*}
  admits a unique solution $\delta \in \mathcal D_n$ and, under the hypotheses of Theorem~\ref{the:concentration_resolvent}, 
  one can estimate:
  \begin{align*}
    \|\mathbb E[Q ]-\tilde Q^\delta(D)\|_F \leq O(\log^{\frac{3}{2}} n).
  \end{align*}
 \end{theorem}

 \begin{remark}\label{rem:D_m_D_(i)}
   Let us give two examples of the matrices $D^{(i)}$ that one could choose, depending on the cases:
   \begin{itemize}
      \item For all $ i \in[n]$, $D_i =f(x_{i})$ for $f: \mathbb R^{2p} \to \mathbb R$, bounded, then, $D_i$ just depends on $(x_i,y_i)$ so one can merely take $D^{(i)} = D_{-i}$ for all $i \in[n]$.
      \item For the robust regression described by Equation~\ref{eq:point_fixe_robust}, as in \cite{SED21}, we can assume for simplicity\footnote{The bound $\|f\|_\infty \leq O(1)$ is not necessary to set the concentration of $Q$, but it avoids a lot of complications.} $\|f\|_\infty,\|f'\|_\infty,\|f'{}'\|_\infty \leq O(1)$. If we choose $D = \diag(f'(x_i^T\beta))$, then it is convenient to assume $\frac{1}{n}\|f'\|_\infty \|X\|^2 \leq 1-\varepsilon $ (which implies in particular $\frac{1}{n}\|X\|\|D\|\|Y\| \leq 1-\varepsilon$) so that $\beta$ is well defined, being solution of a contractive fixed point equation. One can further introduce $\beta^{(i)} \in \mathbb R^p$, the unique solution to
      \begin{align*}
        \beta^{(i)} = \frac{1}{n}\sum_{\genfrac{}{}{0pt}{2}{1\leq j \leq n}{j \neq i}}f(x_j^T \beta^{(i)}) x_j.
      \end{align*}
      By construction, $\beta^{(i)}$ is independent of $x_i$ and so is:
      \begin{align*}
        D^{(i)}\equiv \diag \left(f'(x_1^T\beta^{(i)}), \ldots, f'(x_{i-1}^T\beta^{(i)}), 0, f'(x_{i+1}^T\beta^{(i)}), \ldots, f'(x_n^T\beta^{(i)}) \right).
      \end{align*}
      Besides $\|D_{-i}-D^{(i)}_{-i} \|_F \leq \|f'{}'\|_\infty \|X_{-i}^T(\beta - \beta^{(i)}) \|_F$. Now, the identities:
      \begin{align*}
        X_{-i}^T\beta = \frac{1}{n}X_{-i}^TXf(X^T\beta)&
        &\text{and}&
        &X_{-i}^T\beta^{(i)} = \frac{1}{n}X_{-i}^TX_{-i}f(X_{-i}^T\beta^{(i)})
      \end{align*}
      (where $f$ is applied entry-wise) imply:
      \begin{align*}
        \|X_{-i}^T(\beta - \beta^{(i)}) \|_F \leq \frac{1}{n} \|f'\|_\infty \|X_{-i}\|^2 \|X_{-i}^T(\beta - \beta^{(i)}) \|_F + \frac{1}{n}f(x_i^T\beta)X_{-i}^Tx_i.
      \end{align*}
      We can then deduce (since $\frac{1}{n} \|f'\|_\infty \|X_{-i}\|^2 \leq 1-\varepsilon$ by hypothesis):
      \begin{align*}
         \|D_{-i}-D^{(i)}_{-i} \|_F \leq \|f'{}'\|_\infty\|X_{-i}^T(\beta - \beta^{(i)}) \|_F \leq \frac{ \|f'{}'\|_\infty}{n\varepsilon}f(x_i^T\beta)X_{-i}^Tx_i \leq O(1).
      \end{align*}

    \end{itemize} 
 \end{remark}

 The preliminary lemmas to the proof of Theorem~\ref{the:concentration_resolvent} are here to prove a similar result to Corollary~\ref{cor:borne_norme_d_XAY}, namely the concentration of $\|X^TQAQX\|_d$, given in Lemma~\ref{lem:concentration_norme_d_XQAQY}. 
 The proof of Theorem~\ref{the:Estimation_resolvante} requires a finer approach presented in Appendix~\ref{app:estimation_resolvent}.

\section*{Conclusion}
With the complexity of nowadays machine learning algorithms, it becomes crucial to devise simple and efficient notations to comprehend their structural logic. For that purpose, the present work provides a systematic approach to comprehend the probabilistic issues involving concentrated vectors, as a model for real data, and their use in statistical learning methods. Indeed, on the one hand, as justified in \cite{SED19}, the very realistic artificial images created by generative adversarial networks are concentrated random vectors by construction: this strongly suggests that most commonly studied databases satisfy our hypotheses. On the other hand, the flexibility of the hypotheses of Theorem~\ref{the:Concentration_produit_de_vecteurs_d_algebre_optimise} and of Theorem~\ref{the:concentration_concentrated_variations} ensures that a wide range of real functionals involved in machine learning problems are concerned by those results. 

As such, in essence, the article provides a catalogue of ready-to-use results for a probabilistic approach of machine learning. To summarize, establishing a concentration inequality on a given random quantity $Y$ generally follows the steps:
\begin{enumerate}
    \item Identify the random vectors $X_1,\ldots, X_m$ (independent or not) upon which $Y$ is built, and verify that $(X_1,\ldots, X_m) \propto \mathcal E_2$;
    \item Bound the variations of $Y$ with a functional $\delta_i$ when $X_i$ varies, $\forall i\in [m]$;
    \item Express the concentration of $\delta_i$, for all $i\in[m]$ and deduce the concentration of $Y$ from Theorem~\ref{the:concentration_concentrated_variations}, or from Theorem~\ref{the:Concentration_produit_de_vecteurs_d_algebre_optimise}, depending on $\delta_i$.
\end{enumerate}

\begin{appendix}

Our work strongly relates to general log-concave settings (very similar to the setting we proposed here: in \cite{ADA11} for Wigner matrices and in \cite{PAJ09} for Wishart matrices) for which the asymptotic behavior of the spectral distribution of random matrices was shown only to depend on the first moments of the entries. In these probabilistic contexts, the random objects behave as if the initial data were Gaussian because the only relevant statistics of the asymptotic behavior is composed of the means and covariances of the data. The laborious Gaussian calculus (with the Stein method as in \cite{PAS05}, possibly combined with Poincaré inequalities as in \cite{CHA07}) then appears as superfluous and can be replaced by concentration of measure arguments in more general settings. This being said, a result from \cite{KLA07, FLE07} establishing a central limit theorem (CLT) for deterministic projections of concentrated random vectors allowed us in a parallel contribution to employ Gaussian inference for the estimation of quantities depending on such projections \cite{SED21}. Nonetheless, in this case, a small number of projections do not satisfy the CLT\footnote{Given a random variable $z_1 \sim \text{Unif}([0,1])$ and $p-1$ i.i.d. random variables $z_2,\ldots, z_p \sim \mathcal N(0,1)$, we know that $Z = (z_1,\ldots, z_p) \sim \mathcal E_2$ but $e_1^T Z \sim \text{Unif}([0,1])$ is not Gaussian. It is stated in \cite{KLA07} that for most $u \in \mathbb S^{p-1}$, $u^TZ$ is quasi-Gaussian (the measure of the complementary set to such $u$ is exponentially decreasing with the maximal distance in infinity norm between the Gaussian CDF and the CDFs of $u^TZ$).}, which restricts the application of the argument.


\section{Proof of the concentration of generalized products}\

 \section{Proofs of resolvent concentration properties}\label{app:concentration_resolvante}
\subsection{Lipschitz concentration of Q}\label{app:concentration_lipschitz_Q}
 \begin{lemma}\label{lem:Q_borne}
   Under the assumptions of Theorem~\ref{the:concentration_resolvent}, $\|Q\| \leq \frac{1}{\varepsilon} \leq O(1)$.
 \end{lemma} 
 Then we can show a Lipschitz concentration of $Q$ but with looser observable diameter that the one given by Theorem~\ref{the:concentration_resolvent} (as for $XDX^T$, we get better concentration speed in the linear concentration framework).
 \begin{lemma}\label{lem:concentration_faible_resolvante}
   Under the hypotheses of Theorem~\ref{the:concentration_resolvent}: 
   \begin{align*}
      \left(Q, \frac{1}{\sqrt n} QX\right) \propto \mathcal E_2 \quad \text{in} \ \ (\mathcal M_{p,n}, \| \cdot \|_F).
    \end{align*} 
 \end{lemma}
   \begin{proof}
   Let us just show the concentration of the resolvent, the concentration of $\frac{1}{\sqrt n} QX$ is treated the same way thanks to the bound $\|X\|\leq O(\sqrt n)$. If we note $\phi(X,D) = Q$ and we introduce $X' \in \mathcal M_{p,n}$ and $D' \in \mathcal D_n$, satisfying $\|X'\|\leq \kappa_X \sqrt n$ and $\|D'\| \leq \kappa_D$ as $X,D$, we can bound:
   \begin{align*}
     \left\Vert \phi(X,D) - \phi(X',D)\right\Vert_F
     &\leq \frac{1}{n}\left\Vert \phi(X,D) (X-X') DX^T \phi(X',D)\right\Vert_F \\
     &\hspace{1cm}+ \frac{1}{n}\left\Vert \phi(X,D) X' D(X-X')^T \phi(X',D)\right\Vert_F \ \ 
     \leq \frac{2\kappa_X\kappa_D}{\varepsilon^2\sqrt n} \left\Vert X- X'\right\Vert_F,
   \end{align*}
   thanks to the hypotheses and Lemma~\ref{lem:Q_borne} given above. The same way, we can bound:
   \begin{align*}
     \left\Vert \phi(X,D) - \phi(X,D')\right\Vert_F\leq \frac{\kappa_X^2}{\varepsilon^2} \left\Vert D- D'\right\Vert_F.
   \end{align*}
   Therefore, as a $O(1)$-Lipschitz transformation of $(X,D)$, $Q \propto \mathcal E_2$.
   \end{proof}

 \subsection{Control on the dependency on $(x_i,y_i)$}\label{app:borne_Qx_i_Qy_i}
   
 The dependence between $Q$ and $x_i$ prevent us from bounding straightforwardly $\|Qx_i\|$ with Lemma~\ref{lem:Q_borne} and the hypotheses on $x_i$. We can still disentangle this dependence thanks to some notations and classical random matrix identities. Let us denote:
 \begin{align*}
   Q_{-i} = \left(I_p - \frac{1}{n}X_{-i}^TDX_{-i}^T\right)^{-1}&
   &\text{and}&
   &Q^{(i)}_{-i} = \left(I_p - \frac{1}{n}X_{-i}^TD^{(i)}X_{-i}^T\right)^{-1}.
 \end{align*}
 We can indeed bound:
 \begin{align}\label{eq:borne_Q_m_i_(i)_x_i}
   \|\mathbb E[Q_{-i}^{(i)}x_i]\| \leq \|\mathbb E[Q_{-i}^{(i)}]\mathbb E[x_i]\| \leq O(1),
 \end{align}
 and we even have interesting concentration properties that will be important later: 
 \begin{lemma}\label{lem:concentration_Q_m_i_x_i}
   Under the assumptions of Theorem~\ref{the:concentration_resolvent}:
   \begin{align*}
  &Q^{(i)}_{-i}x_i,
      \frac{1}{\sqrt n}X_{-i}^TQ^{(i)}_{-i}x_i\in O(1) \pm \mathcal E_2&
   &\text{and}&
   &\frac{1}{n}x_{i}^TQ^{(i)}_{-i}x_i\in O(1) \pm \mathcal E_2 \left(\frac{1}{\sqrt{n}}\right) + \mathcal E_1 \left(\frac{1}{n}\right).
    \end{align*} 
    Besides, for any deterministic $A \in \mathcal{M}_{p}$:
    \begin{align*}
      x_i^TQ_{-i}^{(i)}AQ_{-i}^{(i)}x_i \in O(\|A\|_*) + \mathcal E_2(\|A\|_F) + \mathcal E_1(\|A\|).
    \end{align*}
 \end{lemma}
 \begin{proof}
   Considering $u\in \mathbb R^p$, deterministic such that $\|u\|\leq 1$, we can bound thanks to the independence between $Q^{(i)}_{-i}$ and $x_i$:
   \begin{align*}
     \left\vert u^TQ^{(i)}_{-i}x_i - \mathbb E[u^TQ^{(i)}_{-i}x_i]\right\vert \leq \left\vert u^TQ^{(i)}_{-i} \left(x_i - \mathbb E[x_i]\right)\right\vert + \left\vert u^T \left(Q^{(i)}_{-i}- \mathbb E[Q^{(i)}_{-i}]\right) \mathbb E[x_i] \right\vert .
   \end{align*}
   Therefore, the concentrations $x_i\propto \mathcal E_2$ and $Q^{(i)}_{-i}\propto \mathcal E_2$ given in Lemma~\ref{lem:concentration_faible_resolvante} imply that there exist two constants $C,c>0$ such that $\forall t>0$ such that if we note $\mathcal A_{-i}$, the sigma algebra generated by $X_{-i}$ and $X_{-i}$ (it is independent with $x_i$):
   \begin{align*}
     &\mathbb P \left(\left\vert u^TQ^{(i)}_{-i}x_i - \mathbb E[u^TQ^{(i)}_{-i}x_i]\right\vert \geq t\right) \\
     &\hspace{0.5cm}\leq \mathbb E \left[\mathbb P \left( \left\vert u^TQ^{(i)}_{-i} \left(x_i - \mathbb E[x_i]\right)\right\vert \geq \frac{t}{2} \ | \ \mathcal A_{-i}\right) \right]
     + \mathbb P \left(\left\vert u^T \left(Q^{(i)}_{-i}- \mathbb E \left[Q^{(i)}_{-i}\right]\right) \mathbb E[x_i] \right\vert \geq \frac{t}{2}\right)\\
     &\hspace{0.5cm}\leq \mathbb E \left[Ce^{(t/c\|Q^{(i)}_{-i}\|)^2} \right] + Ce^{(t/c\|\mathbb E[x_i]\|)^2} \leq C' e^{-t^2/c'},
   \end{align*}
   for some constants $C',c' >0$, thanks to the bounds $\|\mathbb E[x_i]\|\leq O(1)$ given in the assumptions and $\|Q^{(i)}_{-i}\| \leq O(1)$ given by Lemma~\ref{lem:Q_borne}.

  The linear concentration of $X_{-i}^TQ_{-i}^{(i)}x_i/\sqrt n$ is proven thanks to the concentration $X_{-i}^TQ_{-i}^{(i)}/\sqrt n \propto \mathcal E_2$ given in Lemma~\ref{lem:concentration_faible_resolvante}. 

The concentration of $\frac{1}{\sqrt n}x_{i}^TQ^{(i)}_{-i}x_i$is proven similarly with the bound:
   \begin{align*}
    \left\vert \frac{1}{\sqrt n}x_{i}^TQ^{(i)}_{-i}x_i - \mathbb E \left[ \frac{1}{\sqrt n}x_{i}^TQ^{(i)}_{-i}x_i \right] \right\vert
    &\leq \left\vert \frac{1}{\sqrt n}x_{i}^TQ^{(i)}_{-i}x_i - \frac{1}{\sqrt n}\tr  \left( \Sigma_i Q^{(i)}_{-i} \right) \right\vert \\
    &\hspace{1cm}+ \left\vert \frac{1}{\sqrt n}\tr  \left( \Sigma_i Q^{(i)}_{-i} \right) - \mathbb E \left[ \frac{1}{\sqrt n}\tr  \left( \Sigma_i Q^{(i)}_{-i} \right) \right] \right\vert,
   \end{align*}
   and the concentrations $Q^{(i)}_{-i} \propto \mathcal E_2$ in $(\mathcal{M}_{p}, \|\cdot\|_F)$ and $ x_i^T Ax_i \in  \tr(\Sigma_i A) \pm\mathcal E_2(\sqrt n\|A\|) \pm \mathcal E_1(\|A\|)$ for any deterministic $A\in \mathcal{M}_{p} $.

  Finally, the concentration of $x_i^TQ_{-i}^{(i)}AQ_{-i}^{(i)}x_i$ is also shown the same way, knowing that $M \mapsto \tr(\Sigma_iMAM)$ is $O(\|A\|_F)$-Lipschitz on $\{\|M\|\leq \frac{1}{\varepsilon}\}$ and for any $B \in \mathcal{M}_{p}$, $x_i^TBx_i \in O(\|B\|_*) \pm \mathcal E_2(\|B\|_F) + E_1(\|B\|)$ thanks to Theorem~\ref{pro:concentration_lineaire_YAX} and the fact that $\mathbb E[x_i^TBx_i] = \tr(\Sigma_iB) \leq \|\Sigma_i\|\|B\|_* \leq \|B\|_*$ thanks to the hypotheses on $X$.

 \end{proof}
 The link between $Qx_i$ and $Q_{-i}x_i$ is made possible thanks to classical Schur identities:
\begin{align}\label{eq:lien_q_qj_schur}
  &Q=Q_{-i} -\frac{1}{n}\frac{D_iQ_{-i}x_ix_i^TQ_{-i}}{1 + D_i \Delta_i}&
  &\text{and}&
  &Qx_i=\frac{Q_{-i}x_i}{1+ D_i \Delta_i},
\end{align}
where we noted $\Delta_i \equiv \frac1nx_i^TQ_{-i}x_i$. 
 The link between $Q_{-i}x_i$ and $Q_{-i}^{(i)}x_i$ is made thanks to:
 \begin{lemma}\label{lem:norme_Q_m_i_Q_i_x_i} 
   Under the hypotheses of Theorem~\ref{the:concentration_resolvent}, for all $i\in[n]$: 
   \begin{align*}
     \|Q_{-i}x_i - Q^{(i)}_{-i}x_i\|, \|Q_{-i}x_i - Q^{(i)}_{-i}x_i\| \in O \left( \sqrt{\log n} \right) \pm \mathcal E_2 \left( \sqrt{\log n} \right).
   \end{align*}
 \end{lemma}
 Let us first prove a Lemma of independent interest:
 \begin{lemma}\label{lem:concentration_norm_infty_YQx}
 Under the hypotheses of Theorem~\ref{the:concentration_resolvent}, for all $i\in[n]$
 \begin{align*}
   \left\Vert Q_{-i}^{(i)} x_i \right\Vert_\infty \in O(\sqrt {\log p}) \pm \mathcal E_2(\sqrt {\log p})&
   &\text{and}&
   &\left\Vert \frac{1}{\sqrt n} X_{-i}^T Q_{-i}^{(i)} x_i \right\Vert_\infty
   \in O(\sqrt {\log n}) \pm \mathcal E_2(\sqrt {\log n}).
 \end{align*}
   
 \end{lemma}
 \begin{proof}
   The control on the variation is given by Lemma~\ref{lem:concentration_Q_m_i_x_i} ($\|\cdot\|_\infty \leq \|\cdot \|_F)$ and the bound on the expectation is a consequence of Proposition~\ref{pro:tao_conc_exp} and the bound:
\begin{align}
  \frac{1}{\sqrt n} \left\Vert \mathbb E \left[ X_{-i}^T Q_{-i}^{(i)} x_i \right]\right\Vert_\infty \leq \frac{1}{\sqrt n} \left\Vert \mathbb E \left[ X_{-i}^T Q_{-i}^{(i)} \right]  \mathbb E \left[  x_i \right]\right\Vert \leq O(1).
\end{align}
One can show the same way that $\left\Vert \mathbb E[Q_{-i}^{(i)} x_i] \right\Vert_\infty \leq O(1)$
 \end{proof}
 \begin{proof}[Proof of Lemma~\ref{lem:norme_Q_m_i_Q_i_x_i}]
 Let us bound directly:
   \begin{align*}
    \left\Vert \left(Q_{-i} - Q_{-i}^{(i)}\right) x_i\right\Vert
    &\leq  \left\Vert \frac{1}{n}  Q_{-i} X_{-i} (D_{-i}^{(i)} - D_{-i})  X_{-i}^T Q_{-i}^{(i)} x_i \right\Vert\\
    &\leq \frac{1}{n} \| Q_{-i} X_{-i}\| \|D_{-i}^{(i)} - D_{-i}\|_F \|  X_{-i}^T Q_{-i}^{(i)} x_i \|_\infty 
    \ \leq O \left(\frac{1}{\sqrt n} \|  X_{-i}^T Q_{-i}^{(i)} x_i \|_\infty \right).
\end{align*}
We can then conclude thanks to Lemma~\ref{lem:concentration_norm_infty_YQx}.
 \end{proof}

 We end this subsection with two fast consequences to Lemma~\ref{lem:norme_Q_m_i_Q_i_x_i} that will find some use.
\begin{lemma}\label{lem:COncentration_xQy}
   Under the hypotheses of Theorem~\ref{the:concentration_resolvent}:
   \begin{align*}
     \Delta_i \equiv \frac{1}{n}x_i^TQ_{-i}x_i  \in \bar \Delta_i   \pm \mathcal E_2 \left(\sqrt{\frac{\log n}{n}}\right) + \mathcal E_1 \left(\frac{\sqrt{\log n}}{n}\right),
   \end{align*}
   where we noted $\bar \Delta_i \equiv \frac{1}{n} \mathbb E[x_i^TQ_{-i}x_i]$ (recall that $\forall i\in[n]: \ |\bar \Delta_i| \leq \frac{\kappa^2}{\varepsilon}$.
 \end{lemma}
 \begin{proof}
   We know from Lemma~\ref{lem:concentration_Q_m_i_x_i} that $\frac{1}{n}x_i^TQ_{-i} ^{(i)}x_i  \in O \left( 1 \right) \pm \mathcal E_2 \left(\frac{1}{\sqrt{n}}\right) + \mathcal E_1 \left(\frac{1}{n}\right)$. One can then conclude with Lemma~\ref{lem:norme_Q_m_i_Q_i_x_i}.
 \end{proof}

 \begin{lemma}\label{lem:concentratin_AQx}
 Under the hypotheses of Theorem~\ref{the:concentration_resolvent} and given $A\in \mathcal{M}_{p}$ such that $\|A\|_F\leq 1$:
  \begin{align*}
       &\|AQ_{-i}x_i\| \in O \left( \sqrt {\log n} \|A\|_F \right) \pm \mathcal E_2 \left( \sqrt {\log n}\|A\| \right)
     \end{align*}   
 \end{lemma}
 \begin{proof}
   The concentration of $\|AQ_{-i}x_i\|$ is a consequence of Lemma~\ref{lem:norme_Q_m_i_Q_i_x_i} and the fact that:
   \begin{align*}
     \mathbb E[\|AQ_{-i}^{(i)}x_i\|]\leq \sqrt{\mathbb E[x_i^T Q_{-i}^{(i)}AAQ_{-i}^{(i)}x_i]} \leq \frac{1}{\varepsilon}\|\mathbb E[x_ix_i^T]\|^{1/2} \|A\|_F \leq O(\|A\|_F).
   \end{align*}
 \end{proof}

\subsection{Proof of Theorem~\ref{the:concentration_resolvent}}\label{app:preuve_concentration_x_iQAQx_i}

Let us first show the concentration of the pseudo-quadratic forms $x_i^TQAQx_i$ for some $A\in \mathcal{M}_{p}$.

 \begin{lemma}\label{lem:COncentration_xQAQyd}
   Under the hypotheses of Theorem~\ref{the:concentration_resolvent}, given a deterministic matrix $A \in \mathcal M_{p,n}$ such that $\|A\|_F\leq 1$:
   \begin{align*}
       x_i^TQ_{-i}AQ_{-i}x_i \in O \left(  \|A\|_*+ \|A\|_F\log n \right) \pm \mathcal E_2 \left(\|A\|_F\log n\right) + \mathcal E_1 \left(\|A\|\log n\right) 
   \end{align*}
 \end{lemma}
 \begin{proof}
 We already know from Lemma~\ref{lem:concentration_Q_m_i_x_i} that:
 \begin{align*}
   x_i^TQ_{-i}^{(i)}AQ_{-i}^{(i)}x_i \in O(\|A\|_*) + \mathcal E_2(\|A\|_F) + \mathcal E_1(\|A\|)
 \end{align*}
 we then use the identity:
 \begin{align*}
   x_i^TQ_{-i}AQ_{-i}x_i = x_i^TQ_{-i}^{(i)}AQ_{-i}^{(i)}x_i + 2x_i^T(Q_{-i} -Q_{-i}^{(i)})AQ_{-i}^{(i)}x_i + x_i^T(Q_{-i} -Q_{-i}^{(i)})A(Q_{-i}- Q_{-i}^{(i)})x_i, 
 \end{align*}
 the bounds:
 \begin{align*}
   \left\{ \begin{aligned}
     &\left\vert 2x_i^T(Q_{-i} -Q_{-i}^{(i)})AQ_{-i}^{(i)}x_i \right\vert \leq \|x_i^T(Q_{-i} -Q_{-i}^{(i)})\| \|AQ_{-i}^{(i)}x_i\|\\
     &\left\vert x_i^T(Q_{-i} -Q_{-i}^{(i)})A(Q_{-i}- Q_{-i}^{(i)})x_i \right\vert \leq \|x_i^T(Q_{-i} -Q_{-i}^{(i)})\|^2 \|A\|,
   \end{aligned}\right.
 \end{align*}
 and the concentrations:
 \begin{itemize}
     \item $ \|AQ_{-i}x_i\| \in O(\sqrt {\log n} \|A\|_F) \pm \mathcal E_2(\sqrt {\log n}\|A\|)$ thanks to Lemmas~\ref{lem:concentratin_AQx},
     \item $\left\Vert  \frac{1}{\sqrt n}x_i^TQ_{-i}X_{-i} \right\Vert_\infty \in O(\sqrt {\log n}) \pm \mathcal E_2(\sqrt {\log n})$ thanks to Lemma~\ref{lem:concentration_norm_infty_YQx},
  \end{itemize}
  to show the concentration of $x_i^TQ_{-i}AQ_{-i}x_i$ with Theorem~\ref{the:concentration_concentrated_variations}.
\end{proof}

\begin{lemma}\label{lem:concentration_norme_d_XQAQY}
   Under the hypotheses of Theorem~\ref{the:concentration_resolvent}, given a deterministic matrix $A \in \mathcal M_{p,n}$, $B\in \mathcal{M}_{p,n}$ such that $\|A\|_*\leq 1$, $\|B\|_*\leq 1$:
   \begin{align*}
     X^TQAQX  \propto \mathcal E_2 \left( ( \|A\|_* \log^{\frac{3}{2}}n\right) \pm \mathcal E_1 \left(\|A\|_F\log^{\frac{3}{2}} n\right) + \mathcal E_{\frac{2}{3}} \left(\|A\|\log^{\frac{3}{2}} n\right)&
     &\text{in} \ \ (\mathcal{M}_{n},  \|\cdot\|_d).
   \end{align*}
   and $\mathbb E[\|X^TQAQX\|_d] \leq O \left( \sqrt n  \log^{\frac{3}{2}}n \|A\|_*\right)$.
\end{lemma}

To control the variation of the upper quantities, one first needs the following lemma
\begin{lemma}\label{lem:control_UQY}
  Under the hypotheses of Theorem~\ref{the:concentration_resolvent},
  for any deterministic matrix $U\in \mathcal{M}_{p}$ such that $\|U\|\leq 1$:
  \begin{align*}
    \left\{\begin{aligned}
      &\|U^TQX \|_\infty\in O(\sqrt{\log n}) \pm \mathcal E_2(\sqrt{\log n})\\
    &\left\Vert X^TQ'{}'_{-i}AQ_{-i}X \right\Vert_{\infty} \in O \left( ( \|A\|_*\log^{\frac{3}{2}} n \right) \pm \mathcal E_2 \left(\|A\|_F\log^{\frac{3}{2}} n\right) + \mathcal E_1 \left(\|A\|\log^{\frac{3}{2}} n\right)
    \end{aligned}\right.
  \end{align*}
\end{lemma}
\begin{proof}
  If $p\leq n$, one can replace $U$ with a matrix $U' \in \mathcal{M}_{p, n}$ that satisfies $\|UQX\|_\infty = \|U'QX\|_\infty$, we thus assume from now on that $U\in \mathcal{M}_{p, n}$. Noting the columns of $U$: $u_1,\ldots, u_n$, we can bound:
  \begin{align*}
    \|U^TQxi \|_\infty 
    &= \sup_{i,j} u_i^TQX \leq  \sup_{i,j\in[p]} \frac{u_j^TQ_{-i}x_i}{1+D_i\delta_i}\\
    &\leq \sup_{i,j\in[p]} |u_j^TQ_{-i}^{(i)}x_i| + \sup_{i\in[n]} \left\Vert Q_{-i}x_i - Q_{-i}^{(i)}x_i \right\Vert
  \end{align*}
  and we deduce the concentration of $\|UQX\|_\infty$ since we know:
  \begin{itemize}
     \item from Lemma~\ref{lem:concentration_Q_m_i_x_i} and Proposition~\ref{pro:tao_conc_exp} that $\sup_{i,j\in[p]} |u_j^TQ_{-i}^{(i)}x_i| \in O(\sqrt{\log (n^2)}) \pm \mathcal E_2(\sqrt{\log (n^2)})$, and of course, $\log n^2 \leq O(\log n)$,
     \item from Lemma~\ref{lem:norme_Q_m_i_Q_i_x_i} that $\left\Vert Q_{-i}x_i - Q_{-i}^{(i)}x_i \right\Vert\in O \left( \sqrt{\log n} \right) \pm \mathcal E_2 \left( \sqrt{\log n} \right)$.
   \end{itemize}  

   For the second result, we recall first that for any $x,y\in \mathbb R^p$ and $A\in \mathcal{M}_{p}$ nonnegative symmetric:
   \begin{align*}
       2\left\vert x^TAy \right\vert \leq \left\vert x^TAx \right\vert + \left\vert y^TAy \right\vert.
    \end{align*}  
    Therefore, we rather bound:
   \begin{align*}
     \sup_{i \in [n]} \left\vert x_i^TQAQx_i \right\vert
     \leq \sup_{i \in [n]} \left\vert x_i^TQ_{-i}AQ_{-i}x_i \right\vert,
   \end{align*}
   (we know that $x_i^TQ = \frac{x_i^TQ_{-i}}{1+D_i\Delta_i}$ and $D_i\Delta_i \geq 0)$. 
   We can then conclude thanks to Proposition~\ref{pro:tao_conc_exp} and Lemma~\ref{lem:COncentration_xQAQyd} that states that $x_i^TQ_{-i}AQ_{-i}x_i \in O \left(  \|A\|_*+ \|A\|_F\log n \right) \pm \mathcal E_2 \left(\|A\|_F\log n\right) + \mathcal E_1 \left(\|A\|\log n\right)$.
\end{proof}
\begin{proof}[Proof of Lemma~\ref{lem:concentration_norme_d_XQAQY}]
  Let us introduce $X',D'$, respectively an independent copy of $X$ and $D$ note $\Phi(X,D) = \|X^TQAQX\|_d$, $Q' = (I_p - X'DX'{}^T)^{-1}$ and $Q'{}' = (I_p - XD'X{}^T)^{-1}$.
  One can first bound:
  \begin{align*}
     \left\vert \Phi(X,D) - \Phi(X',D) \right\vert
     &\leq \|(X- X')^TQAQX\|_d + \frac{1}{n}\|X'{}^TQ(X -X')DX^TQ'AQX\|_d \\
     &\hspace{0.5cm}+ \frac{1}{n} \|X'{}^TQX'D(X -X')^TQ'AQX\|_d +\frac{1}{n} \|X'{}^TQ'AQ(X -X')DX^TQ'X\|_d\\ 
     &\hspace{0.5cm}+  \frac{1}{n}\|X'{}^TQ'AQX'D(X -X')^TQ'X\|_d + \|X'{}^TQ'AQ'(X -X')\|_d.
   \end{align*} 
   Some of the term are treated similarly, we will therefore just bound the first and the second one.
  Inspiring from the proof of Corollary~\ref{cor:borne_norme_d_XAY}, we decompose again $A = U \Lambda V^T$ with $\Lambda= \diag(\lambda) \in \mathcal D_n$ and $U,V \in \mathcal O_{p}$, noting $\check X \equiv V^TQX$, $\check X' \equiv V^TQX'$ and $\check Y \equiv U^TQX$ we have the identity:
  \begin{align*}
    \|(X- X')^TQAQX\|_d 
    &\leq \sup_{\genfrac{}{}{0pt}{2}{D \in \mathcal D_n}{\|D\|_F\leq 1}} \tr(D(\check X - \check X')^T \Lambda \check X)\\ &\leq \sup_{\|d\|\leq 1} d^T ((\check X - \check X') \odot \check Y) \lambda \leq \|A\|_F \|(\check X - \check X') \odot \check Y\|\leq \|A\|_F \| X - X'\|_F \|\check Y \|_\infty  
  \end{align*}
  and we know from Lemma~\ref{lem:control_UQY} that $\|\check Y \|_\infty\in O(\sqrt{\log (pn)}) \pm \mathcal E_2(\sqrt{\log (pn)})$.

   The second terms bounds similarly, noting this time $\check X \equiv \frac{1}{n}V^TQ'XDXQX'$, $\check X' \equiv \frac{1}{n}V^TQ'XDX'QX'$ and $\check Y \equiv U^TQX$. Indeed, one can bound:
   \begin{align*}
     \frac{1}{n}\|X'{}^TQ(X -X')DX^TQ'AQX\|_d
    &\leq \|(\check X - \check X') \odot \check Y\|\leq \| X - X'\|_F \|\check Y \|_\infty
   \end{align*}

   Let us now bound the variations on $D$:
   \begin{align*}
     \left\vert \Phi(X,D) - \Phi(X,D') \right\vert
     &\leq  \frac{1}{n} \left\Vert X{}^TQX(D - D')X^TQ'{}'AQX \right\Vert_d + \frac{1}{n} \left\Vert X^TQ'{}'AQX(D- D')Q'{}'X \right\Vert_d.
   \end{align*}
   The two terms are similar, we therefore just bound:
   \begin{align*}
     \frac{1}{n} \left\Vert X{}^TQX(D - D')X^TQ'{}'AQX \right\Vert_d
     \leq \|D - D'\|_F \sup_{i,j} \left\vert x_i^TQ'{}'AQx_j \right\vert
   \end{align*}
   But we know from Lemma~\ref{lem:control_UQY} that:
      \begin{align*}
     \left\Vert X^TQ'{}'_{-i}AQ_{-i}X \right\Vert_{\infty} \in O \left( \|A\|_*\log ^{\frac{3}{2}}n \right) \pm \mathcal E_2 \left(\|A\|_F\log^{\frac{3}{2}} n\right) + \mathcal E_1 \left(\|A\|\log^{\frac{3}{2}} n\right)
   \end{align*}
   One can then conclude on the concentration with Theorem~\ref{the:concentration_concentrated_variations}.

   For the control on the diagonal norm of the expectation, one can merely bound:
   \begin{align*}
     \mathbb E[\|X^TQAQX\|_d]
     &\leq \sqrt {\mathbb E \left[ \sum_{i=1}^n x_i^TQAQx_i \right]}
     = \sqrt {\mathbb E \left[ \sum_{i=1}^n \frac{x_i^TQ_{-i}AQ_{-i}x_i}{(1+D_i\Delta_i)^2} \right]} \\
     &\leq \sqrt { \sum_{i=1}^n \mathbb E \left[x_i^TQ_{-i}AQ_{-i}x_i \right]} \leq O \left( \sqrt n \log^{\frac{2}{3}}\|A\|_* \right).
   \end{align*}
\end{proof}

\begin{proof}[Proof of Theorem~\ref{the:concentration_resolvent}]
  
  Let us consider $A\in \mathcal{M}_{p,n}$ such that $\|A\|_*\leq 1$ and let us note $\phi(X,D)  = \tr(A Q)$. We abusively work with $X,D$ and independent copies $X',D'$  satisfying $\|X\|,\|X'\| \leq \sqrt n \kappa$ and $\|D\|,\|D'\| \leq \kappa_D$ as if they were deterministic variables, and we note $Q'_X \equiv \phi(X,D)$, $Q'_D \equiv \phi(X,D')$. Let us bound the variations
  \begin{align*}
    \left\vert \phi(X,D) - \phi(X,D)\right\vert = \frac{1}{n} \left\vert \tr \left(AQ (X-X') DX Q'_X\right)\right\vert \leq \frac{\kappa\kappa_D}{\varepsilon^2\sqrt n} \|X-X'\|_F.
  \end{align*}
  We can also bound as in the proof of Proposition~\ref{pro:concentration_lineaire_XDY}:
  \begin{align*}
    \left\vert \phi(X,D) - \phi(X,D')\right\vert \leq \frac{1}{n} \left\Vert X Q'_D AQX\right\Vert_d \left\Vert D-D'\right\Vert_F.
  \end{align*}
  and we know from Lemma~\ref{lem:concentration_norme_d_XQAQY} that:
  \begin{align*}
    \left\Vert Y Q'_D AQX\right\Vert_d \in O \left(\sqrt n \log^{\frac{3}{2}}n\right) \pm \mathcal E_{\frac{2}{3}} \left(\log^{\frac{3}{2}} n\right)
  \end{align*}
  (actually Lemma~\ref{lem:concentration_norme_d_XQAQY} gives the concentration of $\left\Vert Y Q AQX\right\Vert_d$, but the proof remains the same if one replaces one of the $Q$ with $Q_D'$, for a diagonal matrix $D'$, independent with $D$).
  We can then conclude on the concentration of $Q$ applying Theorem~\ref{the:concentration_concentrated_variations}.

\end{proof}
\subsection{Proof of Theorem~\ref{the:Estimation_resolvante}}\label{app:estimation_resolvent}


The existence of the deterministic parameters $\delta \in \mathcal D_n$ such that:
\begin{align*}
   \delta = \frac{1}{n} \diag_{i\in [n]} \tr (\Sigma_i \tilde Q^\delta(D))&
   & \left( \text{recall that } \tilde Q^\delta(D) \equiv  \left( I_p - \frac{1}{n} \sum_{i=1}^n \mathbb E \left[ \frac{D_i}{1+ D_i \delta_i} \right] \Sigma_i \right)^{-1}\right)
 \end{align*} is proved in a similar way as in \cite[Theorem 1]{louart2021spectral} thanks to a semimetric $d_s$ defined for any $\delta, \delta' \in \mathcal D_n(\mathbb R^+)$ as $d_s(\delta, \delta') = \|\frac{|\delta - \delta|'}{\sqrt{\delta\delta'}}\|$.
To prove Theorem~\ref{the:Estimation_resolvante}, we are going to invoke a result also taken from \cite{louart2021spectral} providing a deterministic equivalent for resolvent of the form $Q = (I_p - \frac{1}{n}XX^T)^{-1}$ where the columns of $X$ are independent but possibly non identically distributed (that concerns in particular the case of matrices $(I_p - \frac{1}{n}X \tilde DX^T)^{-1}$ for deterministic diagonal matrices $\tilde D$ as stated below).
\begin{theorem}[\cite{louart2021spectral}, Theorem 4, Corollary 1]\label{the:concentrtion_resolvente_digonale_fixe}
  In the regime $p\leq O(n)$, given two random matrices $X = (x_1,\ldots, x_n), Y = (y_1,\ldots, y_n) \in \mathcal M_{p,n}$ such that $X,Y \propto \mathcal E_2$, $O(1) \leq \sup_{i\in[n]}\|\mathbb E[x_i] \|, \sup_{i\in[n]}\|\mathbb E[y_i] \|\leq O(1)$ and all the couples $(x_i,y_i)$ are independent, for any deterministic diagonal matrix $ \tilde D$ satisfying $\| \tilde D\| \leq O(1)$ and $\|X \tilde D X^T\| \leq 1 - \varepsilon$, the equation $\delta = \frac{1}{n}\tr (\Sigma_i \tilde Q^\delta(\tilde D)$ admits a unique solution $\delta$ and we can bound:
  \begin{align*}
    \left\Vert \left( I_p + \frac{1}{n}X \tilde D  X^T \right)^{-1} - \tilde Q^{\delta}( \tilde D) \right\Vert_F \left( \frac{1}{\sqrt n} \right).
  \end{align*}
\end{theorem}

Let us first explain why the resolvent $\bar Q' \equiv (I_p - \frac{1}{n}X \mathbb E[D] X^T)^{-1}$ is not a relevant a deterministic equivalent of $Q$.
\begin{remark}\label{rem:first_resolvent}
Considering a deterministic vector $u\in \mathbb{R}^{p}$ such that $\|u\|\leq 1$:
\begin{align*}
      \left\vert \mathbb E[u^TQ u] -\mathbb E[u^T\bar Q'u]\right\vert
    &=\frac{1}{n}\left\vert \mathbb E \left[u^TQ X\left( D - \mathbb E[D] \right)X^T\bar Q' u \right]\right\vert\\
    &= \frac{1}{n}\sum_{i=1}^n \left\vert \mathbb E \left[ x_i^T Quu^T \bar Q' xi \left( D_i - \mathbb E[D_i]\right)\right]\right\vert\\
    &= \frac{1}{n}\sum_{i=1}^n \left\vert \mathbb E \left[ \left( x_i^T Quu^T\bar Q' xi - \mathbb E[x_i^T Quu^T\bar Q' xi] \right) D_i\right]\right\vert
\end{align*}
We deduce from the identity: 
\begin{align}\label{eq:identity_xQAQx_xQmiAQmix}
  x_i^T QA\bar Q' xi = \frac{x_i^T Q_{-i}A\bar Q' xi}{ 1+D_i\Delta_i}
\end{align}
and the concentrations $D_i \propto O(1) \pm \mathcal E_2$ given by our hypotheses, $x_i^T Q_{-i}A\bar Q' xi \in O \left(  \log n \right) \pm  \mathcal E_1 \left(\log n\right)$ given by Lemma~\ref{lem:COncentration_xQAQyd} and $\Delta_i \in \mathcal O(1) \pm \mathcal E_1 \left(\sqrt{\log n/n}\right)$ given by Lemma~\ref{lem:COncentration_xQy} that:
\begin{align*}
  \mathbb E \left[ \left( x_i^T Quu^T\bar Q xi - \mathbb E[x_i^T Quu^T\bar Q' xi] \right)^2\right] \leq (\log n)^{2}.
\end{align*}
The H\"older inequality then allows us to conclude ($\mathbb E[D_i^2]\leq O(1)$):
\begin{align*}
    \left\Vert \mathbb E[Q] - \tilde Q^{\delta(\mathbb E[D] )}(\mathbb E[D] ) \right\Vert \leq O \left( \log^{2}n \right).
  \end{align*}
This result has no interest because we already knew that $ \|\mathbb E[Q]\| \leq O(1)$.

What happens is that the concentration of $x_i^TQAQx_i$ for general $A$ is not good and can not be improved from the concentration of $x_iQ_{-i}A\bar Q'x_i$ and the identity~\eqref{eq:identity_xQAQx_xQmiAQmix}. Indeed $D_i$ only has an observable diameter of order $O(1)$ and $x_iQ_{-i}A\bar Q'x_i$ can possibly be of size $O(\log n\sqrt n)$ when one only assumes that $\|A\|_F\leq 1$.
\end{remark}

To get a relevant estimation of $\mathbb E[Q]$, one needs the following lemma.

\begin{lemma}\label{lem:COncentration_xQAQyd2}
   Under the hypotheses of Theorem~\ref{the:concentration_resolvent}, given a deterministic matrix $A \in \mathcal M_{p,n}$ such that $\|A\|_F\leq 1$:
   \begin{align*}
     x_i^TQAQx_i (1+D_i\bar \Delta_i) \propto  \mathcal E_1 \left(\log^{\frac{3}{2}} n\right) + \mathcal E_{\frac{1}{2}} \left(\frac{\log^{\frac{3}{2}} n}{\sqrt n}\right).
   \end{align*}
 \end{lemma}
 be careful that the mean is taken on $\Delta_i$ and not on $D_i$, all the subtlety is here since it does not seem possible to show such a tight concentration for $x_i^TQAQx_i (1+\mathbb E[D_i]\bar \Delta_i)$ as explained above.
 \begin{proof}
 The concentration of:
 \begin{align*}
   x_i^TQAQx_i (1+D_i\Delta_i) = x_i^TQ_{-i}AQx_i= x_i^TQAQ_{-i}x_i\in O \left(\sqrt n\log n\right) \pm \mathcal E_2 \left(\log n\right) + \mathcal E_1 \left(\|A\|\log n\right)
 \end{align*}
 is proven the same way as in the proof of Lemma~\ref{lem:COncentration_xQAQyd}, taking advantage of the quasi independence between $x_i$ and $Q_{-i}$.

   To show the concentration of $x_i^TQAQx_i (1+D_i\bar \Delta_i)$, note that:
  \begin{align*}
    \left\vert x_i^TQAQx_i (1+D_i\bar \Delta_i) - x_i^TQAQx_i (1+D_i \Delta_i) \right\vert 
    &\leq \kappa_D \left\vert x_i^TQAQx_i \right\vert \left\vert \Delta_i - \bar \Delta_i \right\vert,
  \end{align*}
  one can then conclude thanks to the concentration of $x_i^TQAQx_i$ and $\Delta_i \in \bar \Delta_i \pm \mathcal E_1(\sqrt {\log n/n})$ given in Lemma~\ref{lem:COncentration_xQy}.
 \end{proof}
We have now all the elements to estimate $\mathbb E[Q] = \mathbb E[(I_p - \frac{1}{n}XDX^T)^{-1}]$.
\begin{proof}[Proof of Theorem~\ref{the:Estimation_resolvante}]
Let us introduce the resolvent $\bar Q \equiv (I_p - \frac{1}{n}X \bar D X^T)^{-1}$ where we defined
\begin{align*}
  \bar D \equiv \left( \bar\Delta - \left( \mathbb E \left[ \left( I_p + D\bar\Delta \right)^{-1} \right] \right)^{-1} \right)^{-1}.
\end{align*}
As will be seen later, this elaborated definition is taken for $\bar D$ to satisfy the following relation:
\begin{align*}
  \frac{\bar D}{1 + \bar D \bar \Delta} = \mathbb E \left[ \frac{D}{I_p+\bar \Delta D} \right],
\end{align*}
it implies in particular that $\tilde Q^\delta(D) = \tilde Q^\delta(\bar D)$ for any $\delta \in \mathcal D_n$.
Let us then consider a deterministic matrix $A\in \mathcal M_{p}$, such that 
$\|A\|_F \leq 1$ and bound:
  \begin{align*}
    &\left\vert \mathbb E[\tr(AQ)] -\mathbb E[\tr(A\bar Q)]\right\vert\\
    &\hspace{1cm}= \frac{1}{n}\sum_{i=1}^n \left\vert \mathbb E \left[ x_i^T QA \bar Q xi \bar \Delta_i^{-1}\left(\bar \Delta_i D_i + 1 - (\bar \Delta_i \bar D_i +1)\right)\right]\right\vert\\
    &\hspace{1cm}= \frac{1}{n}\sum_{i=1}^n \left\vert \mathbb E \left[ x_i^T QA \bar Q xi \bar \Delta_i^{-1} \left( \bar \Delta_i D_i + 1 \right)\left( \bar \Delta_i \bar D_i + 1 \right)\left(\frac{1}{\bar \Delta_i \bar D_i +1} - \frac{1}{\bar \Delta_i D_i + 1} \right) \right]\right\vert\\
    &\hspace{1cm}= \frac{1}{n}\sum_{i=1}^n \left\vert \mathbb E \left[ x_i^T QA \bar Q xi \left( \bar \Delta_i D_i + 1 \right)\left( \bar \Delta_i \bar D_i + 1 \right)\left(\frac{D_i}{\bar \Delta_i D_i + 1}  - \mathbb E \left[ \frac{D_i}{\bar \Delta_i D_i +1} \right]\right) \right]\right\vert\\
    &\hspace{1cm}= \frac{1}{n}\sum_{i=1}^n \left\vert \mathbb E \left[ \left( x_i^T QA \bar Q xi \left( \bar \Delta_i D_i + 1 \right) - \mathbb E \left[ x_i^T QA \bar Q xi \left( \bar \Delta_i D_i + 1 \right) \right] \right)\frac{D_i\left( \bar \Delta_i \bar D_i + 1 \right)}{\bar \Delta_i D_i + 1}  \right]\right\vert\\
    &\hspace{1cm}= \kappa_D \left( \frac{\kappa_D\kappa^2}{\varepsilon} + 1 \right)\sup_{i\in [n]} \sqrt{\mathbb E \left[ \left( x_i^T QA \bar Q xi \left( \bar \Delta_i D_i + 1 \right) - \mathbb E \left[ x_i^T QA \bar Q xi \left( \bar \Delta_i D_i + 1 \right) \right] \right)^2  \right]}\\
    &\hspace{1cm}\leq O \left( \log^{\frac{3}{2}} n \right),
  \end{align*}
  thanks to Lemma~\ref{lem:COncentration_xQAQyd2}.

  Theorem~\ref{the:concentrtion_resolvente_digonale_fixe} then allows us to state that given $\delta \in \mathcal D_n$ solution to:
  \begin{align*}
    \delta = \frac{1}{n} \tr \left( \Sigma_i \tilde Q^\delta(\bar D) \right)&
    & \left( = \frac{1}{n} \tr \left( \Sigma_i \tilde Q^\delta(D) \right) \right),
  \end{align*}
  one has the estimation:
  \begin{align*}
    \left\Vert \mathbb E[ Q] - \tilde Q^\delta(D) \right\Vert_F
    \leq \left\Vert \mathbb E[ Q] -  \mathbb E[ \bar Q] \right\Vert_F+\left\Vert \mathbb E[\bar Q] - \tilde Q^\delta(D) \right\Vert_F
    \leq O \left( \log^{\frac{3}{2}} n + \frac{1}{\sqrt n}   \right)
    \leq O \left( \log^{\frac{3}{2}} n\right).
  \end{align*}

  
\end{proof}

\end{appendix}

\bibliographystyle{alpha}
\bibliography{C:/Users/cosme/Documents/Travail/These/Biblio/biblio}

\end{document}